\documentclass[11pt]{article}
\usepackage{subfiles}
\usepackage{calligra}
\usepackage{mathrsfs}
\usepackage{algorithm}
\usepackage{algpseudocode}
\usepackage{amsmath}
\usepackage{amssymb}
\usepackage{bm}
\usepackage{booktabs}
\usepackage{color}
\usepackage{paralist}
\usepackage{epsfig}
\usepackage{epstopdf}
\usepackage{flafter}
\usepackage{float}
\usepackage{ifthen}
\usepackage{multirow}
\usepackage{amsfonts}
\usepackage{makecell} 
\usepackage{pifont}
\usepackage{subfigure}
\usepackage{times}
\usepackage{threeparttable}
\usepackage{url} 
\usepackage{subfigure,graphicx}
\usepackage{float}
\usepackage{color,multirow}
\usepackage{makecell}
\usepackage{setspace}
\usepackage{cite} 
\usepackage{url} 
\usepackage{ifthen} 
\usepackage[T1]{fontenc}
\usepackage{graphicx}
\usepackage{changepage}
\usepackage{enumitem}

\usepackage{amsmath,amsfonts,amssymb, amsthm, euscript,makeidx,color,mathrsfs}
\usepackage{esint}
\usepackage[numbers,sort&compress]{natbib}
\usepackage[colorlinks,linkcolor=red,anchorcolor=gray,citecolor=green,urlcolor=blue]{hyperref}
\usepackage{graphics}
\usepackage{booktabs}

\usepackage{amsfonts}
\usepackage{color}
\usepackage{amssymb}
\usepackage{graphics}
\usepackage{indentfirst}
\newcommand{\norm}[1]{\left\Vert#1\right\Vert}

\def\sqr#1#2{{\vcenter{\vbox{\hrule height.#2pt
\hbox{\vrule width.#2pt height#1pt \kern#1pt \vrule width.#2pt}
\hrule height.#2pt}}}}

\def\dbR{{\mathop{\rm l\negthinspace R}}}

\def\3n{\negthinspace \negthinspace \negthinspace }
\def\2n{\negthinspace \negthinspace }
\def\1n{\negthinspace }

\def\ds{\displaystyle}

\def\dbN{{\mathop{\rm l\negthinspace N}}}

\def\dbR{{\mathop{\rm l\negthinspace R}}}
\def\={\buildrel \triangle \over =}

\def\g{\gamma}

\def\si{\sigma}

\def\f{\varphi}

\def\ns{\noalign{\ss} }

\def\deq{\mathop{\buildrel\D\over=}}

\def\D{\Delta}

\def\cF{{\cal F}}

\def\ss{\smallskip}
\def\ms{\medskip}

\def\qq{\qquad}

\def\bfA{\mathbf{A}}

\def\bfH{\mathbf{H}}

\def\bfX{\mathbf{X}}  
\def\bfY{\mathbf{Y}}
\def\bfZ{\mathbf{Z}}

\def\obfX{\overline{\bfX}}

\def\bfh{\mathbf{h}}
\def\bfi{\mathbf{i}}

\def\bfn{\mathbf{n}}

\def\bfu{\mathbf{u}}

\def\bfx{\mathbf{x}}

\def\lan{\big\langle}
\def\ran{\big\rangle}

\def\max{\mathop{\rm max}}
\def\min{\mathop{\rm min}}
\def\sup{\mathop{\rm sup}}
\def\inf{\mathop{\rm inf}}

\def\pa{\partial}

\def\wt{\widetilde}
\def\cd{\cdot}
\def\cds{\cdots}

\def\dist{\hbox{\rm dist$\,$}}

\def\bel{\begin{equation}\label}
\def\ee{\end{equation}}
\def\bt{\begin{theorem}}
\def\bcd{\begin{condition}}
\def\ecd{\end{condition}}
\def\et{\end{theorem}}
\def\bc{\begin{corollary}}
\def\ec{\end{corollary}}
\def\bde{\begin{definition}}
\def\ede{\end{definition}}
\def\bl{\begin{lemma}}
\def\el{\end{lemma}}
\def\bp{\begin{proposition}}
\def\ep{\end{proposition}}
\def\br{\begin{remark}}
\def\er{\end{remark}}
\def\ba{\begin{array}}
\def\ea{\end{array}}
\def\ed{\end{document}}
\def\ns{\noalign{\ms}}
\def\ds{\displaystyle}

\def\square#1{\vbox{\hrule\hbox{\vrule height#1%
\kern#1\vrule}\hrule}}
\def\rectangle#1#2{\vbox{\hrule\hbox{\vrule height#1%
\kern#2\vrule}\hrule}}

\font\tenbb=msbm10 \font\sevenbb=msbm7 \font\fivebb=msbm5

\newfam\bbfam
\scriptscriptfont\bbfam=\fivebb \textfont\bbfam=\tenbb
\scriptfont\bbfam=\sevenbb

\newtheorem{lemma}{Lemma}[section]
\newtheorem{remark}{Remark}[section]
\newtheorem{example}{Example}[section]
\newtheorem{theorem}{Theorem}[section]
\newtheorem{corollary}{Corollary}[section]

\newtheorem{definition}{Definition}[section]
\newtheorem{proposition}{Proposition}[section]
\newtheorem{condition}{Condition}[section]
\newtheorem{assumption}{Assumption}[section]

\newcommand{\nmycite}[2]{[\citealp{#1}, #2]}
\newcommand{\medcup}{\mathbin{\scalebox{1.5}{$\cup$}}}

\makeatletter

\@addtoreset{equation}{section}
\makeatother

\allowdisplaybreaks

\begin{document}

\title{\bf  Dynamic Programming Principle  for Stochastic Control Problems on Riemannian Manifolds \footnote{This work is supported by the NSF of China under grant 12025105.}}

\author{
Dingqian Gao
\footnote{School of Mathematics, Sichuan University, Chengdu, P. R. China. Email: 
Email: dingqiangao@outlook.com.}
~~~ and ~~~
Qi L\"{u}
\footnote{School of Mathematics, Sichuan University, Chengdu, P. R. China. Email: lu@scu.edu.cn. }
}

\date{}

\maketitle

\begin{abstract}
In this paper, we first establish the dynamic programming principle for stochastic optimal control problems defined on  compact Riemannian manifolds without boundary. 
Subsequently, we derive the associated Hamilton-Jacobi-Bellman (HJB) equation for the value function.  
We then prove the existence, uniqueness of viscosity solutions to the HJB equation, along with their continuous dependence on initial data and model parameters. Finally, under appropriate regularity conditions on the value function, we establish a verification theorem that characterizes optimal controls.
\end{abstract}



\section{Introduction and main results}

We begin by reviewing some fundamental concepts and notations from Riemannian geometry. For a comprehensive treatment of these topics, we refer the reader to \cite{Lee,doCarmo}.

Let $(M,g)$ be an $n$-dimensional compact Riemannian manifold without boundary ($n\in \mathbb{N}$). We denote by $TM$ and $T^*M$ the tangent and cotangent bundles of $M$, respectively. For $x\in M$,  
$T_x M$ and $T_x^*M$ represents the tangent space  and the cotangent space of $M$ at $x$, respectively.  Given a tangent vector $v\in T_x M$, we define its norm as: $$\|v\|_g=\sqrt{\lan v,v\ran _g},$$ where $\lan \cdot,\cdot\ran _g\deq g(\cdot,\cdot)$ denotes the inner product induced by the metric $g$.

For integers $r,s\in \mathbb{N}$, we denote by $\mathcal{T}_s^r(M)$
the $(r,s)$ type tensor bundle over $M$ and by $\mathcal{T}_s^r(x)$ the vector space of all $(r,s)$-type tensors at the point $x\in M$. Furthermore, we let $\mathsf{X}^{k}(M)$ ($k\in\mathbb{N}$) denote the space of $C^k$-smooth vector fields on $M$. 

Let $D$ be the Levi-Civita connection on $M$. Recall that for any tensor field  $\tau\in \mathcal{T}_s^r(M)$, its covariant derivative $D\tau\in \mathcal{T}_{s+1}^r(M)$. In local coordinates $(\mathscr{U}, x^i)$ on $M$, with the natural tangent space basis $\{\tfrac{\partial}{\partial x_i}\}_{i=1}^n$ and cotangent space basis $\{dx_i\}_{i=1}^n$, the covariant derivative acts on $\tau$ as follows:
\begin{align}
&	D\tau\big(dx_{j_1},...,dx_{j_r},\tfrac{\partial}{\partial x_{i_1}},...,\tfrac{\partial}{\partial x_{i_s}},\tfrac{\partial}{\partial x_{k}}\big) \nonumber\\ &:= \tfrac{\partial}{\partial x_{k}}\tau\big(dx_{j_1},...,dx_{j_r},\tfrac{\partial}{\partial x_{i_1}},...,\tfrac{\partial}{\partial x_{i_s}}\big) -\sum_{l=1}^{r}\tau\Big(...,D_{\frac{\partial}{\partial x_{k}}}dx_{j_l},...\Big)-\sum_{l=1}^{s}\tau\Big(...,D_{\frac{\partial}{\partial x_{k}}}\tfrac{\partial}{\partial x_{i_l}},...\Big),\label{fml0}
\end{align} 
Let $\phi\in C^{\infty}(M)$. The differential $D\phi$ acts on vector fields $\bfX\in \mathsf{X}^0(M)$ as $(D\phi)(\bfX)=\bfX\phi$.  The gradient vector field $\nabla \phi$ is  the unique vector field satisfying 

\begin{equation}\label{DDf}
\lan \nabla \phi,\bfX\ran _g=(D\phi)(\bfX),\qq \forall \bfX\in \mathsf{X}^0(M).
\end{equation}
%

Applying the covariant derivative formula \eqref{fml0} to $\tau=D\f\in \mathcal{T}_1^0(M)$, we obtain the Hessian operator: 
\begin{equation}\label{fml}
D^2\phi(\bfX,\bfY):=D(D\phi)(\bfX,\bfY)=\bfY\bfX \phi-D_{\bfX}\bfY \phi,\quad \bfX,\bfY\in \mathsf{X}^1(M).
\end{equation}

The curvature operator $\mathcal{R} \in \mathcal{T}^1_3(M)$ is defined in local coordinates $\{\tfrac{\partial}{\partial x_i}\}_{i=1}^n$ by
$$\mathcal{R}\Big(\tfrac{\partial}{\partial x_i},\tfrac{\partial}{\partial x_j}\Big)\tfrac{\partial}{\partial x_k}:=D_{\tfrac{\partial}{\partial x_i}}D_{\tfrac{\partial}{\partial x_j}}\tfrac{\partial}{\partial x_k}-D_{\tfrac{\partial}{\partial x_j}}D_{\tfrac{\partial}{\partial x_i}}\tfrac{\partial}{\partial x_k}-D_{[\tfrac{\partial}{\partial x_i},\tfrac{\partial}{\partial x_j}]}\tfrac{\partial}{\partial x_k},$$
where  $[\cdot,\cdot]$ denotes the Lie bracket of vector fields. The corresponding curvature tensor $R \in \mathcal{T}^0_4(M)$ is given by
\begin{equation*}
R(\tfrac{\partial}{\partial x_i},\tfrac{\partial}{\partial x_j},\tfrac{\partial}{\partial x_k},\tfrac{\partial}{\partial x_l}) 
:= \big\langle \mathcal{R}(\tfrac{\partial}{\partial x_i},\tfrac{\partial}{\partial x_j})\tfrac{\partial}{\partial x_k},\tfrac{\partial}{\partial x_l} \big\rangle_g.
\end{equation*}

For the smooth, compact manifold $M$, there exists $C_R$ such that the curvature tensor satisfies the following uniform bounds:
\begin{align}\label{eqrr}
\|R\|+\|DR\|+\|D^2R\|\le C_{R},
\end{align}
where $\|\cdot\|$ denote the tensor norm.

Let $\mathcal{L}(T_xM):=\mathcal{T}_1^0(x)$ and denote by $\mathcal{L}_s^2(T_xM)\subset \mathcal{T}_2^0(x)$  the set of symmetric tensors on $T_xM$.
%

Let $\rho: M \times M \to \dbR$ be the Riemannian distance function induced by the metric on $M$. For any $\epsilon >0$, we define the tangent ball of radius $\epsilon$ at $x\in M$ as $B(x,\epsilon):=\big\{v\in T_x M,\|v\|_g< \epsilon\big\}$ and the geodesic ball of radius $\epsilon$  at $x\in M$ as $B_x^{\rho}(\epsilon):=\big\{y\in M,\rho(x,y)< \epsilon\big\}$.

The injectivity radius at $x \in M $ is defined as
$$
\mathbf{i}(x):=\sup\big\{\epsilon>0; \mbox{ the map }\exp_x:B(x,\epsilon)\rightarrow B_x^{\rho}(\epsilon) \mbox{  is a diffeomophism} \big\}.
$$
Since $M$ is compact and $x \mapsto \mathbf{i}(x)$ is continuous, the global injectivity radius 
\begin{equation}\label{6.3-eq1}
\mathbf{i}_M:=\min\big\{\mathbf{i}(x),x\in M\big\}
\end{equation}
exists and satisfies  $\mathbf{i}_M>0$. 

For any $x,y\in M$ with $\rho(x,y)< \mathbf{i}_M$, there exists a unique minimizing geodesic $\gamma$ connecting $x$ and $y$. We denote by $L_{x y}: T_x M \to T_y M$ the parallel transport along $\gamma$, which preserves the Riemannian metric:%
\begin{equation}\label{LXY}
\lan L_{x y}v,L_{x y}v\ran _g=\lan v,v\ran _g.
\end{equation}
%

%
%

For two  matrices $\mathbb{Y} = (y_{ij})_{n \times m}\in\dbR^{n \times m}$ 
and $\mathbb{Z} = (z_{ij})_{n \times m}\in\dbR^{n \times m}$, we write $\mathbb{Y} = O_C(\mathbb{Z})$ if there exists a constant $C > 0$ such that $|y_{ij}| \leq C|z_{ij}|$ for all $i = 1, \dots, n$ and $j = 1, \dots, m$.

Suppose $\mathbb{Y}_1 = O_{C_1}(\mathbb{Z})$ and $\mathbb{Y}_2 = O_{C_2}(\mathbb{Z})$. Then the following properties hold: 
$$\mathbb{Y}_1+\mathbb{Y}_2=O_{C_1+C_2}(\mathbb{Z}),\quad \mathbb{Y}_1\mathbb{Y}_2=O_{C_1C_2}(\mathbb{Z}^2).$$

Let $T>0$, and let $(\Omega, \mathcal{F}, \{\cF_t\}_{t\geq 0},
\mathbb{P})$ be a complete filtered probability space, on which an
$m$-dimensional Brownian motion
$W(\cdot)$ is defined and $\mathbf{F}\deq \{\cF_t\}_{t\geq 0}$ is
the natural filtration generated by $W(\cdot)$. Denote by
$\mathbb{F}$ the progressive $\bm{\sigma}$-algebra with respect to
$\mathbf{F}$.

For any $t\in [0,T]$, let $L_{\mathcal{F}_t}(\Omega,M)$ denote the space of all $\mathcal{F}_t$-measurable random variables taking values in the manifold $M$.
Let $U\subset \dbR^k$ be a compact set for some $k\in \mathbb{N}$. For any $0<t\le T$, the admissible control set  $\mathcal{U}[t,T]$ is the completion of  the set 
$$ \mathcal{U}_{piece}[t,T]:=\Big\{u(\cdot)=\sum_{i=0}^{N_1}u_i1_{[t_i,t_{i+1}]},t=t_0\le  ...\le t_{N_1}=T,u_i \in L_{\mathcal{F}_{t_i}}(\Omega,U),N_1\in \mathbb{N}\Big\}.$$
with respect to the norm of $L_{\mathbb{F}}^{2}(t,T;\dbR^k)$.

Consider the control system governed by the  following  Stratonovich stochastic differential equation:
\begin{equation}\label{system1}
\left\{
\begin{aligned}
&dX(t)=b(t,X(t),u(t))dt+\sum_{i=1}^m\sigma_{i}(t,X(t),u(t))\circ dW^{i}(t),&\quad t\in (0,T],\\
&X(0)=x \in M,
\end{aligned}
\right.
\end{equation}
with the cost functional 
\begin{equation}\label{cf1}
\mathcal{J}(x;u(\cdot))=\mathbb{E}\bigg(\int_0^Tf(s,X(s),u(s))ds+h(X(T))\bigg).
\end{equation}
In this paper, for the simplicity of notations and without loss of generality, we assume that $m\le n$.  Further, we assume 
%
$$
\begin{cases}
b(\cdot,x,\cdot) \colon  [0,T] \times U \to T_x M,  \quad x \in M, \\
\sigma_{i}(\cdot, x, \cdot) \colon  [0,T] \times U \to T_x M, \quad x \in M, \quad i = 1, \dots, m, \\
f \colon  [0,T] \times M \times U \to \dbR,
\end{cases}
$$
satisfy the following assumption:
\begin{assumption}\label{condp}
For all $x \in M$, $t \in [0,T]$, and $u \in U$, there exist positive constants $C_B$ and $C_L$ such that the following hold: 

\begin{itemize} 
\item[1)] (Uniform Boundedness)
The vector fields $b(\cdot,\cdot,\cdot)$, $\si_i(\cdot,\cdot,\cdot)$ ($i=1,\cds,m$), and the function $f(\cdot,\cdot,\cdot)$ are jointly continuous in $(t,x,u)$, satisfying
$$\max_{t\in [0,T],x\in M,u\in U}\bigg\{\|b(t,x,u)\|_g+\sum_{i=1}^m\|\sigma_i(t,x,u)\|_g+|f(t,x,u)|\bigg\}\le C_{B},$$
where $\|\cd\|_g$ denotes the norm in the tangent space induced by the Riemannian metric $g$.

\item[2)] (Smoothness Requirements) For each $u \in U$ and $t \in [0,T]$, $b(t,\cdot,u)\in \mathsf{X}^1(M)$  and $\sigma_i(t,\cdot,u)\in \mathsf{X}^2(M)$  for $i=1,...,m$.

\item[3)] (Control Lipschitz Condition) For any $(t,x) \in [0,T] \times M$ and unit vector $v \in T_xM$ ($\|v\|_g = 1$), 
\begin{align}
&\|b(t,x,u_1)-b(t,x,u_2)\|_g+\sum_{i=1}^m\|\sigma_i(t,x,u_1)-\sigma_i(t,x,u_2)\|_g\nonumber\\
&+\sum_{i=1}^m\norm{D_{v}\sigma_i(t,x,u_1)-D_{v}\sigma_i(t,x,u_2)}_g\le C_L \|u_1-u_2\|_U,\nonumber
\end{align}
where $\|\cd\|_U$ is the norm on the control space.
\item[4)] (Spatial Lipschitz Continuity) The functions $h$ and $f$ satisfy
$$\max\limits_{x,y\in M,t\in [0,T],u\in U}\big(|h(x)-h(x)|+|f(t,x,u)-f(t,y,u)|\big)\le C_L\rho(x,y),$$
with $\rho$ the Riemannian distance function.
\end{itemize}
\end{assumption}

Let us recall the definition of the solution of the control system \eqref{system1}.
\begin{definition}
An $M$-valued $\{\mathcal{F}_t\}_{0 \leq t \leq T}$-adapted continuous stochastic process $X$ is said to be a solution of \eqref{system1} if, for every test function $\varphi \in C^{1,2}([0,T] \times M)$ and all $t \in [0,T]$, the following holds almost surely:

\begin{equation}\label{df61-1}
\f(t,X(t))=\f(0,x)+ \int_0^t\Big(\frac{\partial }{\partial r}+b+\frac{1}{2}\sum_{i=1}^m \sigma_{i}^2\Big)\f(r,X(r))dr+\int_s^t\sum_{i=1}^m\sigma_{i}\varphi dW^i(r).
\end{equation}
\end{definition}

Under Assumption \ref{condp}, according to Proposition \ref{pSExistence} given below, the controlled stochastic system \eqref{system1} admits a unique solution $X(\cd)$. Consequently, the cost functional \eqref{cf1} is well-defined.

Consider the following optimal control problem:

\ss

\textbf{Problem (OP)}. Given any $x\in M$, find a control $\bar{u}(\cdot)\in \mathcal{U}[0,T]$ such that
$$\mathcal{J}(x,\bar{u}(\cdot))=\inf_{u(\cdot)\in \mathcal{U}[0,T]}\mathcal{J}(x,u(\cdot)).$$
\ss  

The optimal control problem (OP) emerges naturally in a wide range of practical control applications across multiple disciplines. These include: 
Robotic attitude control (e.g., spacecraft stabilization on $SO(3)$),  
Phase demodulation in signal processing, 
spherical mechanics and constrained dynamical systems,  
autonomous vehicles \& path planning (e.g., motion planning on $SE(2)$),  
Quantum systems (e.g., state preparation on $\mathbb{C}P^n$),  geometric Heston models in mathematical finance and other physical and engineering scenarios where state variables are intrinsically constrained to a manifold (e.g., \cite{Altafini2012,Bismut1981,Duncan76,Duncan79,Gao2023,Lelievre,Samiei2015,van Handel}).

While stochastic differential equations (SDEs for short) in $\dbR^n$
and ordinary differential equations  (ODEs for short) on manifolds have been extensively studied in control theory (see, e.g., \cite{Agrachev2004,Fleming2006,Jurdjevic1997,Krylov2009,Yong1999} and references therein), optimal control problems for SDEs evolving on manifolds remain largely unexplored. Consequently, Problem (OP) is still far from being well understood.

In \cite{Duncan76}, the author investigates an optimal control problem governed by the stochastic differential equation.

In this paper, we investigate Problem (OP) using the dynamic programming approach. To the best of our knowledge, only two prior works have  addressed this problem\cite{Duncan76,Zhu2014}.  

In \cite{Duncan76}, the author investigates an optimal control problem governed by the stochastic differential equation
\begin{equation}\label{DCE}
dX(t) = b(t,X(t),u(t))dt + dB(t),
\end{equation}
where $b(t,X(t),u(t)) \in T_{X(t)}M$ represents the controlled drift term and $dB(t) \in T_{X(t)}M$ corresponds to a manifold-valued Brownian motion $B$. 
By establishing a manifold analogue of the Girsanov theorem, the author demonstrated that the cost functional can be expressed through the Wiener measure associated with $B(\cd)$, which subsequently leads to the dynamic programming principle (DPP) for this framework.

However, our system \eqref{system1} presents a crucial difference: the diffusion term depends explicitly on the control variable. This control-dependent diffusion structure introduces potential degeneracy issues that were absent in the formulation \eqref{DCE}, requiring substantially different analytical techniques.

In   \cite{Zhu2014}, the author  establishes DPP  for affine-type stochastic control systems on Riemannian manifolds as follows:
\begin{align}\label{deq1}
\left\{
\begin{aligned}
&dX(t) = u_0(t)b(t,X(t))dt + \sum_{i=1}^m u_i(t)\sigma_{i}(t,X(t))\circ dW^{i}(t), & t \in (0,T],\\
&X(0) = x. 
\end{aligned}
\right.
\end{align}
The analysis was conducted under the following key geometric and regularity assumptions on the vector fields:
\begin{itemize}[nosep]\label{it}
\item   Smoothness Condition: The vector fields $b(t,\cdot)$ and $\{\sigma_i(t,\cdot)\}_{i=1}^m$ are $C^\infty$ smooth; 
\item   Lipschitz-type Condition: For any $x,y \in M$ with $\rho(x,y) < \mathbf{i}_M$
(the injectivity radius), there exists $\mu > 0$ such that 
\begin{equation*}
\|L_{xy}b(t,x) - b(t,y)\|_g \leq \mu\rho(x,y),
\end{equation*}
where $L_{xy}$ denotes parallel transport along the minimal geodesic connecting $x$ and $y$.
\item   Parallel Condition: For any $x,y \in M$ with $\rho(x,y) < \mathbf{i}_M$,
%
\begin{equation*}
L_{xy}\sigma_i(t,x) = \sigma_i(t,y) \quad (i=1,\ldots,m).
\end{equation*}
\end{itemize}

In this work, we make the following extensions to Zhu's work\cite{Zhu2014}:

\ss

(i) Relaxation of Geometric Constraints: We extend prior frameworks by relaxing the parallel transport condition for the vector fields $\{\sigma_i\}_{i=1}^m$ in the Parallel Condition. This extension is essential for two reasons.
First, in Euclidean space, the Parallel Condition restricts diffusion matrices to constant values, representing a special class of stochastic control systems. Second, the constraint imposed by the Parallel Condition may be unsatisfied on curved manifolds---for instance, on $S^2$ the holonomy phenomenon (where $L_{y z}L_{x y} \neq L_{x z}$ for parallel transports along geodesic triangles) prevents non-trivial vector fields from satisfying the exact transport condition $L_{x y}\sigma_i(t,x) = \sigma_i(t,y)$ for $i=1,\dots,m$.

\ss

(ii) Generalization of the Control System Framework:  Our framework enables the analysis of control systems described by equation \eqref{system1} featuring completely control-dependent vector fields, thereby substantially broadening the scope of tractable systems compared to the conventional affine-type framework.

\ss


In their work \cite{Zhu2014}, the authors follow the methodology presented in \nmycite{Azagra2008}{Theorem 4.2} to investigate the corresponding Hamilton-Jacobi-Bellman (HJB) equation (cf. \nmycite {Zhu2011}{Theorem 3.1}). To ensure the HJB equation satisfies the required assumptions in \nmycite{Azagra2008}{Theorem 4.2(2)}, they specifically examine affine-type stochastic control systems subject to the parallel transport condition by utilizing an important result \nmycite{Azagra2008}{Proposition 4.9} satisfied for compact Riemannian manifold. In our current study, we deliberately avoid the application of \nmycite{Azagra2008}{Proposition 4.9}.

\ss

Our contributions are threefold:
\begin{enumerate}
\item Dynamic Programming Principle (DPP): We establish the DPP for optimal control of SDEs evolving on a manifold.
\item Hamilton-Jacobi-Bellman (HJB) Equation: We derive the associated HJB equation and rigorously analyze its structure.
\item Viscosity Solutions: We prove the existence, uniqueness, and continuous dependence of the viscosity solution to the HJB equation. 
\end{enumerate}

Compared to optimal control problems for SDEs in $\dbR^n$, the nonlinear structure of the tangent bundle in a manifold prevents the direct application of many standard techniques--particularly a priori estimates--to the control system \eqref{system1} evolving on a Riemannian manifold.

Furthermore, unlike optimal control problems for ODEs on manifolds, the unbounded nature of Brownian motion introduces additional challenges. Specifically, when perturbing the optimal control, we cannot rely solely on local charts to estimate deviations between the optimal and perturbed states. This makes Problem (OP) inherently global in nature, in contrast to the more localized analysis possible for deterministic systems.

To overcome these challenges, we employ techniques from stochastic analysis on manifolds and Riemannian geometry. Our approach proceeds through three key steps samilar to \cite{Zhu2014}:
\begin{enumerate}
\item Manifold Embedding:
Utilizing the Nash embedding theorem, we isometrically embed the Riemannian manifold  $M$ into a higher-dimensional Euclidean space $\dbR^{n_1}$  (where $n_1\in\dbN$ and $n_1\geq n$). This embedding preserves the geometric structure of $M$ while enabling analysis in a more tractable Euclidean framework.

\item System Extension and Analysis:
We then extend the control system \eqref{system1} to the ambient Euclidean space $\dbR^{n_1}$ in a carefully constructed manner (as guaranteed by Lemma \ref{lm62}). This extension allows us to apply established results and techniques for controlled SDEs in Euclidean spaces, from which we derive the necessary a priori estimates.

\item Metric Equivalence and Translation:
For the compact manifold $M$, we establish the equivalence between the Euclidean distance (induced by the embedding) and the intrinsic Riemannian distance. This crucial equivalence enables us to rigorously transfer our Euclidean-based estimates back to the original Riemannian setting.
\end{enumerate}
%

%

To study \textbf{Problem (OP)}, we introduce a family of stochastic optimal control problems. For any initial condition $\xi\in L_{\mathcal{F}_s}(\Omega,M)$, consider the controlled stochastic differential equation: 
\begin{align}\label{eqve0}
\left\{
\begin{aligned}
&dX(s)=b(s,X(s),u(s))dt+\sum_{i=1}^m\sigma_i(s,X(s),u(s))\circ dW^i(s),&\quad s\in (t,T],\\
&X(t)=\xi,
\end{aligned}
\right.
\end{align}
The definition of the solution to \eqref{eqve0} is similar to the one for \eqref{system1}. We present it here for the convenience of readers. An $M$-valued $\{\mathcal{F}_s\}_{t\le s\le T}$ adapted continuous  stochastic process $X$ is called solution of $\eqref{eqve0}$, if the following equation holds  almost surely:
\begin{equation}\label{df61}
\f(s,X(s))=\f(t,\xi)+ \int_t^s\Big(\frac{\partial }{\partial r}+b+\frac{1}{2}\sum_{i=1}^m \sigma_{i}^2\Big)\f(r,X(r))dr+\int_s^t\sum_{i=1}^m\sigma_{i}\varphi dW^i(r).
\end{equation}
for every function $\f\in C^{1,2}([t,T]\times M)$ and $s\in [t,T]$.

Let $t\in [0,T]$, $\xi\in L_{\mathcal{F}_t}(\Omega,M)$ and  $u(\cdot)\in \mathcal{U}[t,T]$.  Then the equation \eqref{eqve0} exists a unique solution $X(\cd;t,\xi,u(\cdot))$  (see in Proposition \ref{pExistence}).

The associated cost functional is given by
$$\mathcal{J}(t,\xi;u(\cdot))=\mathbb{E}_{t}\bigg\{\int_t^Tf(s,X(s),u(s))ds+h(X(T))\bigg\},$$
where $\mathbb{E}_t=\mathbb{E}(\cdot|\mathcal{F}_t)$ denotes conditional expectation, and $X(s)$ is short
for $X(s;t,\xi,u(\cdot))$. 

We formulate the augmented value function of \textbf{Problem (OP)$_t$} as: 
\begin{equation}\label{defpJ-1}
\mathbb{V}(t,\xi):=\mathop{\rm essinf}_{u(\cdot)\in \mathcal{U}[t,T]} \mathcal{J}(t,\xi;u(\cdot)).
\end{equation}

Since $\mathcal{J}(t,\xi;u(\cdot))$  is a random variable, the right-hand side of \eqref{defpJ-1} may not exist. However, by Lemmas \ref{prop33} and \ref{effJ}, the essential infimum is ensured to exist, which implies that $\mathbb{V}$ is well-defined.

The value function $V$ of \textbf{Problem (OP)$_t$} is defined as follows:  
\begin{equation}\label{defpJ}
V(t,x):=\mathbb{V}(t,x),\quad t\in [0,T],x\in M.
\end{equation}

For fixed $t$ and $x$, the value function  $V(t,x)$ would appear to be a random variable due to its definition through the conditional expectation. However, by utilizing the intrinsic structure of the optimal control problem and adapting techniques from \cite{yong_stochastic_2022} and \cite{yu}, we establish the following result: 
\begin{proposition} \label{Detm}
The value function $V$ is a deterministic function.
\end{proposition}

Similar to stochastic control problems in Euclidean space, the value function maintains Lipschitz continuity with respect to spatial variables. More precisely, we have the following quantitative estimate:  
\begin{proposition} \label{pLip}
For any $x,y\in M,t\in [0,T]$, there exists an uniform constant $C\ge 0$ such that
$$|V(t,x)-V(t,y)|\le C\rho(x,y).$$
\end{proposition}

Most significantly, the value function $V$ satisfies the following dynamic programming principle:
\begin{theorem}\label{pDPP}
For any $t\in [0,T)$, $s\in (t,T]$ and $\xi\in L_{\mathcal{F}_t}(\Omega,M)$, we have
\begin{align*}
V(t,\xi)=\inf_{u(\cdot)\in \mathcal{U}[t,s]}\mathbb{E}_t\bigg\{\int_t^{s}f(r,X(r),u(r))dr+ V(s,X(s))\bigg\}.
\end{align*}
\end{theorem}
With the help of Theorem \ref{pDPP}, we can prove the following time H{\"o}lder property of value function:
\begin{proposition} \label{pHold}
For any $t,s\in M,t\in [0,T]$, there exists an uniform constant $C\ge 0$ such that
$$|V(t,x)-V(s,y)|\le C|t-s|^{\frac{1}{2}}.$$
\end{proposition}

Based on the Theorem \ref{pDPP} and Proposition \ref{Detm}, we can derive a deterministic HJB equation as follows.
\begin{theorem} \label{phjbl}
Suppose the value function $V\in C^{1,2}([0,T]\times M)$. Then $V$ satisfies the following  HJB equation:
\begin{equation}\label{phjb}
\begin{cases}
\pa_tV +\bfH(t,x,DV,D^2V)=0, & (t,x)\in [0,T]\times M,\\
V(T,x)=h(x), &  x \in M,
\end{cases}
\end{equation}
where
$$\bfH(t,x,\chi,A)=\inf_{u\in U}\mathbb{H}(t,x,u,\chi,A),$$
and
$$\mathbb{H}(t,x,u,\chi,A)=\frac{1}{2}\sum_{i=1}^m A(\sigma_i,\sigma_i)+\chi\Big(b+\frac{1}{2}\sum_{i=1}^m D_{\sigma_i}\sigma_i\Big)+f.$$
\end{theorem}

The value function enables the construction of optimal feedback control through the following result: 

\begin{corollary} \label{pVe}
Let $(t,x)\in [0,T)\times M$, and suppose the value function $V(\cdot,\cdot)\in C^{1,2}([0,T]\times M)$. Assume that there is a measurable mapping
$$\Phi(t,x,\cdot,\cdot):\mathcal{L}(T_{x}M)\times  \mathcal{L}_s^2(T_{x}M)\rightarrow U$$
satisfing the following minimization condition for the Hamiltonian $\mathbb{H}$: 
$$\mathbb{H}(t,x,\Phi(t,x,DV(t,x),D^2V(t,x)),DV(t,x),D^2V(t,x)) = \inf_{u\in U}\mathbb{H}(t, x, u,DV(t,x),D^2V(t,x)).$$
Furthermore, suppose that under the feedback control law
\begin{equation}\label{pVe-eq1}
u(t) = \Phi(t, X(t),DV(t, X(t)),D^2V(t, X(t))),\quad t\in [0, T],
\end{equation}
the associated control system \eqref{system1} admits a unique solution $X(\cdot;0,x,u(\cdot))$, then $u(\cdot)$ is an
optimal (feedback) control.
\end{corollary}

Since the value function is not necessarily smooth, it is standard to study the HJB equation \eqref{phjb} within the framework of viscosity solutions. To this end, we first introduce the notion of parabolic semijets.

\begin{definition}
Let $\bm{v}\in C([0, T]\times M)$. The second-order parabolic superjet of $\bm{v}$ at a point $(t, x)\in (0, T)\times M$ is
\begin{align*}
\mathcal{P}^{2,+}\bm{v}(t,x):=\big\{&(\phi_t(t,x), D\phi(t,x), D^2\phi(t,x)),\phi\in C^{1,2}([0, T]\times M),\bm{v}-\phi \text{ attains}\\
&\quad\text{a local maximum at (t, $x$)}\big\},
\end{align*}
and second-order  parabolic subjet $\mathcal{P}^{2,-}\bm{v}(t,x):=-\mathcal{P}^{2,+}(-\bm{v}(t,x))$.
\end{definition}

Now we give the definition of the viscosity solution to the HJB equation \eqref{phjb}.

\begin{definition}\label{dfvss}
A function $\bm{v} \in C((0,T] \times M)$ is called a {\it viscosity supersolution} of \eqref{phjb} if for every $(t,x) \in (0,T) \times M$, the following conditions hold:
\begin{align*}
\left\{
\begin{aligned}
&p+\bfH(t,x,\chi,A)\le 0,&\quad (p,\chi,A)\in \mathcal{P}^{2,-}\bm{v}(t,x),\\
&\bm{v}(T,x)\ge h(x),
\end{aligned}
\right.
\end{align*}
Similarly, a function $\bm{v} \in C((0,T] \times M)$ is called a  {\it viscosity subsolution} of \eqref{phjb} if for every $(t,x) \in (0,T) \times M$, the following conditions hold:
\begin{align*}
\left\{
\begin{aligned}
&p+\bfH(t,x,\chi,A)\ge 0,&\quad(p,\chi,A)\in \mathcal{P}^{2,+}\bm{v}(t,x),\\
&\bm{v}(T,x)\le h(x).
\end{aligned}
\right.
\end{align*}
A function $\bm{v} \in C([0,T] \times M)$ is called a  {\it viscosity solution} of \eqref{phjb} if it is both a viscosity subsolution and a viscosity supersolution.
\end{definition}
\begin{theorem}\label{pExistence}
The value function $V(\cdot,\cdot)$ is the unique viscosity solution to \eqref{phjb}.
\end{theorem}

Furthermore, we establish the continuous dependence of viscosity solutions on parameters, a property that is central to numerical analysis and applications(  \cite{jakobsen_continuous_2002,jakobsen_continuous_2005,barles_error_2007,barles_convergence_2002}).
For simplicity of presentation, for $j=1,2$ and functions

$$
\begin{cases}
b_j(\cdot,x,\cdot) \colon &[0,T] \times U \to T_x M,  \quad x \in M, \\
\sigma_{ji}(\cdot, x, \cdot) \colon &[0,T] \times U \to T_x M, \quad x \in M, \quad i = 1, \dots, m, \\
f_j \colon &[0,T] \times M \times U \to \dbR,
\end{cases}
$$
we formally define the Hamiltonian
$$
\bfH_{j}(t,x,\chi,A):=\inf_{u\in U}\Big\{\frac{1}{2}\sum_{i=1}^m A(\sigma_{ji},\sigma_{ji})+\chi\big(b+\frac{1}{2}\sum_{i=1}^m D_{\sigma_{ji}}\sigma_{ji}\big)+f_j\Big\}.
$$
The last result is formulated as follows:

\begin{theorem}\label{pThmCont}
For $j=1,2$,  suppose  $b_j$, $\{\sigma_{ji}\}_{i=1}^m$, $f_j$  and $h_j$ satisfy Assumption \ref{condp}, and assume $V_j$ is the viscosity solution to
\begin{equation}\label{pExistence-eq1}
\begin{cases}
\pa_t V_{j} + \mathbf{H}_{j}(t, x, DV_j, D^2V_j) = 0, \\
\ns\ds V_j(T, x) = h_j(x).
\end{cases}
\end{equation}
Then we have 
\begin{equation}\label{pExistence-eq2}
\begin{aligned}
\max_{x \in M} |V_1 - V_2| \leq  & C\sqrt{T - t} \max_{s \in [t, T], x \in M, u \in U} \bigg\{ \sum_{i=1}^m \Big\{ \|\sigma_{1i} - \sigma_{2i}\|_g + \|D_{\sigma_{1i}}\sigma_{1i} - D_{\sigma_{2i}}\sigma_{2i}\|_g \Big\} + \|b_1 - b_2\|_g \bigg\} \\
&  + C(T - t) \max_{s \in [t, T], x \in M, u \in U} |f_1 - f_2| + \max_{x \in M} |h_1-h_2|,
\end{aligned}
\end{equation}
where the constant $C$ depends on the vector fields $b_j$, $\{\sigma_{ji}\}_{i=1}^m$, and functions $f_j$ and $h_j$ for $j=1,2$ .
\end{theorem}

The rest of the paper is organized as follows.  In section \ref{s2}, for the reader's convenience, we give some basic notations and results in Riemannian Geometry and viscosity solution in short. Then we prove the technical estimate in our work. In section \ref{s4}, we give the proofs of the main results.

\section{Some prelimilary results}\label{s2}

In this section, we begin by introducing essential notations and reviewing preliminary results that will serve as foundational tools for our subsequent analysis. This preparatory material is followed by the development of key auxiliary results in Subsections \ref{s21} and \ref{IE}, which provide the technical framework for our main contributions.

Subsection \ref{s23} then examines fundamental properties of the control system described in \eqref{eqve0}, building directly upon the theoretical infrastructure established in Subsection \ref{s21}. In Subsection \ref{s24}, we derive an approximation result for controls in the space $\mathcal{U}[t,T]$, a crucial step for our analytical approach.

Finally, Subsection \ref{s25} is devoted to establishing important qualitative and quantitative properties of the value function, culminating our theoretical development.

\begin{definition} \nmycite{Zhu2011}{ Definition 2.5}\label{zhudf}
Let  $\bm{v}\in C([0,T]\times M)$ and $(t, x)\in (0, T)\times M$, define
\begin{align*}
\bar{\mathcal{P}}^{2,-}\bm{v}(t, x):=\Big\{&(p,\chi, A)\in R\times \mathcal{L}(T_x M)\times\mathcal{L}_s^2(T_x M):
\exists (t_n, x_n, p_n, \chi_n, A_n)\in (0,T)\times M\times R\\
&\quad\times \mathcal{L}(T_{x_n}M)\times \mathcal{L}_s^2(T_{x_n}M), s.t.  (p_n, \chi_n, A_n)\in \mathcal{P}^{2,-}\bm{v}(t_n,x_n),\\
&\quad(t_n, x_n, \bm{v}(t_n, x_n), p_n, \chi_n, A_n)\rightarrow (t, x, \bm{v}(t, x), p, \chi, A)\Big\},
\end{align*}
and $\bar{\mathcal{P}}^{2,+}\bm{v}(t, x):=-\bar{\mathcal{P}}^{2,-}\{-\bm{v}(t, x)\}$.
\end{definition}

Recall that $\rho: M \times M \to \dbR$ denotes the Riemannian distance function induced by the metric on $M$. Under the framework of Definition \ref{zhudf}, we present a  maximum principle for systems of two continuous functions.

\begin{lemma} \label{pmaxp}  \nmycite{Azagra20081}{Theorem 3.8}
Let $\mathcal{O}_1,\mathcal{O}_2$ be the open sets of $M$ and $\mathcal{O}:=\mathcal{O}_1\times\mathcal{O}_2$, and let $\bm{v}_1,\bm{v}_2\in C([0,T]\times M),\varphi\in C^{1,2}((0,T)\times \mathcal{O})$ and define
$$\bm{w}(t,x_1,x_2) := \bm{v}_1(t,x_1)+\bm{v}_2(t,x_2 )-\varphi(t,x_1,x_2),\quad (t,x_1,x_2 )\in(0,T)\times \mathcal{O}.$$
Assume that $(\hat{t},\hat{x}_1,\hat{x}_2 )\in (0,T)\times \mathcal{O}$ such that
$$\bm{w}(\hat{t},\hat{x}_1,\hat{x}_2 )\ge \bm{w}(t,x_1,x_2),\quad (t,x_1,x_2)\in (0,T)\times \mathcal{O}.$$
Assume, moreover, that there is a $\delta> 0$ s.t. for every $N>0$  there is a constant $C$ such that
for $i = 1,2$
\begin{align}\label{condmaxp}
&p_i\le C \quad \text{whenever}\quad  (p_i,\chi_i,A_i)\in \mathcal{P}^{2,+}\bm{v}_i(t,x_i),\nonumber\\
&\rho(x_i, \hat{x}_i)+|t-\hat{t}|\le \delta \quad {\rm and}\quad |\bm{v}_i(t,x_i)|+\|\chi_i\|+\|A_i\|\le N.
\end{align}
Then for each $\epsilon>0$, there are $A_i\in \mathcal{L}_s^2(T_{\hat{x}_i}M)$, $ i=1,2$ such that
$$
\left\{
\begin{aligned}
&(p_i,D_{x_i}\varphi(\hat{t},\hat{x}),A_i)\in \bar{\mathcal{P}}^{2,+}\bm{v}_{i}(\hat{t},\hat{x_i}),\quad i=1,2\cr
&-(\tfrac{1}{\epsilon}+||A||)I\le 
\begin{bmatrix}
A_1& 0\\
0 &A_2
\end{bmatrix}
\le A+\epsilon A^2,\cr
&p_1+p_2 = \varphi_t(\hat{t},\hat{x}_1,\hat{x}_2),\cr
\end{aligned}
\right.
$$
for $A=D_{(\hat{x_1},\hat{x}_2)}^2\varphi(\hat{t},\hat{x}_1,\hat{x}_2)\in \mathcal{L}_s^2(T_{\hat{x}_1}M\times T_{\hat{x}_2}M)$.
\end{lemma} 

Next, we recall some fundamental properties of the Riemannian distance function $\rho(\cd,\cd)$.
\begin{lemma}\label{rho}\nmycite{deng_2016}{Lemma 2.2}
Let $x,y\in M$ such that $0<\rho(x,y)< \mathbf{i}_M$ (recall \eqref{6.3-eq1} for the definition of $\mathbf{i}_M$). Then we have 
\begin{align*}
&\nabla_x\rho^2(x,y)=-2\exp_x^{-1}(y),\quad\nabla_y\rho^2(x,y)=-2\exp_y^{-1}(x),\\
&L_{x y}\exp_x^{-1}(y)=-\exp_y^{-1}(x),\quad|\exp_x^{-1}(y)|=\rho(x,y).
\end{align*}
\end{lemma}

Next we give an estimate for $C^1$-smooth vector fields. 
\begin{lemma} \label{LIPV}
Let $X\in \mathsf{X}^1(M)$, $x,y \in M$ and $l:=\rho(x,y)$. Assume $0<l< \mathbf{i}_M$, and let $\gamma$ be the unique geodesic
connecting $x,y$ and parameterized by arc length with $\gamma(0) = x$ and $\gamma(l) = y$. Then
$$\|L_{y x}X(y)-X(x)\|_g\le \max_{\|v\|_g=1,z\in M}\|D_{v}X(z)\|_g\rho(x,y), \quad v\in T_zM.$$
\end{lemma}
\begin{proof}
First, we define the following quantities along the geodesic $\gamma$:
\begin{align*}
X(t) &:= X(\gamma(t)), \quad t \in [0,l], \\
\zeta(t) &:= \lan L_{\gamma(t) x}X(t) - X(0), L_{\gamma(t) x}X(t) - X(0)\ran _g, \quad t \in [0,l].
\end{align*}	
The derivative of $\zeta$ is computed as:
\begin{align*}
\zeta'(t)&=\lim_{\Delta t\rightarrow 0}\frac{\zeta(t+\Delta t)-\zeta(t)}{\Delta t}\\
&=\lim_{\Delta t\rightarrow 0}\bigg\{\lan L_{\gamma(t+\Delta t) x}X(t)+L_{\gamma(t) x}X(t),\frac{L_{\gamma(t+\Delta t) x}X(t+\Delta t)-L_{\gamma(t) x}X(t)}{\Delta t}\ran _g\\
&\qquad\qquad -2\lan X(0),\frac{L_{\gamma(t+\Delta t) x}X(t+\Delta t)-L_{\gamma(t) x}X(t)}{\Delta t}\ran _g\bigg\}\\
&=2\lan L_{\gamma(t) x}X(t)-X(0),L_{\gamma(t) x}D_{\gamma'(t)}X(t)\ran _g.
\end{align*}
Define the norm function
$$\zeta_1(t):=\|L_{\gamma(t) x}X(t)-X(0)\|_g =\sqrt{\zeta(t)},$$ 
which satisfies $\zeta_1(0) = 0$. Its derivative satisfies 
$$|\zeta_1'(t)|=\Big|\frac{\zeta'(t)}{\zeta_1(t)}\Big|\le \|L_{\gamma(t) x}D_{\gamma'(t)}X(t)\|_g.$$
Using equality \eqref{LXY}, we obtain
$$\zeta_1(l)=|\zeta_1(l)-\zeta_1(0)|\le \max_{\theta\in [0,l]}|\zeta_1'(\theta)| l\le \max_{\|v\|_g=1,z\in M}\|D_{v}X(z)\|_g\rho(x,y),\quad v\in T_zM,$$
which completes the proof.
\end{proof}

\subsection{Equivalence and extension lemmas on submanifolds}\label{s21}

This subsection establishes preliminary results for proving the continuous dependence of solutions to equation \eqref{eqve0} on initial data and control terms. Following our methodological approach (outlined in the introduction), we organize the key technical steps through the following lemmas:
\begin{itemize}[itemsep=0pt, parsep=0pt, topsep=0pt] 
\item
Isometric embedding of $M$ into Euclidean space $\dbR^{n_1}$ ($n_1 \geq n$) via Lemmas \ref{Ne} and \ref{eqrho};
\item Derivation of key equality in the Euclidean framework (Lemma \ref{LMXX});
\item Extension of associated vector fields to $\dbR^{n_1}$ (Lemma \ref{lm62});
\item Projection of Euclidean estimates back to the manifold $M$ (Lemma \ref{lm61}).
\end{itemize}

The central results of this subsection, presented in Lemmas \ref{lm62} and \ref{lm61}, constitute key technical components of our analysis. While these results are likely part of the existing mathematical literature, we have been unable to locate explicit references stating them in the precise form required for our purposes. To ensure rigorous treatment and maintain self-containment of our work, we furnish complete proofs with detailed derivations.

First, we recall the notation of isometric embedding:

\begin{definition}\label{isoembed} 
We say that an $n$-dimensional Riemannian manifold $(\wt M,\wt g)$ is isometrically embedded into $(\dbR^{n_1},g^{E})$ ($n_1 \in \mathbb{N}$, $g^{E}$ the standard Euclidean metric), if there exists a smooth map $\mathcal{E}: \wt M \rightarrow \dbR^{n_1}$ satisfying:
\begin{enumerate}
\item Smooth Embedding:
\begin{itemize}
\item $\mathcal{E}$ is a smooth immersion: 
$\forall x \in \wt M,\text{the tangent map } d\mathcal{E}_x: T_x\wt M \rightarrow T_{\mathcal{E}(x)}\dbR^{n_1} \text{ is injective}$;
\item $\mathcal{E}: \wt M \to \mathcal{E}(\wt M)$ is a homeomorphism with subspace topology;
\end{itemize}

\item Metric Isometry:
$\left\langle d\mathcal{E}_x(v), d\mathcal{E}_x(w) \right\rangle_{g^{E}} = \lan v,w\ran_{\wt g},  \forall x \in \wt M, \, \forall v,w \in T_x\wt M$.
\end{enumerate}
\end{definition}

Denoted by $g^{E}$ the standard Euclidean metric on Euclidean space(the dimension of the Euclidean space is sometimes omitted for simpliciy).

\begin{lemma} \nmycite{Han2006}{Theorem 1.0.3} \label{Ne} There exists an isometric embedding $\mathcal{E}: (M ,g)\rightarrow (\dbR^{n_1},g^E)$  for $n_1=\max\big\{\frac{n(n+5)}{2},\frac{n(n+3)}{2}+5\big\}$. 
\end{lemma}

In the rest of this paper, we fix $n_1:=\max\big\{\frac{n(n+5)}{2},\frac{n(n+3)}{2}+5\big\}$.

\begin{definition}\nmycite{Lee}{p.98}\label{dfemb}
An embedded submanifold of Euclidean space $\dbR^k (k\in \mathbb{N})$ is a subset $S$ that is a manifold in the subspace topology, endowed with a smooth structre with respect to which the inclusion map $\iota:S\rightarrow \dbR^k$ is a smooth embedding.  
\end{definition}

\begin{lemma} \nmycite{Lee}{Proposition 5.2}\label{lee52}
Suppose $M$ is isometrically embedded into $(\dbR^{n_1},g^E)$ through $\mathcal{E}$, then $\mathcal{E}(M)$ is an embeded submanifold of $(\dbR^{n_1},g^E)$. 
\end{lemma}

Following the setup in Lemma \ref{lee52}, the inclusion map  $\iota:\mathcal{E}(M)\rightarrow \dbR^{n_1}$ is a smooth embedding. The space $\mathcal{E}(M)$ inherits a Riemannian metric $\wt g$ from $g^{E}$, given by 
\begin{align}
&\wt g:\mathsf{X}(\mathcal{E}(M)) \times \mathsf{X}(\mathcal{E}(M)) \to C^{\infty}(\mathcal{E}(M)),\nonumber\\
&\wt g_x(u, v) := \langle u,v\rangle_{g^{E}(x)}, \quad x\in \mathcal{E}(M),\quad u,v\in T_x\mathcal{E}(M). \nonumber
\end{align}
This endows $(\mathcal{E}(M),\wt g)$ with the structure of a Riemannian manifold.

\begin{lemma}\nmycite{Lee}{Exercise 13.27}\label{eqrho}
Let $\rho$ denote the Riemannian distance on $M$, and let $\wt\rho$ be the Riemannian distance on $\mathcal{E}(M)$ induced by $\wt g$. Then the following equality holds:
$$\wt \rho(\mathcal{E}(x),\mathcal{E}(y)))=\rho(x,y),\quad \forall x,y\in M.$$ 
\end{lemma}

\begin{proof}
By definition, the Riemannian distances are given by
\begin{align*}
\rho(x, y) &= \inf \left\{ \mathsf{L}_g(\gamma) \mid \gamma: [t_0, t_1] \to M \text{ piecewise smooth}, \gamma(t_0) = x, \gamma(t_1) = y \right\}, \\
\widetilde{\rho}(\wt x, \wt y) &= \inf \left\{ \mathsf{L}_{\widetilde{g}}(\wt\gamma) \mid \wt \gamma: [\wt t_0, \wt t_1] \to \mathcal{E}(M) \text{ piecewise smooth}, \wt \gamma(\wt t_0) = \wt x, \gamma(\wt t_1) = \wt y \right\},
\end{align*}
where the curve lengths are defined as
$$\mathsf{L}_g(\gamma) = \int_{t_0}^{t_1} \|\gamma'(t)\|_g dt, \quad 
\mathsf{L}_{\widetilde{g}}(\wt \gamma) = \int_{\wt t_0}^{\wt t_1} \|\wt\gamma'(t)\|_{\widetilde{g}}dt.$$
Since $\mathcal{E}$ is an isometric embedding, for any piecewise smooth curve $\gamma:[t_0, t_1] \to M$, the composition $\wt\gamma = \mathcal{E} \circ \gamma$ satisfies $\wt\gamma'(t) = d\mathcal{E}_{\gamma(t)}(\gamma'(t))$ for all $t\in [t_0,t_1]$.

The isometry condition is expressed by
\begin{align*}
\lan v, w\ran _{g(x)}&=\lan d\mathcal{E}_x(v),d\mathcal{E}_x(w)\ran _{g^{E}(\mathcal{E}(x))} 
=\wt g_{\mathcal{E}(x)}(d\mathcal{E}_x(v), d\mathcal{E}_x(w)),\quad x\in M,v,w\in T_xM
\end{align*}
which implies the length preservation property:
\begin{equation}\label{lgg}
\mathsf{L}_g(\gamma) = \int_{t_0}^{t_1} \|\gamma'(t)\|_g dt = \int_{t_0}^{t_1} \|\wt\gamma'(t)\|_{\widetilde{g}}dt=\mathsf{L}_{\widetilde{g}}(\wt\gamma).
\end{equation}

Furthermore, since $\mathcal{E}$ is an embedding (and hence a diffeomorphism onto its image), the mapping between two piecewise smooth path sets
\begin{align*}
\{\gamma \mid \gamma \text{ connects } x \text{ to } y \text{ in } M\} \to \big\{\wt\gamma \mid \wt\gamma \text{ connects } \mathcal{E}(x) \text{ to } \mathcal{E}(y) \text{ in } \widetilde{M}\big\},\quad \gamma \rightarrow \mathcal{E} \circ \gamma
\end{align*}
is bijective. From equation \eqref{lgg}, we see that corresponding curves have equal lengths: $\mathsf{L}_g(\gamma) = \mathsf{L}_{\widetilde{g}}(\mathcal{E} \circ \gamma)$. 
This establishes the equality between the sets of lengths:
$$
\left\{ \mathsf{L}_g(\gamma) \mid \gamma \text{ connects } x \text{ to } y \right\} = \left\{ \mathsf{L}_{\widetilde{g}}(\wt\gamma) \mid \wt\gamma \text{ connects } \mathcal{E}(x) \text{ to } \mathcal{E}(y) \right\}.
$$
and consequently their infima must coincide:
$$
\widetilde{\rho}(\mathcal{E}(x), \mathcal{E}(y)) = \inf \mathsf{L}_{\wt g}(\wt\gamma) = \inf \mathsf{L}_g(\gamma) = \rho(x, y).
$$
\end{proof}

The  lemma below establishes the equivalence between the control system \eqref{eqve0} and a control system in $\dbR^{n_1}$.

\begin{lemma}\label{LMXX}
Suppose $X$ is the solution to equation \eqref{eqve0}, and suppose $M$ is isometriclly embedded into
$\dbR^{n_1}$ through $\mathcal{E}$. Define
\begin{equation}\label{dfbs}
\wt b(s,\mathcal{E}(x),u):=d\mathcal{E}_{x}(b(s,x,u)),\quad 
\wt \sigma_i(s,\mathcal{E}(x),u):=d\mathcal{E}_{x}(\sigma_i(s,x,u)),
\quad \wt \xi:=\mathcal{E}(\xi),
\end{equation}
for $x\in M,u\in U,s\in [t,T]$ and $i=1,...,m$. Then  $\wt X:=\mathcal{E}(X)$  solves
\begin{align}\label{eqve11}
\left\{
\begin{aligned}
&d\wt X(s)=\wt b(s,\wt X(s),u(s))ds+\sum_{i=1}^m\wt \sigma_i(s,\wt X(s),u(s))\circ dW^i(s), &\quad s\in (t,T),\\
&\wt X(t)=\wt \xi,  
\end{aligned}
\right.
\end{align}
in $\mathcal{E}(M)$. 
\end{lemma}
We borrow the idea for the proof of \nmycite{hsu}{Proposition 1.2.4} to prove Lemma \ref{LMXX}.

\begin{proof}
Let $\varphi(\cdot,\cdot) \in C^{1,2}([t,T]\times \dbR^{\bfn_1})$. Applying equation \eqref{df61} to  the composition $\varphi(\cdot,\mathcal{E}(\cdot)) \in C^{1,2}([t,T]\times M)$ 
and using the definition \eqref{dfbs}, we obtain for any $t_1\in [t,T]$: 
\begin{align}\label{vpx}
\varphi(t_1,\wt X(t_1))& = \varphi(t,\mathcal{E}(\xi))\! +\! \int_{t}^{t_1}\!\Big(\frac{\partial}{\partial s}\!+\! b\!+\!\frac{1}{2}\sum_{i=1}^m\!\sigma_i^2\Big)\varphi(s,\mathcal{E}(X(s))ds\!+\!\int_{t}^{t_1}\!\sum_{i=1}^m\!\sigma_i\varphi(s,\mathcal{E}( X(s)))\circ dW^{i}(s) \nonumber\\
&=\varphi(t,\wt \xi) + \int_{t}^{t_1}\Big(\frac{\partial}{\partial s}+ \wt b+\frac{1}{2}\sum_{i=1}^m\wt \sigma_i^2\Big)\varphi(s,\wt X(s))ds+\int_{t}^{t_1}\sum_{i=1}^m\wt \sigma_i\varphi(s,\wt X(s))\circ dW^{i}(s).
\end{align}
The first equation of \eqref{vpx} treats $\varphi(\cdot,\mathcal{E}(\cdot))$ as a function on $[t,T]\times M$, whereas the second equation considers $\varphi(\cdot,\cdot)$ as a function on $[t,T]\times\dbR^{\bfn_1}$. This demonstrates that $\wt X = \mathcal{E}(X)$ satisfies equation \eqref{eqve11}.
\end{proof}

By Lemma \ref{lee52}, $\mathcal{E}(M)$ is a submanifold of $\dbR^{n_1}$, making it natural to employ concepts and theorems related to submanifolds. Throughout the following definitions and lemmas in this subsection, we consistently assume that:
$$
\wt M \text{ is a compact embedded submanifold of an Euclidean space, equipped with the induced metric } \wt g.
$$

(as seen in Proposition \ref{pSExistence} and \ref{prop41}, we will take $\wt M$ to be $\mathcal{E}(M)$.) 

While the subsequent definitions and lemmas remain valid for a general manifold $\wt M$ and its ambient Euclidean space of arbitrary finite dimension, we adopt the following notational conventions for simplicity: when the dimension of $\wt M$ is needed, we denote it by $n$; when referring to the dimension of the ambient Euclidean space containing $\wt M$, we typically denote it by $n_1$. We will employ both the Euclidean norm $\|\cdot\|$ and tensor norm interchangeably throughout this subsection when the context makes the distinction clear.

\begin{definition}\nmycite{Lee}{p.138}
Assume that $\wt M$ is an embedded submanifold of Euclidean space $\dbR^{n_1}$.
The normal bundle of  $\wt M$, denoted by $\mathcal{N}(\wt M)$, is the subset of $T\dbR^{n_1}$ consisting of vectors that are normal to $\wt M$
$$\mathcal{N}(\wt M):=\{(x,v_x)\in \dbR^{n_1}\times \dbR^{n_1},x\in \wt M,v_x\in \mathcal{N}_x\},$$
where $\mathcal{N}_x:=(T_x\wt M)^{\perp}$ is the normal space of $\wt M$ at $x$. Denote by $\mathcal{N}(\mathcal{O})$ the restriction of $\mathcal{N}(\wt M)$ to $\mathcal{O}\subset \wt M$ without ambiguity.
Denote by $\mathsf{N}^k(\wt M)(k\in \mathbb{N})$ the space of the $C^k$-smooth normal vector fields on $\wt M$. 
\end{definition}

\begin{definition}\nmycite{Lee}{p.139}
Assume that $\wt M$ is an embedded submanifold of Euclidean space $\dbR^{n_1}$.
A tubular neighborhood of $\wt M$ is a neighborhood $\mathscr{U}$ of $\wt M$ in $\dbR^{n_1}$ that is the diffeomorphic image under $E_{\mathcal{N}}$ of an open subset $\mathcal{N}_{\delta}\subset \mathcal{N}(\wt M)$ of the form
\begin{equation}\label{ND}
\mathcal{N}_{\delta}:=\{(x,v_x)\in \mathcal{N}(\wt M):\|v_x\|< \delta(x)\},
\end{equation}
for some positive continuous function $\delta:\wt M\rightarrow \dbR$, and for $E_{\mathcal{N}}:\mathcal{N}(\wt M)\rightarrow \dbR^{n_1}$,
\begin{equation}\label{EN}
E_{\mathcal{N}}(x,v_x):=x+v_x.
\end{equation}
\end{definition}

The next lemma guarantees the existence of the tubular neighborhood,

\begin{lemma}\label{Tubular}\nmycite{Lee}{Theorem 6.24}  Assume that $\wt M$ is an embedded submanifold of Euclidean space, then $\wt M$ has a tubular neighborhood. 
\end{lemma}
\begin{remark}\label{rktu} Since $\wt M$ is compact, we can choose the function $\delta$ defined in \eqref{ND} to be a small positive constant $\epsilon$.
\end{remark}

Thus, for simplicity of presentation, denote by 
\begin{equation}\label{epsm}
\epsilon_{\wt M}:=\sup\{\epsilon \mid E_{\mathcal{N}} \text{ is diffeomorphsim between } \mathcal{N}_{\epsilon}\text{ and } E_{\mathcal{N}}(\mathcal{N}_{\epsilon})\}
\end{equation}
the maximum radius of tubular neighborhood. Denote by 
\begin{equation}\label{meps}
\wt M_{\epsilon}:=E_{\mathcal{N}}(\mathcal{N}_{\epsilon}),\quad \epsilon\in (0,\epsilon_{\wt M})
\end{equation}
the $\epsilon$- tubular neighborhood of $\wt M$
in the rest of this paper.

Denote by $D^{E}$ the Levi-Civita connection on the ambient Euclidean space.
\begin{lemma} \label{Gauss}\cite[pp.125--134]{doCarmo}
Assume that $\wt M$ is an embedded submanifold of Euclidean space.
Let $\bfX,\bfY\in \mathsf{X}^1(\wt M),\bfn\in \mathsf{N}^1(\wt M)$. Then we have the following  Gauss's formula
$$ D_\bfX^E\bfY=\wt D_\bfX \bfY+\mathbf{h}(\bfX,\bfY),$$
for $\wt D_\bfX\bfY:=(D_\bfX^E\bfY)^{\top}$ the tangential component and $\bfh(\bfX,\bfY):=(D_\bfX^E\bfY)^{\perp}$ the normal component,  and the Weingarten's formula
$$D_\bfX^E\bfn=\mathbf{A}(\bfn,\bfX)+D_\bfX^{\perp}\bfn,$$
for $\mathbf{A}(\xi,\bfX):=(D_{\bfX}^E\bfn)^{\top}$ the tangential component and $D_\bfX^{\perp}\bfn:=(D_\bfX^E\bfn)^{\perp}$ the normal component.
\end{lemma}

In classic references \cite[p.126]{doCarmo}, it can be proved that $\wt D$ is the Levi-Civita connection induced by $\wt g$, and $\mathbf{h}$ is always called the second fundamental form, $\mathbf{A}_{\xi} := \mathbf{A}(\xi, \cdot)$ is called the Weingarten transform, and $D^{\perp}$ is called the normal connection. In the next lemma, it can be found that $\mathbf{h}$, $\mathbf{A}$ are tensors and $D^{\perp}$ is a connection.

\begin{lemma}\label{basicprop}\cite[pp.125--134]{doCarmo}  Let $\bfX,\bfY,\bfZ\in \mathsf{X}^1(\wt M)$, $\bfn_1,\bfn_2\in \mathsf{N}^1(\wt M)$, $\lambda\in C^{1}(\wt M)$. Then the following formulas hold.
\begin{itemize}
\item[1)] $\wt D:\mathsf{X}^0(\wt M)\times \mathsf{X}^1(\wt M)\rightarrow \mathsf{X}^0(\wt M)$ satisfies
$$\wt D_{\lambda \bfX+\bfY}\bfZ =\lambda \wt D_{\bfX}\bfZ+\wt D_{\bfY}\bfZ,\quad \wt D_\bfX(\bfY+\bfZ)=\wt D_\bfX\bfY+\wt D_\bfX\bfZ,\quad \wt D_{\bfX}(\lambda \bfY)=\bfX(\lambda)\bfY+\lambda \wt D_{\bfX}^{\perp}\bfY.$$
\item[2)] $\mathbf{h}: \mathsf{X}^0(\wt M)\times \mathsf{X}^0(\wt M)\rightarrow \mathsf{N}^0(\wt M)$ satisfies 
$$\mathbf{h}(\lambda \bfX+\bfY,\bfZ)=\lambda\mathbf{h}(\bfX,\bfZ)+\mathbf{h}(\bfY,\bfZ),\quad \mathbf{h}(\bfX,\bfY)=\mathbf{h}(\bfY,\bfX).$$
\item[3)] $\mathbf{A}:\mathsf{N}^0(\wt M)\times \mathsf{X}^0(\wt M)\rightarrow \mathsf{X}^0(\wt M)$  satisfies
$$\mathbf{A}(\lambda \bfn_1+\bfn_2,\bfX)=\lambda\mathbf{A}(\bfn_1,\bfX)+\mathbf{A}(\bfn_2,\bfX),\quad \lan \mathbf{A}(\bfn_1,\bfX),\bfY\ran =\lan \mathbf{h}(\bfX,\bfY),\bfn_1\ran .$$
\item[4)] $D^{\perp}:\mathsf{X}^0(\wt M)\times \mathsf{N}^1(\wt M)\rightarrow \mathsf{N}^0(\wt M)$ satisfies
$$D_{\lambda \bfX+\bfY}^{\perp}\bfn_1 =\lambda D_{\bfX}^{\perp}\bfn_1+D_{\bfY}^{\perp}\bfn_1,\quad D_\bfX^{\perp}(\bfn_1+\bfn_2)=D_\bfX^{\perp}\bfn_1+D_\bfX^{\perp}\bfn_2,\quad D_{\bfX}^{\perp}(\lambda \bfn_1)=\bfX(\lambda)\bfn_1+\lambda D_{\bfX}^{\perp}\bfn_1.$$
\end{itemize}
\end{lemma}

Define the index sets as follows: 
$$
\mathcal{I}_1 := \big\{1,\, 2,\, \dots,\, n\big\}, \quad 
\mathcal{I}_2 := \big\{n+1,\, \dots,\, n_1\big\}, \quad 
\mathcal{I} := \mathcal{I}_1 \cup \mathcal{I}_2.
$$
Denote by $B_k(r):=\big\{y\in \dbR^k,\|y\|< r\big\}$  the standard $k$-dimensional Euclidean ball of radius $r>0$ centered at the origin; by $\wt B_{w}^{\wt\rho}(r):=\big\{y\in \wt M,\wt\rho(w,y)< r\big\}$ the geodesic ball of radius $r>0$ centered at $w\in \wt M$ under the distance function $\wt\rho$ on $\wt M$ without ambiguity, and  by $\bfi_{\wt M}$ the global injectivity radius of $\wt M$.

Since we need to perform calculations in local coordinates, it is necessary to construct a local smooth orthonormal normal frame field. The following lemma guarantees the existence of such a local smooth orthonormal normal frame field for the normal bundle over sufficiently small geodesic balls. 

\begin{lemma}\label{lmn}
Assume that $\wt M$ is an $n$-dimensional compact embedded submanifold of $\dbR^{n_1}$. 
For any $w\in \wt M$, there exists a smooth orthonormal normal frame field $\{\bfn^{k}\}_{k=1}^{n_1-n}$ on $\mathcal{N}(\wt B_{w}^{\wt \rho}(\bfi_{\wt M}))$.
\end{lemma}
Although Lemma \ref{lmn} should be considered a standard result, we include a proof here for completeness as we were unable to locate a precise reference in the literature.

\begin{proof}
Fix a point $w \in \wt M$ and consider the geodesic ball $\mathscr{U}_{w} := \wt B_w^{\wt\rho}(\bfi_{\wt M})$. Since $\mathscr{U}_{w}$ is diffeomorphic to the Euclidean ball $B_n(\bfi_{\wt M})$, which is contractible via the homotopy $\mathsf{h}: B_n(\bfi_{\wt M}) \times [0,1] \to B_n(\bfi_{\wt M})$, $\mathsf{h}(x,t) = (1-t)x$, contracting the ball to the origin, it follows that $\mathscr{U}_{w}$ is also contractible. 

By a standard result in vector bundle theory (see \nmycite{Husemoller}{Corollary 4.8}), any vector bundle over a contractible base space is trivial. Therefore, the normal bundle $\mathcal{N}(\mathscr{U}_{w})$ admits a smooth trivialization
$$\Psi: \mathcal{N}(\mathscr{U}_{w}) \to \mathscr{U}_{w} \times \dbR^{n_1-n}.$$
Let ${ e_1, \dots, e_{n_1-n} }$ denote the standard basis of $\dbR^{n_1-n}$. We define smooth sections
$$
s_i: \mathscr{U}_{w} \to \mathcal{N}(\mathscr{U}_{w}), \quad s_i(p) = \Psi^{-1}(p, e_i), \quad i = 1, \dots, n_1-n.
$$
Since $\Psi$ is a smooth bundle isomorphism, each $s_i$ is a smooth normal vector field. Moreover, for each $p \in \mathscr{U}_{w}$, the fiber map $\Psi(p,\cdot): \mathcal{N}_p \to \dbR^{n_1-n}$ is a linear isomorphism. Consequently, the sections ${ s_1(p), \dots, s_{n_1-n}(p) }$ form a basis for $\mathcal{N}_p$, as they correspond to the basis $e_1, \dots, e_{n_1-n}$ under $\Psi(p,\cdot)$. Applying the Gram-Schmidt orthogonalization process to these sections yields a smooth orthonormal normal frame field $\{\bfn^{k}\}_{k=1}^{n_1-n}$ for $\mathcal{N}(\mathscr{U}_{w})$.
\end{proof}

We now present the first main result of this subsection:

\begin{lemma} \label{lm62}
Let $\wt M$ be an $n$-dimensional compact embedded submanifold of $\dbR^{n_1}$. Consider a family of vector fields $\bfX(\cdot,u) \in \mathsf{X}^2(\wt M)$ parameterized by $u \in U$ that satisfies the Lipschitz condition:
$$\max_{x\in \wt M,\|v\|=1}\Big\{\|\bfX(x,u_1)-\bfX(x,u_2)\|+\|\wt D_{v}\bfX(x,u_1)-\wt D_{v}\bfX(x,u_2)\|\Big\}\le C_L\|u_1-u_2\|,$$
for all $v \in T_x \wt M$. Then there exists an extension $\overline{\bfX}(\cdot,u) \in \mathsf{X}^2(T\dbR^{n_1})$ with compact support on $\dbR^{n_1}$ such that:

1. $\overline{\bfX}\big|_{\wt M} = \bfX$ (the restriction agrees with the original vector field);

2. The following uniform bound holds for some constant $C_E > 0$:
\begin{align*}
\|\obfX(z_1,u_1)-\obfX(z_2,u_2)\|+\|D_{\obfX}^E \obfX(z_1,u_1)-D_{\obfX}^E \obfX(z_2,u_2)\|&\le C_{E}\big(\|u_1-u_2\|+\|z_1-z_2\|\big).\nonumber\\
\end{align*}
\end{lemma} 

\begin{proof} We divide the proof into five steps. 

\textbf{Step 1.} Let $\mathscr{A}$ be the smooth structure on $\wt M$. Fix $r\in (0,\bfi_{\wt M})$. For any $w \in \wt M$, define $\mathscr{U}_{w} := \wt B_w^{\wt\rho}(r)$ and let $\varphi_w := \exp_w^{-1}: \mathscr{U}_{w} \rightarrow B_{n}(r)$ be the inverse exponential map. Given any chart $(\mathscr{U}',\phi) \in \mathscr{A}$ with $\mathscr{U}_w \cap \mathscr{U}' \neq \emptyset$, the transition maps
$$\phi\circ\varphi_{w}^{-1}:\varphi_{w}(\mathscr{U}_w\cap \mathscr{U}')\rightarrow \phi(\mathscr{U}_w\cap \mathscr{U}'),\quad\varphi_w\circ\phi^{-1}:\phi(\mathscr{U}_w\cap \mathscr{U}')\rightarrow \varphi_{w}(\mathscr{U}_w\cap \mathscr{U}')$$
are smooth since $\varphi_w$ is a diffeomorphism. This shows that $(\mathscr{U}_w, \varphi_w) \in \mathscr{A}$.

By the compactness of $\wt M$, there exists a finite cover $\{\mathscr{U}_{\alpha}\}_{\alpha \in \mathcal{S}}$ where $\mathscr{U}_{\alpha} := \mathscr{U}_{w_{\alpha}}$ and $\varphi_{\alpha} := \varphi_{w_{\alpha}}$. The collection $\mathscr{A}_1 := \{(\mathscr{U}_{\alpha}, \varphi_{\alpha}) : \alpha \in \mathcal{S}\}$ forms a smooth atlas on $\wt M$.

Using Lemma \ref{Tubular},Remark \ref{rktu}, and the notations in \eqref{epsm}-\eqref{meps}, there exists $\epsilon_{\wt M}>0$ such that the normal exponential map $E_{\mathcal{N}}$ restricts to a diffeomorphism between $\mathcal{N}_{\epsilon} \subset \mathcal{N}(\wt M)$ and $\wt M_{\epsilon} := E_{\mathcal{N}}(\mathcal{N}_{\epsilon})$ for $\epsilon\in (0,\epsilon_{\wt M})$. Fix some $\epsilon \in (0,\epsilon_{\wt M})$, then we construct a smooth atlas on $\wt M_{\epsilon} \subset \dbR^{n_1}$ as follows:

For each $\alpha \in \mathcal{S}$, Lemma \ref{lmn} provides a smooth orthonormal normal frame field $\{\bfn_{\alpha}^k\}_{k \in \mathcal{I}_2}$ on $\mathcal{N}(\mathscr{U}_{\alpha})$. Let $I_{n_1}$ denote the standard Euclidean coordinates on $\dbR^{n_1}$ and define $\psi^0_{\alpha} := I_{n_1} \circ \iota \circ \varphi_{\alpha}^{-1}: B_{n}(r) \rightarrow \dbR^{n_1}$ as the coordinate representation of the inclusion $\iota$ on $\mathscr{U}_{\alpha}$. For $\mathscr{W} := B_{n}(r) \times B_{n_1-n}(\epsilon)$, define 
$$\psi_{\alpha}:\mathscr{W}\rightarrow \wt M_{\epsilon},\quad \psi_{\alpha}(x,v):=\psi_{\alpha}^0(x)+\sum_{i\in \mathcal{I}_2}v_i \bfn_{\alpha}^i(x).$$
Since $E_{\mathcal{N}}$ is a diffeomorphism between $\wt M_{\epsilon}$ and $\mathcal{N}_{\epsilon}$, each $z \in \mathscr{V}_{\alpha} := \psi_{\alpha}(\mathscr{W}) \subset \wt M_{\epsilon}$ corresponds to a unique $(z_x, z_v) \in \mathcal{N}_{\epsilon}$ with $\|z_v\| \leq \epsilon$. The coordinates
$$x:=(\psi_{\alpha}^0)^{-1}(z_x)\in B_{n}(r),\quad v:=[\lan z_v,\bfn_{\alpha}^{n+1}\ran ,...,\lan z_v,\bfn_{\alpha}^{n_1}\ran ]\in B_{n_1-n}(\epsilon)$$
are uniquely determined, making $\psi_{\alpha}$ a bijection between $\mathscr{V}_{\alpha}$ and $\mathscr{W}$. The smoothness of $\iota$ implies that $\psi_{\alpha}^0$ and consequently $\psi_{\alpha}$ are smooth in both $x$ and $v$. Therefore, both $\psi_{\alpha}$ and its inverse are smooth diffeomorphisms.

Since $\wt M_{\epsilon}$ is an embedded submanifold, for every $z \in \wt M_{\epsilon}$, there exists a unique $(z_x,z_v)\in \mathcal{N}_{\epsilon}$ and some $\alpha\in \mathcal{S}$ such that $z_x\in \psi_{\alpha}^0(B_n(r))$. Thus $z \in \mathscr{V}_{\alpha}$, which implies $\wt M_{\epsilon} = \medcup_{\alpha \in \mathcal{S}} \mathscr{V}_{\alpha}$. For any $\alpha, \beta \in \mathcal{S}$ with $\mathscr{V}_{\alpha} \cap \mathscr{V}_{\beta} \neq \emptyset$, the transition map
$$\psi_{\alpha}^{-1}\circ\psi_{\beta}:\psi_{\beta}^{-1}(\mathscr{V}_{\alpha}\cap \mathscr{V}_{\beta})\subset\dbR^{n_1}\rightarrow \psi_{\alpha}^{-1}(\mathscr{V}_{\alpha}\cap \mathscr{V}_{\beta})\subset\dbR^{n_1}$$
is smooth. Thus, $\mathscr{A}_2 := \{(\mathscr{V}_{\alpha}, \psi_{\alpha}^{-1}) : \alpha \in \mathcal{S}\}$ forms a smooth atlas on $\wt M_{\epsilon}$.  

The smoothness of $\psi_{\alpha}^{-1} \circ I_{n_1}$ and $I_{n_1} \circ \psi_{\alpha}$ shows that $\mathscr{A}_2$ and the standard Euclidean atlas $I_{n_1}$ belong to the same smooth structure on $\wt M_{\epsilon}$.

Let $z\in \mathscr{V}_{\alpha}$ and let $\{\frac{\partial}{\partial z_j}\big|_z\}_{j \in \mathcal{I}}$ be the natural basis of $T_z\wt M_{\epsilon}$ induced by $I_{n_1}$, and let $\{\frac{\partial}{\partial x_i}|_z\}_{i\in \mathcal{I}_1}\cup\{\frac{\partial}{\partial v_i}\big|_z\}_{i\in \mathcal{I}_2}$ be the basis induced by $(\mathscr{V}_{\alpha},\psi_{\alpha}^{-1})$. The basis transformation is given by: 
\begin{equation}\label{eqzv}
\frac{\partial}{\partial z_i}=\sum_{j\in \mathcal{I}_1} a_{ij}(z)\frac{\partial}{\partial x_j}+\sum_{j\in \mathcal{I}_2} a_{ij}(z)\frac{\partial}{\partial v_j},\quad  z\in \mathscr{V}_{\alpha}, 
\end{equation}
where the coefficients $a_{ij},i,j\in \mathcal{I}$ are smooth functions depending on $\alpha \in \mathcal{S}$. In matrix form, with $\mathcal{Z}:=[\frac{\partial }{\partial z_1},...,\frac{\partial }{\partial z_{n_1}}]^{\top}$ and $\mathcal{X}:=[\frac{\partial }{\partial x_1},...,\frac{\partial }{\partial v_{n_1-n}}]^{\top}$, we have:
\begin{equation}\label{eqzvm}
\mathcal{Z}=\mathbb{A}_{\alpha} \mathcal{X},\quad z\in\mathscr{V}_{\alpha},\quad \alpha\in \mathcal{S}
\end{equation}
where $\mathbb{A}_{\alpha}:=(a_{ij}(z))_{i,j\in \mathcal{I}}$. Since $r\in (0,\bfi_{\wt M})$ and $\epsilon\in (0,\epsilon_{\wt M})$ are fixed, $\psi_{\alpha}^{-1}$ extends to a diffeomorphism  between $\overline{\mathscr{V}_{\alpha}}$ and $\overline{\mathscr{W}}$, Consequently, the functions $a_{ij}$ can be extended to smooth functions on $\overline{\mathscr{V}_{\alpha}}$ for $i,j\in \mathcal{I}$, 
 which implies that $\mathbb{A}_{\alpha}$ and its inverse are uniformly bounded:
\begin{equation}\label{AZB}
\big\|\tfrac{\partial }{\partial x_i}\big\|,\big\|\tfrac{\partial }{\partial v_j}\big\|
\le \|\mathcal{X}\|\le \max_{z\in \overline{\mathscr{V}_{\alpha}}}\|\mathbb{A}_{\alpha}^{-1}\| \|\mathcal{Z}\|\le \sqrt{n_1} \max_{z\in \overline{\mathscr{V}_{\alpha}}}\|\mathbb{A}_{\alpha}^{-1}\|,\quad i\in\mathcal{I}_1, j\in\mathcal{I}_2.
\end{equation}

\ss

\textbf{Step 2.} In this step, we extend the vector field $\bfX$ to $\wt M_{\epsilon}$  and study the properties of the extension. 

Extend $\bfX$ to $\wt M_{\epsilon}$ by defining:
\begin{align}\label{stp2}
\obfX(z,u):=\left\{
\begin{aligned}
& \bfX(x,u)\lambda(\|v\|^2),& & z\in \wt M_{\epsilon},\cr
& 0,& &z\not \in \wt M_{\epsilon},
\end{aligned}
\right. {\rm for\quad }
\lambda(s):=\left\{
\begin{aligned}
& \frac{(s-\epsilon^2)^2}{\epsilon^4},& & 0\le s\le \epsilon,\cr
& 0,& &x>\epsilon.
\end{aligned}
\right.
\end{align}
It can be seen $|\lambda|\le 1$ and $|\lambda'|\le \frac{C}{\epsilon^2}$. 
Note that $ \lambda $ is supported on $ [0, \epsilon] $, and thus $ \obfX $ is supported on $ \wt M_{\epsilon}$.  

Using \eqref{eqzv}, we obtain in local coordinates:
\begin{align}\label{eqdzv}
D_{\frac{\partial }{\partial z_i}}^E \obfX(z,u)
&=\sum_{j\in \mathcal{I}_1}a_{ij}(z)D_{\frac{\partial }{\partial x_j}}^E\bfX(x,u)\lambda(\|v\|^2)+\bfX(x,u)\sum_{j\in \mathcal{I}_2}a_{ij}(z)\partial_{v_j}\lambda(\|v\|^2),
\end{align}
for $z\in \mathscr{V}_{\alpha}$. 

Let us first
estimate the term $\partial_{v_j}\lambda(||v||^2)$:
\begin{equation}\label{eqlambda}
|\partial_{v_j}\lambda(||v||^2)|=2\lambda'(||v||^2)|\lan v,v_j\ran |\le 2\lambda'(||v||^2)\|v\|\big\|\tfrac{\partial}{\partial v_i}\big\|. 
\end{equation}
By \eqref{AZB} and \eqref{stp2}, the second term in \eqref{eqdzv} is bounded on $\overline{\mathscr{V}_{\alpha}}$. 

From Lemma \ref{Gauss}, we have:
\begin{equation}\label{eqdxv}
D_{\frac{\partial }{\partial x_j}}^E \bfX(x,u)=\wt D_{\frac{\partial }{\partial x_j}} \bfX(x,u)+\bfh\big(\tfrac{\partial }{\partial x_j},\bfX(x,u)\big).
\end{equation}
Since $\bfX(\cdot,u)\in \mathsf{X}^2(\wt M)$, then
\begin{equation}\label{eqdxV}
\big\|\wt D_{\frac{\partial }{\partial x_j}} \bfX(x,u)\big\|\le \max_{x\in \overline{\mathscr{U}_{\alpha}},u\in U}\big\|\wt D_{\frac{\partial }{\partial x_j}} \bfX(x,u)\big\|<\infty.
\end{equation}
The second fundamental form satisfies:
\begin{equation}\label{eqhh}
\big\|\bfh(\tfrac{\partial }{\partial x_j},\bfX(x,u))\big\|\le \max_{x\in \overline{\mathscr{U}_{\alpha}},u\in U}\big\{\|\bfh\| \big\|\tfrac{\partial}{\partial x_j}\big\| \|\bfX\|\big\}<\infty
\end{equation}
where finiteness follows from smoothness of $\bfh$, assumption on $\bfX$ and \eqref{AZB}.

Combining these results with \eqref{AZB}, we establish that \eqref{eqdzv} is bounded on $\overline{\mathscr{V}_{\alpha}}$ for all $\alpha\in\mathcal{S}$. The Cauchy-Schwarz inequality yields:
\begin{equation}\label{VG}
\max_{\|v\|=1,z\in M_{\epsilon}}\|D_{v}^{E}\obfX(z,u)\|\le \max_{\alpha\in\mathcal{S}}\max_{z\in \overline{\mathscr{V}_{\alpha}},u\in U}\sqrt{\sum_{i\in \mathcal{I}}\big\|D_{\tfrac{\partial}{\partial z_i}}^E \obfX(z,u)\big\|^2}< \infty.
\end{equation}
The extension satisfies
$$\|\obfX(z_1,u)-\obfX(z_2,u)\|\le \sup_{\|v\|=1,z\in \dbR^{n_1}}\|D_{v}^E \obfX(z,u)\| \|z_1-z_2\|= \max_{\|v\|=1,z\in M_{\epsilon}}\|D_{v}^E \obfX(z,u)\| \|z_1-z_2\|,$$
proving $\obfX$ is Lipschitz in the spatial variable.

\ss 

\textbf{Step 3.} Since 
$$\|\obfX(z,u_1)-\obfX(z,u_2)\|\le \|\obfX(x,u_1)-\obfX(x,u_2)\|\lambda(\|v\|^2)\le C_L\|u_1-u_2\|,$$
there exists bounded $C_E$ depending on $C_L,C_B,C$ such that
\begin{align}\label{CX1}
\|\obfX(z_1,u_1)-\obfX(z_2,u_2)\|&\le \|\obfX(z_1,u_1)-\obfX(z_2,u_1)\|+\|\obfX(z_2,u_1)-\obfX(z_2,u_2)\|\nonumber\\
&\le C_E\big(\|z_1-z_2\|+ \|u_1-u_2\|\big).
\end{align}

\textbf{Step 4.} In this step, following the approach in Step 2, we analyze the term
$$D_{\frac{\partial }{\partial z_k}}^E\{D_{\obfX}^E \obfX(z,u)\}.$$

For fixed $\alpha\in\mathcal{S}$, we compute in $\mathscr{V}_{\alpha}$:
\begin{align}\label{eqzvv}
&D_{\frac{\partial }{\partial z_k}}^E\Big\{\obfX_j(z,u)D_{\frac{\partial}{\partial z_j}}^E \obfX(z,u)\Big\}=\partial_{z_k}\obfX_j(z,u)D_{\frac{\partial}{\partial z_j}}^E \obfX(z,u)+\obfX_j(z,u)D_{\frac{\partial}{\partial z_k}}^ED_{\frac{\partial}{\partial z_j}}^E \obfX(z,u).
\end{align}
The first term in \eqref{eqzvv} is bounded by \eqref{eqdzv}.

By applying  the formula \eqref{eqzv} twice, we get 
\begin{align}\label{eqzzv}
D_{\frac{\partial}{\partial z_k}}^ED_{\frac{\partial}{\partial z_j}}^E \obfX(z,u)=&
\sum_{i\in \mathcal{I}_1}\partial_{z_k}a_{ji}(z)D_{\frac{\partial }{\partial x_i}}^E\bfX(x,u)\lambda(||v||^2)+
\sum_{i\in \mathcal{I}_1}a_{ji}(z)D_{\frac{\partial}{\partial z_k}}^ED_{\frac{\partial }{\partial x_i}}^E\bfX(x,u)
\lambda(||v||^2)\nonumber\\
&+\sum_{i\in \mathcal{I}_1}a_{ji}(z)D_{\frac{\partial }{\partial x_i}}^E\bfX(x,u)\partial_{z_k}\lambda(||v||^2)+
D_{\frac{\partial }{\partial z_k}}^E\bfX(x,u)\sum_{i\in \mathcal{I}_2}a_{ji}(z)\partial_{v_i}\lambda(||v||^2)\nonumber\\
&+\bfX(x,u)\sum_{i\in \mathcal{I}_2}\partial_{z_k}a_{ji}(z)\partial_{v_i} \lambda(||v||^2)+\bfX(x,u)\sum_{i\in \mathcal{I}_2}a_{ji}(z)\partial_{v_iz_k}^2 \lambda(\|v\|^2).
\end{align}
Smoothness of $a_{ij}$ ensures $\partial_{z_k}a_{ij}$ are bounded on $\overline{\mathscr{V}_{\alpha}}$ for all $i,j,k\in \mathcal{I}$. The derivatives of $\lambda$ satisfy 
\begin{equation}\label{DZL}
\partial_{z_k}\lambda(||v||^2)=\sum_{j\in \mathcal{I}_2}a_{kj}(z)\partial_{v_j}\lambda(||v||^2),\qquad \partial_{z_kv_i}^2\lambda(||v||^2)=\sum_{j\in \mathcal{I}_2}a_{kj}(z)\partial_{v_jv_i}^2\lambda(||v||^2).
\end{equation}
From \eqref{stp2} and \eqref{AZB}, we obtain the boundedness of \eqref{DZL} in $\overline{\mathscr{V}_{\alpha}}$, which consequently yields the boundedness of each term in \eqref{eqzzv} except for the second term on the right-hand side in $\overline{\mathscr{V}_{\alpha}}$.

Next, we compute the term $T:=D_{\frac{\partial}{\partial z_k}}^ED_{\frac{\partial }{\partial x_i}}^E\bfX(x,u)$ by using the formulas in Lemma \ref{Gauss} and equation \eqref{eqzv}:
\begin{align}\label{eqT}
T&=\sum_{j\in \mathcal{I}_1}a_{ij}(z)D_{\frac{\partial }{\partial x_k}}^ED_{\frac{\partial }{\partial x_j}}^E\bfX(x,u)\nonumber\\
&=\sum_{j\in \mathcal{I}_1}a_{ij}(z)\bigg\{\wt D_{\frac{\partial }{\partial x_k}}\wt D_{\frac{\partial }{\partial x_j}} \bfX+\bfA\Big(\bfh\big(\tfrac{\partial }{\partial x_j},\bfX\big),\tfrac{\partial }{\partial x_k}\Big)
+\bfh\Big(\tfrac{\partial }{\partial x_k},\wt D_{\frac{\partial }{\partial x_j}}\bfX\Big)+D_{\frac{\partial }{\partial x_k}}^{\perp}\Big(\bfh\big(\tfrac{\partial }{\partial x_j},\bfX\big)\Big)\bigg\}.
\end{align}
By using the tensor properties of $\bfA$ and $\bfh$,  the second and third terms in \eqref{eqT} satisfy
\begin{equation}\label{equA1}
\Big\|\bfA\big(\bfh\big(\tfrac{\partial }{\partial x_j},\bfX\big),\tfrac{\partial }{\partial x_k}\big)\Big\|\le \|\bfA\|\big\|\bfh\big(\tfrac{\partial }{\partial x_j},\bfX\big)\big\| \big\|\tfrac{\partial }{\partial x_k}\big\|\le \max_{x\in \overline{\mathscr{U}_{\alpha}},u\in U}\Big\{\|\bfA\|\|\bfh\|\|\bfX(x,u)\|\big\|\tfrac{\partial }{\partial x_i}\big\|\big\|\tfrac{\partial }{\partial x_k}\big\|\Big\}, 
\end{equation}
and
\begin{equation}\label{equh2}
\Big\|\bfh\Big(\tfrac{\partial }{\partial x_k},\wt D_{\frac{\partial }{\partial x_j}}\bfX\Big)\Big\|\le 
\max_{x\in \overline{\mathscr{U}_{\alpha}},u\in U}\Big\{\big\|\bfh\big\| \big\|\tfrac{\partial }{\partial x_k}\big\|\Big\| \wt D_{\tfrac{\partial }{\partial x_j}}\bfX\Big\|\Big\}.
\end{equation}
For the last term in \eqref{eqT}, by \eqref{fml0}, we have
$$D_{\frac{\partial }{\partial x_k}}^{\perp}\Big(\bfh\big(\tfrac{\partial }{\partial x_j},\bfX\big)\Big)=
\big(D\bfh\big)\Big(\tfrac{\partial }{\partial x_j},\bfX,\tfrac{\partial }{\partial x_k}\Big)+\bfh\Big(\wt D_{\frac{\partial }{\partial x_k}}\tfrac{\partial }{\partial x_j},\bfX\Big)+
\bfh\Big(\tfrac{\partial }{\partial x_i},\wt D_{\frac{\partial }{\partial x_k}}\bfX\Big),$$
which yields the estimate
\begin{align}\label{equDh}
\Big\|D_{\frac{\partial }{\partial x_k}}^{\perp}\Big(\bfh\big(\tfrac{\partial }{\partial x_j},X\big)\Big)\Big\| \le \max_{x\in \overline{\mathscr{U}_{\alpha}},u\in U}\bigg\{&\|D\bfh\|\big\|\tfrac{\partial }{\partial x_j}\big\|\big\|\tfrac{\partial }{\partial x_k}\big\|\|\bfX\|+\sum_{i\in \mathcal{I}_1}\Gamma_{jk}^i\|\bfh\|\|\tfrac{\partial }{\partial x_i}\|\|\bfX\|\nonumber\\
&+\|\bfh\|\|\tfrac{\partial }{\partial x_j}\|\Big\|\wt D_{\tfrac{\partial }{\partial x_k}}\bfX\Big\|\bigg\},
\end{align}
where $\Gamma_{ij}^k(i,j,k\in \mathcal{I})$ are the corresponding Christoffel symbols.
Combining \eqref{AZB}, \eqref{VG}, and estimates \eqref{equA1}-\eqref{equDh} with the assumption $\bfX(\cdot,u)\in\mathsf{X}^2(\wt M)$, we conclude that each term in \eqref{eqzzv} is bounded on $\overline{\mathscr{V}_{\alpha}}$.  Consequently, \eqref{eqzvv} is bounded on $\overline{\mathscr{V}_{\alpha}}$.  Similarly to inequality \eqref{VG}, we have
\begin{equation}\label{DXXz}
\|D_{\obfX}^E \obfX(z_1,u)-D_{\obfX}^E \obfX(z_2,u)\|\le\max_{\alpha \in \mathcal{S}}\max_{z\in\overline{\mathscr{V}_{\alpha}},u\in U}\sqrt{\sum_{i\in \mathcal{I}}\Big\|D_{\frac{\partial}{\partial z_i}}^ED_{\obfX}^E \obfX(z,u)\Big\|^2}\|z_1-z_2\|.
\end{equation}

\textbf{Step 5.} In this final step, we investigate the Lipschiz properties of the vector field $D_{\obfX}^E \obfX$ with respect to parameter $u$ to complete our proof. 

We first proceed to compute the vector field $D_{\obfX}^E \obfX$ in the neighborhood $\mathscr{V}_{\alpha}$:
\begin{align}\label{eqzzv1}
D_{\obfX}^E \obfX(z,u)&=\sum_{i\in \mathcal{I}}\obfX_i(z,u)D_{\frac{\partial}{\partial z_i}}^E \obfX(z,u)\nonumber\\
&=\sum_{i\in \mathcal{I}}\sum_{j\in \mathcal{I}_1}a_{ij}(z)\bfX_i(z,u)\wt D_{\frac{\partial }{\partial x_j}}\bfX(x,u)\lambda^2(||v||^2)\nonumber\\
&\quad+\bfX(x,u)\sum_{i\in \mathcal{I}}\sum_{j\in \mathcal{I}_2}a_{ij}(z)\bfX_i(x,u)\lambda(||v||^2)\partial_{v_j}\lambda(||v||^2).
\end{align}
Let 
$$ C_{a}:=\max_{\alpha\in \mathcal{S},i,j\in \mathcal{I}}\max_{z\in \overline{\mathscr{V}_{\alpha}}}|a_{ij}(z)|,\quad C_{x}:=\max_{x\in \wt M,u\in U}\|\bfX(x,u)\|+\max_{x\in \wt M,u\in U,\|v\|= 1}\big\|\wt D_{v}\bfX(x,u)\big\|.$$
Applying the expression from \eqref{eqzzv1}, we obtain the following bound:
\begin{equation}\label{DXXu}
\|D_{\obfX}^E \obfX(z,u_1)-D_{\obfX}^E \obfX(z,u_2)\|\le T_1+T_2
\end{equation}
for
\begin{align}\label{T2T3}
T_1&:=\sum_{i\in \mathcal{I}}\sum_{j\in \mathcal{I}_1}a_{ij}(z)\big\|\bfX_i(x,u_1)\wt D_{\frac{\partial }{\partial x_j}}\bfX(x,u_1)-\bfX_i(x,u_2)\wt D_{\frac{\partial }{\partial x_j}}\bfX(x,u_2)\big\|\lambda^2(||v||^2)\nonumber\\
&\le n_1 n C_LC_aC_x^2\big\|\tfrac{\partial}{\partial x_j}\big\|\|u_1-u_2\|,\nonumber\\
T_2&:=\sum_{i\in \mathcal{I}}\sum_{j\in \mathcal{I}_2}a_{ij}(z)\big\|\bfX(x,u_1)\bfX_i(x,u_1)-\bfX(x,u_2)\bfX_i(x,u_2)\big\|\lambda(||v||^2)\partial_{v_j}\lambda(||v||^2)\nonumber\\
&\le n_1(n_1-n) C_LC_aC_x^2\frac{C}{\epsilon}\|u_1-u_2\|.
\end{align}
From the estimates in \eqref{AZB}, \eqref{T2T3}, and \eqref{DXXu}, we conclude that the mapping $D_{\obfX}^E \obfX(z,\cdot)$ is Lipschitz continuous. Moreover, using \eqref{DXXz}, there exists a constant $C_{E} = C_E(C_B,C_L,C_a,C_x)$ such that
\begin{align*}
\|D_{\obfX}^E \obfX(z_1,u_1)-D_{\obfX}^E \obfX(z_2,u_2)\|&\le \|D_{\obfX}^E \obfX(z_1,u_1)-D_{\obfX}^E \obfX(z_2,u_1)\|+ \|D_{\obfX}^E \obfX(z_2,u_1)-D_{\obfX}^E \obfX(z_2,u_2)\|\nonumber\\
&\le C\|z_1-z_2\|+\|D_{\obfX}^E \obfX(z_2,u_1)-D_{\obfX}^E \obfX(z_2,u_2)\|.\nonumber\\
&\le C_{E}\big(\|z_1-z_2\|+\|u_1-u_2\|\big).
\end{align*}
The desired inequality is therefore established by combining this result with \eqref{CX1}. 
\end{proof}

For clarity and to aid understanding of the extension process, we now give specific examples of vector field extensions as defined in \eqref{stp2}.

\begin{example}
Consider the subspace $\wt M :=\{(z_1,0,\ldots,0) \mid z_1 \in \dbR\} \subset \dbR^n$ and let $\bfX \in \mathsf{X}(\wt M)$ be a vector field on $\wt M$. For any point $z = (z_1,\ldots,z_n) \in \dbR^{n}$, we define the decomposition:
$$x=(z_1,0,...),\quad v=(0,z_2,....,z_n).$$
The extension $\obfX$ of $\bfX$ to $\dbR^n$, supported on $\wt M_{\epsilon}$ (the $\epsilon$-tubular neighborhood of $\wt M$, see in \eqref{meps}), is then given by:
$$\obfX(z)=\bfX(x)\lambda(\|v\|).$$
for $\lambda$ defined in \eqref{stp2}.
\end{example}
\begin{example}
Consider the unit circle $\wt M := S^1 \subset \dbR^2$ and let $\bfX \in \mathsf{X}(\wt M)$ be a smooth vector field on $\wt M$. For any point $z = (z_1, z_2) \in \dbR^2 \setminus {0}$, we define the projection onto the unit circle and radial deviation respectively as:
$$x=\frac{z}{\|z\|},\quad v=\|z\|-1,$$
The extension $\obfX$ of $\bfX$ to $\dbR^2$,supported on $\wt M_{\epsilon}$, is then given by 
$$\obfX(z)=\bfX(x)\lambda(\|v\|).$$
for $\lambda$ defined in \eqref{stp2}. 
\end{example}

The following lemma is employed in establishing the well-posedness of the equations \eqref{eqve0} and \eqref{eqve11}.

\begin{lemma} \label{lmf}
Let $\varphi_{\wt M}(z) := \dist^2(z, \wt M)$ denote the squared distance function from a point $z \in \wt M_{\epsilon} \subset \dbR^{n_1}$ to the submanifold $\wt M$. Given vector fields $\bfX, \bfY \in \mathsf{X}^1(\wt M)$ that have been extended to $\dbR^{n_1}$ (still denoted by $\bfX, \bfY$) via the construction in \eqref{lm62}, there exists a positive constant $C > 0$ such that the following estimates hold for all $z \in \wt M_{\epsilon}$:

1. First-order  estimate:
$|\bfX(z)\varphi_{\wt M}(z)|\le C\varphi_{\wt M}(z)$, for all $z\in M_{\epsilon}$. 

2. Second-order estimate: $|\bfY(z)\bfX(z)\varphi_{\wt M}(z)|\le C\varphi_{\wt M}(z)$, for all  $z\in M_{\epsilon}$.
\end{lemma}

\begin{proof} Using formula \eqref{eqzv} and adopting the same notations $\mathscr{V}_{\alpha}$, $\mathcal{S}$ as in Lemma \ref{lm62}, we  fix $\alpha$ and express the vector field as
$$\bfX(z)=\sum_{i\in \mathcal{I}}\bfX_i\frac{\partial}{\partial z_i}=\sum_{i\in \mathcal{I}}\bfX_i\bigg\{\sum_{j\in\mathcal{I}_1}a_{ij}(z)\frac{\partial}{\partial x_j}+\sum_{j\in\mathcal{I}_2}a_{ij}(z)\frac{\partial}{\partial v_j}\bigg\},\quad z\in \mathscr{V}_{\alpha}.$$
Defining $\bar{\bfX}_j(z) := \sum_{i\in \mathcal{I}}\bfX_i(z)a_{ij}(z)$ for $j\in \mathcal{I}$, we obtain the decomposition:
\begin{equation}\label{Vxv}
\bfX(z)=\sum_{j\in\mathcal{I}_1}\bar{\bfX}_j(z)\frac{\partial}{\partial x_j}+\sum_{j\in\mathcal{I}_2}\bar{\bfX}_j(z)\frac{\partial}{\partial v_j},\quad z\in \mathscr{V}_{\alpha}
\end{equation}
with the property that
$$\bar{\bfX}_j(z)=0,\quad \forall j\in\mathcal{I}_2,z\in \wt M.$$ 

Noting that $a_{ij}$ are smooth in $\overline{\mathscr{V}_{\alpha}}$ for $i,j\in\mathcal{I}$, we conclude that $\bar{\bfX_j}\in C^1(\overline{\mathscr{V}_{\alpha}})$.
This yields the estimate:
\begin{equation}\label{nbX}
|\bar{\bfX}_j(z)|=|\bar{\bfX}_j(x+v)-\bar{\bfX}_j(x)|\le \max_{z\in \overline{\mathscr{V}}_{\alpha}}\|\nabla\bar{\bfX}_j\|\|v\|,\quad j\in \mathcal{I}_2,z\in \overline{\mathscr{V}}_{\alpha}.
\end{equation}
In the local coordinates $(\mathscr{V}_{\alpha},\psi_{\alpha}^{-1})$, the squared distance function satisfies $\varphi_{\wt M}(z) = \|v\|^2$. Consequently,
\begin{equation*}
\bfX(z)\varphi_{\wt M}=\sum_{j\in \mathcal{I}_2}\bar{\bfX}_j\partial_{v_j}||v||^2=2\sum_{j\in \mathcal{I}_2}\bar{\bfX}_j(z)\lan v,v_j\ran .
\end{equation*}
Applying \eqref{nbX} and the Cauchy-Schwarz inequality, we obtain the bound:
\begin{equation*}
|\bfX(z)\varphi_{\wt M}(z)|\le 2\sum_{j\in \mathcal{I}_2}\max_{\alpha\in \mathcal{S}} \max_{z\in \overline{\mathscr{V}}_{\alpha}}\|\nabla\bar{\bfX}_j\|\|v\|^2 \|v_j\|\le C\|v\|^2 
\end{equation*}
for some constant $C > 0$.

Similarly, we compute the second-order term in the local coordinates $(\mathscr{V}_{\alpha},\psi_{\alpha}^{-1})$:
\begin{align}\label{YX}
\bfY(z)\{\bfX(z)\varphi_{\wt M}(z)\}&=2\sum_{j\in \mathcal{I}_1}\sum_{i\in \mathcal{I}_2}\bar{\bfY}_j\tfrac{\partial}{\partial x_j}\big\{\bar{\bfX}_i\lan v,v_i\ran \big\}+2\sum_{j\in \mathcal{I}_2}\sum_{i\in \mathcal{I}_2}\bar{\bfY}_j\tfrac{\partial}{\partial v_j}\big\{\bar{\bfX}_i\lan v,v_i\ran \big\}\nonumber\\
&=2\sum_{j\in \mathcal{I}_1}\sum_{i\in \mathcal{I}_2}\bar{\bfY}_j\tfrac{\partial \bar{\bfX}_i}{\partial x_j}\lan v,v_i\ran +2\sum_{j\in \mathcal{I}_2}\sum_{i\in \mathcal{I}_2}\bar{\bfY}_j\tfrac{\partial \bar{\bfX}_i}{\partial v_j}\lan v,v_i\ran +2\sum_{j\in \mathcal{I}_2}\sum_{i\in \mathcal{I}_2}\bar{\bfY}_j\bar{\bfX}_i\lan v_j,v_i\ran.
\end{align}
Analyzing the order of $\|v\|$ in \eqref{YX}, we establish the existence of $C>0$ such that
$$|\bfY(z)\bfX(z)\varphi_{\wt M}(z)|\le C\|v\|^2.$$
which completes the proof.
\end{proof}

On the submanifold $\wt M$, we consider two distinct distance functions: the Euclidean distance induced from the ambient space $\dbR^{n_1}$, and the intrinsic distance $\wt\rho$ defined as the infimum of path lengths measured along $\wt M$. The following lemma, which constitutes the second key result of this subsection, establishes a fundamental relationship between these two distance metrics. This connection plays a pivotal role in rigorously demonstrating the continuous dependence of solutions to equation \eqref{eqve0} on both controls and initial data.

\begin{lemma} \label{lm61}
Let $(\wt M,\wt g)$ be an $n$-dimensional compact embedded submanifold of $\dbR^{n_1}$, then there exists a  constant $C_{\rho}\ge 1$, such that for any $x,y\in \wt M$,
\begin{equation}\label{apeq}
\|x-y\|\le \wt \rho(x,y)\le  C_{\rho}\|x-y\|.
\end{equation}
\end{lemma}

\begin{proof}
The left-hand inequality in \eqref{apeq} is immediate. For the right-hand inequality, we demonstrate that the function 
$$\zeta(x,y):=\frac{\wt\rho(x,y)}{\|x-y\|},$$
admits a continuous extension to $\wt M \times \wt M$. The compactness of $\wt M$ then ensures the boundedness of $\zeta$. Setting $C_{\rho} := \max_{x,y \in \wt M} \zeta(x,y)$ yields the desired inequality \eqref{apeq}.

Since $\wt M$ is an embedded submanifold of $\dbR^{n_1}$, the inclusion map $\iota: \wt M \rightarrow \dbR^{n_1}$ is injective. Consequently, for any $x, y \in \wt M$, $\|x - y\| = 0$ implies $\wt\rho(x, y) = 0$. Thus, the only singularities of $\zeta$ occur on the diagonal $\{(z, z) \mid z \in \wt M\}$. 

We now establish that for every $z \in \wt M$ and $\epsilon > 0$, there exists $\delta > 0$ such that whenever $\wt\rho(x, z) + \wt\rho(y, z) \le \delta$, there exists a positive constant $c_z$ for which 
$$|\zeta(x,y)-c_z|\le \epsilon.$$
Defining $\zeta(z, z) := c_z$ ensures continuity by construction.

The proof proceeds in three steps for a fixed $z \in \wt M$:

\textbf{Step 1:  Local Coordinates and Geometric Setup}. Identify the tangent space $T_z \wt M$ with $\dbR^n$ and consider the geodesic ball $\mathscr{U}_z:=\wt B_z^{\wt\rho}(\frac{1}{2}\bfi_{\wt M})$ centered at $z$. Let $(\mathscr{U}_z, \psi_z)$ denote normal coordinates centered at $z$, where
\begin{equation}\label{equu}
\psi_z:\mathscr{U}_z\rightarrow B_{n}(\tfrac{1}{2}\bfi_{\wt M}),\quad\psi_z(x):=\exp_{z}^{-1}(x),
\end{equation}
with inverse $\phi_z := \psi_z^{-1}$ given by
\begin{equation}\label{eqphi}
\phi_z: B_{n}\big(\tfrac{1}{2}\bfi_{\wt M}\big)\rightarrow\mathscr{U}_z,\quad\phi_z(\bfu):=\exp_{z}(\bfu).
\end{equation}
(Note the distinction between the control variable $u$ and the parameter $\mathbf{u}$.)

In these coordinates, the inclusion map $\iota$ (see Definition \ref{dfemb}) has component functions $x_1, \dots, x_{n_1}$ expressible as:
\begin{align*}
x_{1}=x^1\circ\iota\circ\phi_z(\bfu_1,\dots,\bfu_n),\quad\dots,\quad x_{n_1}=x^{n_1}\circ\iota\circ\phi_z(\bfu_1,\dots,\bfu_{n}).
\end{align*}
The smoothness of $\iota$ ensures the smoothness of each $x_i$. Define
$$\varphi:=(x_1,x_2,...,x_{n_1}),\quad \varphi_j=\frac{\partial\varphi}{\partial \bfu_j},\quad \varphi_{ij}=\frac{\partial^2\varphi}{\partial \bfu_i\partial \bfu_j},\quad \varphi_{ijk}=\frac{\partial^3\varphi}{\partial \bfu_i\partial \bfu_j\partial \bfu_k}.$$ 
The basis of tangent space and metric components  in local coordinate $(\mathscr{U}_z,\psi_z)$  are given by
$$
\varphi_i=\frac{\partial}{\partial \bfu_i}=\sum_{j\in \mathcal{I}}\frac{\partial x_j}{\partial \bfu_i}\frac{\partial }{\partial x_{j}},\quad \wt g_{ij}=\lan \frac{\partial}{\partial \bfu_i},\frac{\partial}{\partial \bfu_j}\ran =
\lan \varphi_i,\varphi_j\ran,\quad i,j\in \mathcal{I}_1,$$
By Lemma \ref{lmn}, there exists a smooth orthonormal normal frame $\{\mathbf{n}^\alpha\}_{\alpha \in \mathcal{I}_2}$ on $\mathsf{N}(\mathscr{U}_z)$. The Gauss-Weingarten equations in components yield 
\begin{equation}\label{GW}
\varphi_{ij}=\sum_{s\in \mathcal{I}_1}\Gamma_{ij}^s \varphi_{s}+\sum_{\alpha\in \mathcal{I}_2} H_{ij}^{\alpha} \bfn^{\alpha},\quad \frac{\partial \bfn^{\alpha}}{\partial \bfu_j}=-\sum_{i,s\in \mathcal{I}_1}\wt g^{is}H_{ij}^{\alpha}\varphi_{s}+\sum_{\beta\in \mathcal{I}_2}N_{\alpha j}^{\beta}\bfn^{\beta},
\end{equation}
where $\Gamma_{ij}^k$, $H_{ij}^\alpha$, and $N_{\alpha j}^\beta$ are smooth in $\mathbf{u} \in B_n(\frac{1}{2} \bfi_{\wt M})$.

Differentiating once more, we obtain the third derivatives:
\begin{align*}
\varphi_{ijk}&= \sum_{s\in \mathcal{I}_1}\Big\{\frac{\partial \Gamma_{ij}^s}{\partial \bfu_k} \varphi_{s}+\Gamma_{ij}^s \varphi_{sk}\Big\}+\sum_{\alpha\in \mathcal{I}_2}\Big\{\frac{\partial H_{ij}^{\alpha}}{\partial \bfu_k} \bfn^{\alpha}+H_{ij}^{\alpha}\frac{\partial \bfn^{\alpha}}{\partial \bfu_k}\Big\}\nonumber\\
&=\sum_{s\in \mathcal{I}_1}\bigg\{\frac{\partial \Gamma_{ij}^s}{\partial \bfu_k}+\sum_{r\in \mathcal{I}_1}\Gamma_{ij}^r\Gamma_{rk}^s-\sum_{r\in \mathcal{I}_1,\beta\in \mathcal{I}_2}H_{ij}^{\beta}\wt g^{rs}H_{rk}^{\beta}\bigg\}\varphi_{s}\nonumber\\
&\quad+\sum_{\alpha\in \mathcal{I}_2}\bigg\{\frac{\partial H_{ij}^{\alpha}}{\partial \bfu_k}+\sum_{\beta\in \mathcal{I}_2}H_{ij}^{\beta}N_{\beta k}^{\alpha}+\sum_{s\in \mathcal{I}_1}\Gamma_{ij}^sH_{sk}^{\alpha}\bigg\} \bfn^{\alpha},
\end{align*}
for $i,j,k\in \mathcal{I}_1$.

\ss

\textbf{Step 2: Geodesic Analysis and Distance Comparison}. For $x,y\in \mathscr{U}_{z}$, the Riemannian distance satisfies $l_{xy} := \wt\rho(x, y) \le \wt\rho(z, x) + \wt\rho(z, y) < \bfi_{\wt M}$. The inverse exponential map $\exp_x^{-1}(y)$ is well-defined, and there exists a unique unit-speed geodesic $\gamma: [0, l_{xy}] \to \mathscr{U}_z$ connecting $x$ and $y$, with $\gamma'(0)=\frac{\exp_x^{-1}(y)}{l_{xy}}\in T_x\wt M$.

In coordinates $(\mathscr{U}_z,\psi_z)$, 
$$\wt \gamma(t):=(\wt\gamma_1(t),\dots,\wt \gamma_n(t))):=\psi_z(\gamma(t))\in B_n(\bfi_{\wt M}),\quad t\in [0,l_{xy}],$$ 
satisfies the geodesic equations:
\begin{equation}\label{equk}
\wt\gamma_k''(t)+\Gamma_{ij}^k(\gamma(t))\wt\gamma_i'(t)\wt\gamma_j'(t)=0,\quad \wt\gamma_k(0)=[\psi_z(x)]_k,\quad \wt\gamma_k'(0)=v_k,\quad i,j,k\in \mathcal{I}_1.
\end{equation}
where $v_k$ is the $k$-th components of $\gamma'(0)$ in the basis $\{\tfrac{\partial}{\partial \bfu_k}\big|_{x}\}_{k\in \mathcal{I}_1}$.

Let $\mathbf{x}(t) := \varphi(\wt\gamma(t))$ be the Euclidean representation of $\gamma$. Differentiating $\mathbf{x}(t)$ yields:
\begin{equation}\label{eqxk}
\begin{cases}
\ds\bfx'(t)=\sum_{i\in \mathcal{I}_1}\varphi_{i}\wt\gamma_i'(t),\\
\ns\ds \bfx''(t)=\sum_{i,j\in\mathcal{I}_1}(\sum_{\alpha\in\mathcal{I}_2}H_{ij}^{\alpha} \bfn^{\alpha}+\sum_{k\in\mathcal{I}_1}\Gamma_{ij}^k\varphi_k)\wt\gamma_i'(t)\wt\gamma_j'(t), \\
\ns\ds
\lan \bfx'(t),\bfx''(t)\ran =2\sum_{i,j,l,s\in \mathcal{I}_1}\Gamma_{ij}^s\gamma_i'(t)\wt\gamma_j'(t)\wt\gamma_l'(t)\wt g_{sl},
\\
\ns\ds\bfx'''(t)=\sum_{i,j,k\in \mathcal{I}_1}\varphi_{ijk}\wt\gamma_i'(t)\wt\gamma_j'(t)\wt\gamma_k'(t)+3\sum_{i,j\in \mathcal{I}_1}\varphi_{ij}\wt\gamma_i''(t)\wt\gamma_j'(t)+\sum_{i\in \mathcal{I}_1}\varphi_i \wt\gamma_i'''(t).
\end{cases}
\end{equation}
By the equation \eqref{equk}, we get that
\begin{align}\label{eqtxu33}
\wt\gamma_k'''(t)=&-\sum_{i,j,l\in \mathcal{I}_1}\frac{\partial\Gamma_{ij}^k}{\partial \bfu_l}(t)\wt\gamma_i'(t)\wt\gamma_j'(t)\wt\gamma_l'(t)+\sum_{i,j,l,s\in \mathcal{I}_1}\Gamma_{ij}^k \wt\gamma_j'(t)\Gamma_{ls}^i \wt\gamma_l'(t)\wt\gamma_s'(t)\nonumber\\
&+\sum_{i,j,l,s\in \mathcal{I}_1}\Gamma_{ij}^k \wt\gamma_i'(t)\Gamma_{ls}^j \wt\gamma_l'(t)\wt\gamma_s'(t),
\end{align}
Since the  normal coordinates $(\mathscr{U}_z,\psi_z)$ and $\{\bfn^{\alpha}\}_{\alpha\in \mathcal{I}_2}$ are fixed, then quantities $\wt g,\Gamma,H,N$ are fixed and smooth with respect to $\bfu_i,i\in \mathcal{I}_1$(Here, $\Gamma$ is short for $\{\Gamma_{ij}^k\}_{i,j,k\in\mathcal{I}_1}$; analogous usage apply to other quantities). From \eqref{eqxk}, we obtain that
\begin{equation}\label{eqx1x2}
|\lan \bfx'(t),\bfx''(t)\ran |\le C\max_{t\in [0,l_{xy}]}\|\wt\gamma'(t)\|^3,\quad t\in [0,l_{xy}],
\end{equation}
where positive $C$ depends on the values of $\wt g$ and $\Gamma$ in $\mathscr{U}_z$.
Similarly, by using \eqref{eqtxu33}, we have that
\begin{equation}\label{eqx3}
\|\bfx'''(t)\|\le C\max_{t\in [0,l_{xy}]}\|\wt\gamma'(t)\|^3, \quad t\in [0,l_{xy}],
\end{equation}
where positive constant $C$ depends on the values of $\wt g,H,H',\Gamma,\Gamma',N$ in $\mathscr{U}_z$, for $H',\Gamma'$ denoting the derivatives of $H,\Gamma$ with respect to $\bfu_i(i\in \mathcal{I}_1)$ without ambiguity. 

Since $\gamma$ is parameterized by arc-length, the Riemannian distance between $\bfx(t)$ and $x$ satisfies
$$\wt\rho(\bfx(t),x)=t, \quad t\in [0,l_{xy}]. $$
We now examine the corresponding Euclidean distance. Define the squared distance function:
$$F(t):= \|\bfx(t)-\bfx(0)\|^2=\sum_{i\in \mathcal{I}}(\bfx_i(t)-\bfx_i(0))^2,\quad t\in [0,l_{xy}].$$
Direct computation of derivatives yields that
\begin{equation}\label{eqfx}
F'(0)=0,\quad F''(0)=2\|\gamma'(0)\|^2=2,
\end{equation}
and that
\begin{equation}\label{eqfx2}
F'''(t)=2\lan \bfx'''(t),\bfx(t)-\bfx(0)\ran +6\lan \bfx'(t),\bfx''(t)\ran ,\quad t\in [0,l_{xy}].
\end{equation}
Applying Cauchy-Schwarz inequality in conjunction with inequalities \eqref{eqx1x2}, \eqref{eqx3}, and the bound
$$\|\bfx(t)-\bfx(0)\|\le\wt \rho(\bfx(t),\bfx(0))\le \bfi_{\wt M},\quad t\in [0,l_{xy}],$$
we obtain the uniform estimate:
$$|F'''(t)|\le C_0 \max_{t\in [0,l_{xy}]}\|\wt\gamma'(t)\|^3,\quad t\in [0,l_{xy}].$$
where the positive constant $C_0$ depending on values of $\wt g,H,\Gamma,N,H',\Gamma'$ in $\mathscr{U}_z$ and $\bfi_{\wt M}$.

Consequently, Taylor expansion gives
\begin{equation}\label{rhoxyt}
\big|\|\bfx(t)-\bfx(0)\|^2 -t^2\big|\le\frac{1}{6}\max_{t\in [0,l_{xy}]}|F'''(t)|\|\wt\gamma'(t)\|^3 t^3\le C_0 \|\wt\gamma'(t)\|^3t^3,\quad t\in [0,l_{xy}].
\end{equation}
The Riemannian metric $\wt g$ is non-degenerate and bounded in $\mathscr{U}_z$, and for any arc-length parameterized geodesic $\bm{\g}$(bold, distinguished from $\gamma$) in $\mathscr{U}_z$ and the corresponding coordinate represention $\widetilde{\bm{\g}}'$ in $(\mathscr{U}_z,\psi_z)$ , we have 
$$\|\bm{\g}'\|^2=\wt g_{ij}(\bm{\g})\widetilde{\bm{\g}}_i'\widetilde{\bm{\g}}_j'=1.$$ 
By the continuity of $\wt g$, there exists a constant $ G_1 \ge 1 $, depending only on $\wt g$ over $\mathscr{U}_z$(uniform in $\bm{\g}$), such that
\begin{equation}\label{mbfu}
\frac{1}{G_1} \le \|\widetilde{\bm{\g}}'\| \le G_1.
\end{equation}
Combining \eqref{rhoxyt} and \eqref{mbfu} at $t=l_{xy}$ yields
$$\big|\|x-y\|^2-\wt\rho^2(x,y)\big|\le C_0G_1^3\wt\rho^3(x,y),$$
which implies the key estimate of $1/\zeta$:
\begin{equation}\label{zetaf}
\frac{\|x-y\|}{\wt\rho(x,y)}\in \Big[\sqrt{1- C_0G_1^3\wt\rho(x,y)},\quad \sqrt{1+C_0G_1^3\wt\rho(x,y)}\Big],\quad x,y\in \mathscr{U}_z.
\end{equation}

\textbf{Step 3: Continuity of $\zeta$}. For $\epsilon > 0$ sufficiently small, set $\delta := \epsilon / (C_0 G_1^3)$. If $\wt\rho(x, z) + \wt\rho(y, z) \le \delta$, then \eqref{zetaf} and the inequality $|1/\sqrt{1 - x} - 1| \le x$ for $0 < x < \frac{1}{3}$ yield 
\begin{align}
\bigg|\frac{\wt\rho(x,y)}{\|x-y\|}-1\bigg|&\le \bigg|\frac{1}{\sqrt{1-C_0G_1^3\wt\rho(x,y)}}-1\bigg|\le C_0G_1^3\wt\rho(x,y)\nonumber\\
&\le C_0G_1^3\{\wt\rho(x,z)+\wt\rho(y,z)\}\le \epsilon.
\end{align}
Defining $\zeta(z, z) := 1$ ensures continuity at $z$.

Since $z \in \wt M$ was arbitrary, $\zeta$ is continuous on $\wt M \times \wt M$, completing the proof. 
\end{proof}

\subsection{Estimates related to $D_{(x,y)}^2\rho^2(x,y)$ on $M$}\label{IE}

In establishing the uniqueness and continuous dependence properties of viscosity solutions for equation \eqref{phjb}, a crucial step involves the careful selection of the test function $\varphi$ as the squared-distance-related function in Lemma \ref{pmaxp}. This approach necessitates precise estimates for the operators $A$ and $A^2$, where
\begin{equation}\label{scrA}
A = D_{(x,y)}^2 \rho^2(x, y),
\end{equation}
represents the Hessian of the squared distance function $\rho^2(x,y)$ on the manifold $M$. The fundamental purpose of Lemma \ref{mlemma1} and Lemma \ref{SecondOrderLemma} is to provide rigorous derivations of these essential operator estimates.

First, we introduce the notion of Jacobi fields. Let $\gamma:[0,l]\rightarrow M$ be a geodesic on $M$, and let $J$ be a smooth vector field along $\gamma$. If $J$ satisfies the Jacobi equation
$$J''(t)=\mathcal{R}(\gamma'(t),J(t))\gamma'(t),\quad t\in (0,l)$$
then $J$ is called a Jacobi field along $\gamma$. 

To facilitate the presentation of Lemma \ref{Jaco} concerning Jacobi fields, we first state the following standard result:


%
\begin{lemma}\nmycite{doCarmo}{p.114} 
Let $\gamma(0) = x$, $\gamma'(0) = w$, and let $J$ be a Jacobi field along $\gamma$ satisfying $J(0) = 0$ and $J'(0) = v$. Then $J$ can be expressed as $J(t)=(\exp_x)_{*tw}(tv),t\in [0,l]$.
\end{lemma}
\begin{lemma}\label{Jaco}
Let $w \in T_x M$ with $\|w\|_g = 1$ and choose $l \leq \min\left\{\frac{1}{2}\bfi_M,\sqrt{\frac{1}{2C_1}}, 1\right\}$. Consider the geodesic $\gamma(t) = \exp_x(wt)$ for $t \in [0,l]$. For any tangent vector $v \in T_x M$, we define the associated Jacobi field along $\gamma$ as:
\begin{equation*}
J(t):=(\exp_x)_{*tw}(tv),\quad t\in [0,l], 
\end{equation*}
and its derivative field
\begin{equation*}
J_1(t):=(\exp_x)_{*tw}(v),\quad t\in [0,l].
\end{equation*}
Moreover, we have:
\begin{equation}\label{eq35}
\frac{1}{2}\|v\|_g \leq \|J_1(t)\|_g \leq \frac{3}{2}\|v\|_g, \quad \forall t \in [0,l],
\end{equation}
where the constant $C_1$ depends only on the curvature bound $C_R$ and the injectivity radius $\bfi_M$. 
\end{lemma}
\begin{proof}
Let $t$ be the arc-length parameter of $\gamma$, and let $\{e_i\}_{i=1}^n$ be an orthonormal parallel frame field along $\gamma$ satisfying
\begin{equation*}
e_n(t)=\gamma'(t),\quad D_{\gamma'}e_i(t)=0.
\end{equation*}
In this frame, the Jacobi field decomposes as $J(t) = \sum_{i=1}^n J^i(t)e_i(t)$, with covariant derivatives
\begin{equation*}
J'=D_{\gamma'}J=\sum_{i=1}^n{J^i}'(t)e_i(t),\quad J''=D_{\gamma'}J'=\sum_{i=1}^n{J^i}''(t)e_i(t).
\end{equation*}
The Jacobi equation $J'' = \mathcal{R}(\gamma,J)\gamma$ with initial conditions $J(0)=0$, $J'(0)=v$ yields the component system:
\begin{equation}\label{ui}
{J^i}''=\sum_{j=1}^n J^jR_{n in j},\quad i=1,..,n,\quad J^i(0)=0,\quad {J^i}'(0)=v_i:=\lan v,e^i\ran _g,
\end{equation}
where $R_{ninj}=\lan \mathcal{R}(e_n,e_j)e_n,e_j\ran $. 

Define $\zeta(t) := \langle J(t),J(t)\rangle_g$. Using properties of the Levi-Civita connection:
$$D_{\gamma'(t)}\lan X,Y\ran _g=\lan D_{\gamma'(t)}X,Y\ran _g+\lan X,D_{\gamma'(t)}Y\ran _g,\quad\text{for}\quad X,Y\in \mathsf{X}^1(M),$$
we obtain that
\begin{align*}
&\zeta'(t)=2\lan J'(t),J(t)\ran _g, \nonumber\\
&\zeta''(t)=2\lan J''(t),J(t)\ran _g+2\lan J'(t),J'(t)\ran _g, \nonumber\\
&\zeta'''(t)=6\lan J'',J'\ran _g+2\lan J''',J\ran _g=6\lan \mathcal{R}(\gamma,J)\gamma,J'\ran _g+ 2\lan J''',J\ran _g,\nonumber\\
&\zeta^{(4)}(t)=8\lan J''',J'\ran _g+6\lan J'',J''\ran _g+2\langle J^{(4)},J\rangle_g.
\end{align*}
Initial conditions at $t=0$ give
\begin{equation*}
\zeta(0)=0,\quad \zeta'(0)=0,\quad \zeta''(0)=2\|v\|_g^2,\quad \zeta'''(0)=0.
\end{equation*}
By Taylor's formula, for any $t\in [0,l]$, there exists $\theta\in (0,t)$ such that
\begin{equation}\label{eq313}
\zeta(t)=||v||^2t^2+\frac{1}{4!}\zeta^{(4)}(\theta)t^4.
\end{equation}
The Jacobi equation $J''=\mathcal{R}(\gamma',J)\gamma'$ can be written in matrix form: 
$$D_{\gamma'}\begin{pmatrix}
J\\
J'
\end{pmatrix}
= \mathbb{M}
\begin{pmatrix}
J\\
J'
\end{pmatrix},
\quad \text{for} \quad
\mathbb{M}=\begin{pmatrix}
0&1\\
\mathcal{R}(\gamma',\cdot)\gamma'&0
\end{pmatrix}.
$$
Using the curvature bound from \eqref{eqrr}, we find that
$$
\|\mathbb{M}(J,J')^{\top}\|_g=\sqrt{\|J'\|_g^2+\|\mathcal{R}(\gamma',J)\gamma'\|_g^2}\le \sqrt{\|J'\|_g^2+C_R^2\|J\|_g^2}\le \sqrt{1+C_R^2}\|(J,J')^{\top}\|_g.
$$
Gronwall's inequality yields for $l\leq 1$: 
\begin{equation*}
\|(J,J')^{\top}\|_g\le \exp(\|\mathbb{M}\|t)\|(J(0),J'(0))^{\top}\|_g\le\exp\Big(l\sqrt{1+C_R^2}\Big)\|v\|_g\le \exp\Big(\sqrt{1+C_R^2}\Big)\|v\|_g.
\end{equation*}
By the equation \eqref{ui}, we get that
\begin{equation}\label{eq317}
|J^j|=|\lan J,e_j\ran _g|\le \|J\|_g\le \exp\Big(\sqrt{1+C_R^2}\Big)\|v\|_g,\quad |{J^j}'|=|\lan J',e_j\ran _g|\le \|J\|_g\le \exp\Big(\sqrt{1+C_R^2}\Big)\|v\|_g.
\end{equation}
Higher derivatives satisfy
\begin{align*}
&{J^i}'''=\sum_{j=1}^n\Big({J^j}'R_{n i n j}+J^jR_{n i n j}'\Big),\nonumber\\
&{J^i}^{(4)}=\sum_{j=1}^n\Big({J^j}''R_{n i n j}+2{J^j}'R_{n i n j}'+J^jR_{n i n j}''\Big),
\end{align*}
which leads to
\begin{align}\label{eq319}
\zeta^{(4)}=&\sum_{i,j=1}^n 8\Big({J^i}'{J^j}'R_{n i n j}+{J^i}'J^jR_{n i n j}'\Big)+6\sum_{i=1}^n(\sum_{j=1}^n J^jR_{n i n j})^2\nonumber\\
&+2\sum_{i,j=1}^n\Big(\sum_{k=1}^n{J^k}R_{n j n k}R_{n i n j}+2{J^j}'R_{n i n j}'+J^jR_{n i n j}''\Big)J^i.
\end{align}
Combining \eqref{eqrr}, \eqref{eq317}, and \eqref{eq319}, there exists $C>0$ depending only on $C_R$ and $n$ such that
\begin{equation*}
\sup_{\theta\in [0,l]}|\zeta^{(4)}(\theta)|\le C\|v\|_g^2.
\end{equation*}

Taking $C_1 = \sqrt{C/4!}$ and substituting into \eqref{eq313} yields
\begin{equation*}
\|v\|_g^2 t^2-C_1 \|v\|_g^2 t^4\le \zeta(t)\le \|v\|_g^2 t^2+C_1 \|v\|_g^2 t^4.
\end{equation*}
Since $\|J_1\|_g^2=\frac{\zeta(t)}{t^2}$ and $l \leq \sqrt{1/(2C_1)}$, we obtain the final estimate:
\begin{equation*}\label{eq315}
(1-C_1 t^2)\|v\|_g^2\le \|J_1(t)\|_g^2\le (1+C_1 t^2)\|v\|_g^2,
\end{equation*}
completing the proof.
\end{proof}
%
%
For simplicity of notations, 
we denote 
\begin{equation}\label{rk21}
r_{m}=\min\Big\{\frac{1}{2}\bfi_M,\sqrt{\frac{1}{2C_1}},1\Big\}  
\end{equation}
in the rest of this paper.

Building upon Lemma \ref{Jaco}, we establish the following fundamental result concerning Jacobi fields with two-point boundary conditions, which serves as a crucial prerequisite for Lemma \ref{mlemma1}.

\begin{lemma} \label{Jacobi}
Let $x,y \in M$ with $l := \rho(x,y) \leq r_m$, and let $\gamma$ be the unique minimizing geodesic connecting $x$ and $y$, parameterized by arc length. Then for any tangent vectors $v \in T_xM$ and $w \in T_yM$, there exists a unique Jacobi field $J$ along $\gamma$ satisfying the boundary conditions
\begin{equation*}
J(0)=v,\qq J(l)=w.
\end{equation*}
Moreover, this Jacobi field admits the uniform estimate
\begin{equation*}
\|J(t)\|_g\le 3(\|v\|_g+\|w\|_g),\quad t\in [0,l].
\end{equation*}
\end{lemma}

\begin{proof}
By assumption and  Lemma \ref{Jaco}, then $(\exp_x)_{*l\gamma'_0}:T_xM\rightarrow T_yM$ is a linear isomorphism. Consequently, there exists a unique vector $\tilde{w} \in T_xM \cong T_{l\gamma_0'}(T_xM)$ satisfying	
\begin{equation*}
w=(\exp_x)_{*l\gamma'_0}(\tilde{w}),
\end{equation*}
with the lower bound, 
\begin{equation}\label{eq325}
\frac{1}{2}\|\tilde{w}\|_g\le \|(\exp_x)_{*l\gamma'_0}(\tilde{w})\|_g= \|w\|_g.
\end{equation}
We now construct the required Jacobi field in two parts. First, define
\begin{equation}\label{eq326}
J_1(t):=\frac{t}{l}(\exp_x)_{*t\gamma'_0}(\tilde{w}),\quad t\in [0,l].
\end{equation}
which yields a Jacobi field along $\gamma$ satisfying the boundary conditions
\begin{equation*}
J_1(0)=0,\quad J_1(l)=(\exp_x)_{*l\gamma'_0}(\tilde{w})=w.
\end{equation*}
Second, since $y$ is not conjugate to $x$, there exists a unique Jacobi field $J_2$ along $\gamma$ with
\begin{equation*}
J_2(0)=v,\quad J_2(l)=0.
\end{equation*}
The linear combination $J(t) := J_1(t) + J_2(t)$, $t\in [0,l]$ then gives the desired Jacobi field satisfying
\begin{equation*}
J(0)=v,\quad J(l)=w.
\end{equation*}
The uniqueness of $J$ follows from the non-conjugacy assumption.

For the norm estimate, combining \eqref{eq325} and \eqref{eq326} with Lemma \ref{Jaco} yields
\begin{equation*}
\|J_1(t)\|_g\le \frac{t}{l}\frac{3}{2}\|\tilde{w}\|_g\le \frac{3}{2}\|\tilde{w}\|_g\le 3\|w\|_g,
\end{equation*}
and consequently
\begin{equation*}
\|J(t)\|_g\le 3(\|w\|_g+\|v\|_g),\quad t\in [0,l].
\end{equation*}
This completes the proof of the lemma.
\end{proof}

Let $\Theta: [0,l] \times (-\epsilon, \epsilon) \rightarrow M$ be a smooth variation of the geodesic $\gamma$ satisfying $\Theta(t,0) = \gamma(t)$ for all $t \in [0,l]$. We define the associated length and energy functionals $\mathsf{L}$ and $\mathsf{E}$ by:
\begin{equation}\label{LE}
\mathsf{L}(s)=\int_0^l\|\partial_t \Theta(t,s)\|_gdt,\quad \mathsf{E}(s)=\int_0^l\|\partial_t \Theta(t,s)\|_g^2dt.
\end{equation}
where $\partial_t \Theta$ denotes the velocity vector field of the variation.

The variational vector field $\mathbb{J}$ along $\gamma$ is given by  $\mathbb{J}(t):=\frac{\partial}{\partial s}\Theta(t,s)\big|_{s=0}$. The second variation of energy is characterized by the following fundamental result:
\begin{lemma} \label{l2}\nmycite{doCarmo}{p.200}
For any smooth variation $\Theta$ with variational field $\mathbb{J}$ along the geodesic $\gamma$, the second derivative of the energy functional at $s=0$ satisfies:
\begin{equation*}
\frac{1}{2}\mathsf{E}''(0)=\int_0^l \big(\lan \mathbb{J}',\mathbb{J}'\ran _g  - \lan \mathbb{J},\mathcal{R}(\gamma',\mathbb{J})\gamma'\ran _g\big)dt+\lan D_{\frac{\partial}{\partial s}}\frac{\partial\Theta}{\partial s},\gamma'(t)\ran _g\bigg|_{t=0}^l,
\end{equation*}
where $\mathbb{J}'=D_{\gamma'}\mathbb{J}$ denotes the covariant derivative along $\gamma$.  
\end{lemma}

The following is the first main result of this subsection.
\begin{lemma} \label{mlemma1}
Let $(x,y) \in M \times M$ with $l := \rho(x,y) \leq r_m$, and consider tangent vectors $(v,w) \in T_xM \times T_yM$. For the operator $A$ defined in \eqref{scrA}, we have the following quadratic form estimate:
\begin{equation*}
A\big((v,w)^{\otimes 2}\big)= 2\|w-L_{xy}v\|_g^2+O_{C_2}(\rho^2(x,y)),
\end{equation*}
where
\begin{itemize}
\item $L_{xy}: T_xM \to T_yM$ denotes the parallel transport along the unique minimizing geodesic connecting $x$ and $y$
\item The remainder term satisfies $|O_{C_2}(\rho^2(x,y))| \leq C_2 \rho^2(x,y)$
\item The constant $C_2 = 6C_R(\|v\|_g + \|w\|_g)^2$ depends on the curvature bound $C_R$ and the vector norms.
\end{itemize}
\end{lemma}

\begin{proof} The proof is divided into two steps. 

{\bf Step 1: Variational Setup and Second Variation Formula}. This step is  almost the same as the proof of  \cite[Proposition 3.1]{Azagra2008}. We provide is here for the completeness. 

Consider the unique geodesic $\gamma$ connecting $x$ and $y$, parameterized by arc length with $\gamma(0)=x$ and $\gamma(l)=y$ where $l=\rho(x,y)$. For given tangent vectors $v\in T_xM$ and $w\in T_yM$, we examine the squared distance function $\varphi=\rho^2(x,y)$ and its second derivative:
\begin{equation*}
A\big((v,w)^{\otimes 2}\big)=\frac{d^2}{ds^2}\varphi(\exp_{x}(sv),\exp_{y}(sw)).
\end{equation*}
Let $\Theta_s$ denote the unique geodesic connecting $\exp_x(sv)$ to $\exp_y(sw)$, parameterized on $[0,l]$. For the associated length and energy functionals $\mathsf{L}$ and $\mathsf{E}$ defined in \eqref{LE}, the Cauchy-Schwarz inequality yields
$$\mathsf{L}(s)^2\le l\mathsf{E}(s)$$
with equality when $\|\Theta_s'(t)\|_g$ is constant. Since geodesics have constant speed, we obtain 
\begin{equation}\label{eq3220}
\frac{d^2}{ds^2}\varphi(\exp_{x}(sv),\exp_{y}(sw))=l\mathsf{E}''(0).
\end{equation}
Applying Lemma \ref{l2} gives the second variation formula
\begin{equation}\label{eq32200}
\frac{1}{2}\mathsf{E}''(0)=\int_0^l\big(\lan J',J'\ran _g-\lan \mathcal{R}(\gamma',J)\gamma',J\ran _g\big)dt+\lan D_{\frac{\partial}{\partial s}}\frac{\partial\Theta}{\partial s},\gamma'(t)\ran _g\bigg|_{t=0}^l.
\end{equation}
Since the variation $\Theta$ is a geodesic variation, then $J$ is a Jacobi field along $\gamma$ satisfying
\begin{equation}\label{eq322}
J''(t)+\mathcal{R}(\gamma'(t),J(t))\gamma'(t)=0,\quad J(0)=v,\quad J(l)=w.
\end{equation}
The existence and uniqueness of $J$ follows from Lemma \ref{Jacobi}. 

Let $\gamma_1=\Theta(0,s)=\exp_x(sv)$, and thus  
\begin{equation*}
D_{\frac{\partial}{\partial s}}\frac{\partial\Theta}{\partial s}\Big|_{s=0}=D_{\gamma_1'}\gamma_1'=0.
\end{equation*}
Integration by parts yields 
\begin{equation}\label{eq3222}
\int_0^l\big(\lan J',J'\ran _g-\lan \mathcal{R}(\gamma',J)\gamma',J\ran _g\big)dt=\lan w,J'(l)\ran _g-\lan v,J'(0)\ran _g.
\end{equation}

{\bf Step 2: Precise Estimation of Boundary Terms}. 
In this step,  we focus on the precise estimation of the right-hand side of the equation \eqref{eq3222}.

Define
$$\zeta_1(s):=\lan w,L_{\gamma(s) y}J(s)\ran _g,\quad s\in [0,l].$$
By applying the Mean Value Theorem to the function $\zeta$, we get that there exists $\theta_1\in (0,l)$ such that
\begin{equation}\label{ef}
\zeta_1'(\theta_1)=\frac{\zeta_1(l)-\zeta_1(0)}{l}=\lan w,\frac{w-L_{xy}v}{l}\ran _g.
\end{equation}
Since
\begin{align*}
\zeta_1'(\theta_1)&=\lim_{t\rightarrow 0}\frac{\zeta_1(\theta_1+t)-\zeta_1(\theta_1)}{t}\nonumber\\
&=\lim_{t\rightarrow 0}\lan w,\frac{L_{\gamma(\theta_1+t)y} J(\theta_1+t)-L_{\gamma(\theta_1)y} J(\theta_1)}{t}\ran _g\nonumber\\
&=\lan w,\lim_{t\rightarrow 0}L_{\gamma(\theta_1)y}\frac{L_{\gamma(\theta_1+t)\gamma(\theta_1)}J(\theta_1+t)-J(\theta_1)}{t}\ran _g\nonumber\\
&=\lan w,L_{\gamma(\theta_1) y}J'(\theta_1)\ran _g,
\end{align*}
we have
\begin{equation*}
\lan w,L_{\gamma(\theta_1) y}J'(\theta_1)\ran _g=\lan w,\frac{w-L_{xy}v}{l}\ran _g.
\end{equation*}

Applying the Mean Value Theorem to the function  $\zeta_2(s):=\lan w,L_{\gamma(s)y} J'(s)\ran $, we get that there exists $\theta_2\in (\theta_1,l)$ such that
\begin{equation*}
\zeta_2'(\theta_2)=\frac{\zeta_2(l)-\zeta_2(\theta_1)}{l-\theta_1}=\lan w,\frac{J'(l)-L_{\gamma(\theta_1)y}J'(\theta_1)}{l-\theta_1}\ran _g.
\end{equation*}
Using the Jacobi equation \eqref{eq322} and properties of the curvature tensor and parallel transport, we obtain that 
\begin{align}\label{eg}
\zeta_2'(\theta_2)&=\lan w,L_{\gamma(\theta_2)y}J''(\theta_2)\ran _g\nonumber\\
&=\lan L_{y\gamma(\theta_2)}w,-\mathcal{R}(\gamma'(\theta_2),J(\theta_2))\gamma'(\theta_2)\ran _g\nonumber\\
&=-R(\gamma'(\theta_2),J(\theta_2),\gamma'(\theta_2),L_{y\gamma(\theta_2)}w).
\end{align}
Combining this with results \eqref{ef}--\eqref{eg} and noting that $\rho(\gamma(\theta_1),y) = l-\theta_1$, we derive that
\begin{equation}\label{eq330}
\lan w,J'(l)\ran =\lan w,\frac{w-L_{xy}v}{l}\ran _g-R(\gamma'(\theta_2),J(\theta_2),\gamma'(\theta_1),L_{y\gamma(\theta_2)}w)\rho(\gamma(\theta_1),y).
\end{equation}
Applying the same methodology to $\langle v,J'(0)\rangle_g$, we define $\zeta_3(s) := \langle v,L_{\gamma(s)x}J(s)\rangle_g$ and find $\theta_3\in(0,l)$ satisfying
\begin{equation}\label{eq331}
\lan v,\frac{L_{yx}w-v}{l}\ran _g=\frac{\zeta_3(l)-\zeta_3(0)}{l}=\zeta_3'(\theta_3)=\lan v,L_{\gamma(\theta_3)x}J'(\theta_3)\ran _g.
\end{equation}
Further defining $\zeta_4(s) := \langle v,L_{\gamma(s)x}J'(s)\rangle_g$, there exists $\theta_4\in(0,\theta_3)$ with 
\begin{equation}\label{eq332}
\frac{\zeta_4(\theta_3)-\zeta_4(0)}{\theta_3}=\zeta_4'(\theta_4)=\lan v,L_{\gamma(\theta_4)x}J''(\theta_4)\ran _g,
\end{equation}
which implies that
\begin{equation}\label{eq333}
\lan v,J'(0)\ran _g=\lan v,\frac{L_{yx}w-v}{l}\ran _g - R(\gamma'(\theta_4),J(\theta_4),\gamma'(\theta_4),L_{x\gamma(\theta_4)}v)\rho(x,\gamma(\theta_4)).
\end{equation}
Using the curvature bound \eqref{eqrr} and Lemma \ref{Jacobi}, we establish that
\begin{equation}\label{eq334}
|R(\gamma'(\theta_2),J(\theta_2),\gamma'(\theta_2),L_{y\gamma(\theta_2)}w)| \le C_{R} \|L_{y\gamma(\theta_2)}w\|_g \|J\|_g
\le 3C_R\|w\|_g(\|v\|_g+\|w\|_g).
\end{equation}
and 
\begin{equation}\label{eq335}
|R(\gamma'(\theta_4),J(\theta_4),\gamma'(\theta_4),L_{x\gamma(\theta_4)}v)|\le C_{R} \|v\|_g \|J\|_g\le 3C_{R}\|v\|_g(\|v\|_g+\|w\|_g).
\end{equation}
Combining these results yields
\begin{align}\label{eq3351}
&\lan v,J'(0)\ran _g=\lan v,\frac{L_{yx}w-v}{l}\ran _g+O_{c_1}(l),\quad c_1= 3C_R\|v\|_g(\|v\|_g+\|w\|_g),\nonumber\\
&\lan w,J'(l)\ran _g=\lan w,\frac{w-L_{xy}v}{l}\ran _g+O_{c_2}(l),\quad c_2= 3C_R\|w\|_g(\|v\|_g+\|w\|_g).
\end{align}
Finally, combining with  equations from \eqref{eq3220} to \eqref{eq3222}  and \eqref{eq3351}, we obtain that
\begin{align*}
A\big((v,w)^{\otimes 2}\big)&=2l\big(\lan w,J'(l)\ran _g-\lan v,J'(0)\ran _g\big)\nonumber\\
&=2\lan w,w-L_{xy}v\ran _g+O_{2c_2}(l^2)-2\lan v,L_{yx}w-v\ran _g+O_{2c_1}(l^2)\nonumber\\
&=2\lan w,w-L_{xy}v\ran _g+O_{2c_2}(l^2)-2\lan L_{xy}v,L_{xy}(L_{yx}w-v)\ran _g+O_{2c_1}(l^2)\nonumber \\
&= 2\|w-L_{xy}v\|_g^2 + O_{C_2}(l^2),
\end{align*}
where the final constant is $C_2=6C_R(\|v\|_g+\|w\|_g)^2$.
\end{proof}

\begin{remark}
A result analogous to Lemma \ref{mlemma1} appears in \nmycite{Azagra2008}{Proposition 3.3}, established under the assumptions that $\rho(x,y) <\mathbf{i}_M$ and the manifold has sectional curvature bounded below by $-K_0$ (i.e., $K \geq -K_0$). This yields the inequality
\begin{equation}\label{AIshii} 
A\big((v, L_{xy}v)^{\otimes 2}\big) \leq 2K_0\rho^2(x,y)\|v\|^2, \quad \forall v \in T_xM,
\end{equation}
where $L_{xy}: T_xM \to T_yM$ denotes parallel transport along the minimal geodesic connecting $x$ and $y$.

It is crucial to observe that \eqref{AIshii} applies specifically to vectors in the $n$-dimensional subspace  $\{(v, L_{xy}v) \mid v \in T_xM\}$ of the full $2n$-dimensional tangent space $T_xM \times T_yM$. However, for proving uniqueness of viscosity solutions to \eqref{phjb} without the Parallel Condition (as discussed in the introduction), we require control of $A$ over the entire space $T_xM \times T_yM$, as provided by Lemma \ref{mlemma1}. This distinction is particularly important because the viscosity solution argument necessitates estimates for general vector pairs $(v,w) \in T_xM \times T_yM$, not just those related by parallel transport. 
\end{remark}

To analyze the higher-order term $A^2$, we begin by representing the bilinear form $A$ in an appropriately chosen basis.

For points $x, y \in M$ with $\rho(x,y) \leq r_m$, we select orthonormal bases $\{e_x^i\}_{i=1}^n$ for $T_xM$ and $\{e_y^i\}_{i=1}^n$ for $T_yM$ that are synchronized via parallel transport:
\begin{equation}\label{sync}
L_{xy}e_x^i = e_y^i \quad \forall 1 \leq i \leq n,
\end{equation}
where $L_{xy}: T_xM \to T_yM$ denotes parallel transport along the minimal geodesic connecting $x$ and $y$.

\begin{lemma} \label{lemma210} 
Suppose $x,y\in M$ and $\rho(x,y)\le  r_m$, under the bases defined in \eqref{sync}, then there exists positive constant $C$ depending only on $C_R$ such that
\begin{align*}
&A((e_x^i,0),(e_x^j,0))=2\delta_{ij}+O_{C}(\rho^2),\\
&A((0,e_y^i),(e_x^j,0))=-2\delta_{ij}+O_{C}(\rho^2),\\
&A((e_x^i,0),(0,e_y^j))=-2\delta_{ij}+O_{C}(\rho^2),\\
&A((0,e_x^i),(0,e_y^j))=2\delta_{ij}+O_{C}(\rho^2).
\end{align*}
where $\delta_{ij}$ denotes the Kronecker delta and the remainder terms $O_{C}(\rho^2)$ satisfy $|O_{C}(\rho^2)| \leq C\rho^2(x,y)$.
\end{lemma}
\begin{proof}
We establish the first two identities; the remaining cases follow similarly.

{\bf First Identity}: Applying Lemma \ref{mlemma1} through polarization and equality \eqref{sync}, we compute
\begin{align*}
A(x,y)((e_x^i,0),(e_x^j,0))&=\frac{1}{4}\Big\{A((e_x^i+e_x^j,0),(e_x^i+e_x^j,0))-A((e_x^i-e_x^j,0),(e_x^i-e_x^j,0))\Big\}\nonumber\\
&=\frac{1}{4}\Big\{2\|L_{xy}(e_x^i+e_x^j)\|_g^2-2\|L_{xy}(e_x^i-e_x^j)\|_g^2+2O_{C}(\rho^2)\Big\}\nonumber\\
&=2\lan e_x^i,e_x^j\ran _g+O_{C}(\rho^2)\nonumber\\
&=2\delta_{ij}+O_{C}(\rho^2),
\end{align*}
where the equalities use the orthonormality of ${e_x^i}$ and the synchronized basis property \eqref{sync}.
\ss

{\bf Second Identity}: Using the bilinearity and symmetry of $A$, we derive that
\begin{align*}
A((0,e_y^i),(e_x^j,0))&=\frac{1}{2}\Big\{A((e_x^i,e_y^j),(e_x^i,e_y^j))-A((e_x^i,0),(e_x^i,0))-A((0,e_y^j),(0,e_y^j))\Big\}\nonumber\\
&=\frac{1}{2}\Big\{2\|L_{xy}e_x^i-e_y^j\|_g^2-2\|L_{xy}e_x^i\|_g^2-2\|e_y^j\|_g^2+O_{C}(\rho^2)\Big\}\nonumber\\
&=\frac{1}{2}\Big\{-4\lan L_{xy}e_x^i,e_y^j\ran _g+O_{C}(\rho^2)\Big\}\nonumber\\
&= -2\delta_{ij}+O_{C}(\rho^2),
\end{align*}
where the final reduction follows from $\langle e_y^i,e_y^j\rangle_g = \delta_{ij}$ and the synchronized basis property \eqref{sync}.
\end{proof}
\begin{lemma} \label{lm38}(Error Term Estimates)
Let $\mathscr{E}_k$ denote the $k\times k$ matrix with all  elements being $1$, and $X\in \dbR^{k}$, $\mathbb{A}=(a_{ij})_{k\times k}\in \dbR^{k\times k}$. Then the following estimates hold for the error terms: %
\begin{equation}\label{lm381}
O_C(\mathscr{E}_k)(X,X)\le Ck\|X\|^2,
\end{equation}
where $\|\cdot\|$ denotes the Euclidean norm, 
and
\begin{equation}\label{lm382}
\begin{cases}\ds
\mathbb{A}O_C(\mathscr{E}_k)= O_C(\mathrm{abs}(\mathbb{A})\mathscr{E}_k), \\
\ns\ds O_C(\mathscr{E}_k)\mathbb{A}= O_C(\mathscr{E}_k\mathrm{abs}(\mathbb{A})),
\end{cases}
\end{equation}
where ${\mathrm{abs}}(\mathbb{A}) := (|a_{ij}|)_{k\times k}$ is the element-wise absolute value matrix.
\end{lemma}

\begin{proof}
Let $\mathbb{B} = (b_{ij})_{k \times k} \in \mathbb{R}^{k \times k}$ satisfy $\mathbb{B} = O_C(\mathscr{E}_k)$, and consider any vector $X = [x_1, \dots, x_k]^\top \in \mathbb{R}^k$. The quadratic form can be estimated as:	
\begin{equation*}
\mathbb{B}(X,X)=\sum_{i,j=1}^k x_ib_{ij}x_j\le \sum_{i,j=1}^k |x_i| |b_{ij}| |x_j|\le C \sum_{i,j=1}^k |x_i||x_j|\le  \frac{1}{2}C\sum_{i,j=1}^k(|x_i|^2+|x_j|^2)=Ck\|X\|^2,
\end{equation*}
which establishes the inequality \eqref{lm381}.

For the matrix product estimate, let $(\mathbb{C})_{ij}$ denote the $(i,j)$-entry of matrix $\mathbb{C}$. Then for $\mathbb{C} = \mathbb{A}\mathbb{B}$ we have 
\begin{equation*}
(\text{abs}(\mathbb{A}\mathbb{B}))_{ij}=|\sum_{l=1}^{k}a_{il}b_{lj}|\le  \sum_{l=1}^{k}|a_{il}||b_{lj}|\le C\sum_{l=1}^{k}|a_{il}|\cdot 1= C(\text{abs}(\mathbb{A})\mathscr{E}_k)_{ij},\quad i,j=1,\dots, n.
\end{equation*}
This proves that $\mathbb{A}\mathbb{B} = O_C(\text{abs}(\mathbb{A})\mathscr{E}_k)$, or equivalently $\mathbb{A}O_C(\mathscr{E}_k) = O_C(\text{abs}(\mathbb{A})\mathscr{E}_k)$. The second identity in \eqref{lm382} follows by similar arguments, completing the proof. 
\end{proof}

By using the Lemma $\ref{lemma210}$ and Lemma $\ref{lm38}$, we obtain the estimate for $A^2$, which is the second main result of this subsection.

\begin{lemma} \label{SecondOrderLemma}
Let $\{e_{x,i}\}_{i=1}^n \subset (T_xM)^*$ and $\{e_{y,i}\}_{i=1}^n \subset (T_yM)^*$ denote the corresponding dual basis satisfying:
$$
e_{x,i}(e_x^j) = \delta_{ij} \quad \text{and} \quad e_{y,i}(e_y^j) = \delta_{ij}, \quad \forall i,j = 1,\dotsc,n.
$$ 
Define the extended basis for the product space $T_xM \times T_yM$:
$$
e^i := 
\begin{cases}
e_x^i & 1 \leq i \leq n \\
e_y^{i-n} & n < i \leq 2n
\end{cases}, \quad 
e_i := 
\begin{cases}
e_{x,i} & 1 \leq i \leq n \\
e_{y,i-n} & n < i \leq 2n
\end{cases}
$$
where $i = 1,\dotsc,2n$. Let
$$
\mathbb{A}_I := \begin{bmatrix}
\mathbb{I}_n & -\mathbb{I}_n \\
-\mathbb{I}_n & \mathbb{I}_n
\end{bmatrix} \in \dbR^{2n \times 2n},
$$
where $\mathbb{I}_n$ denotes the $n \times n$ identity matrix, 
then $A$ admits the matrix representation:
\begin{equation}\label{eq:A_rep}
\mathbb{A} = 2\mathbb{A}_I + O_{C\rho^2(x,y)}(\mathscr{E}_{2n})
\end{equation}
with respect to the tensor basis $\{e_i \otimes e_j\}_{i,j=1}^{2n}$ for some positive constant $C$ depending only on $C_R$. For any tangent vectors $v,w \in T_xM \times T_yM$ with coordinates representations:
$$
\bar{v} = (\bar{v}_1,\dotsc,\bar{v}_{2n}), \quad \bar{w} = (\bar{w}_1,\dotsc,\bar{w}_{2n})
$$
under the basis $\{e^i\}_{i=1}^{2n}$, the bilinear forms $A$ and $A^2$ evaluate as:
$$
A(v,w) = \bar{v}^\top\mathbb{A}\bar{w},\quad A^2(v,w) = \bar{v}^\top\mathbb{A}^2\bar{w}.
$$

Then, the squared operator $\mathbb{A}^2$ satisfies
\begin{align}\label{eq:A2_est}
\mathbb{A}^2 = 4\mathbb{A}+O_{C'\rho^2(x,y)}(\mathscr{E}_{2n})
\end{align}
for some positive constant $C'$ depending on $C_R,n$.
\end{lemma}

\begin{proof}
Starting from the matrix representation \eqref{eq:A_rep}, we square both sides to obtain
\begin{equation}\label{lm2221}
\mathbb{A}^2=4\mathbb{A}_I^2+2\mathbb{A}_I O_{C\rho^2(x,y)}(\mathscr{E}_{2n})+2O_{C\rho^2(x,y)}(\mathscr{E}_{2n})\mathbb{A}_I+O_{C\rho^2(x,y)}(\mathscr{E}_{2n})O_{C\rho^2(x,y)}(\mathscr{E}_{2n}).
\end{equation}
A direct calculation shows that $\mathbb{A}_I^2=2\mathbb{A}_I$. Applying Lemma \ref{lm38} and using the structure of $\mathbb{A}_I$ and $\mathscr{E}_{2n}$, we obtain
\begin{equation}\label{lm2222}
2\mathbb{A}_I O_{C\rho^2(x,y)}(\mathscr{E}_{2n})= O_{2C\rho^2(x,y)}(\mathrm{abs}(\mathbb{A}_I)\mathscr{E}_{2n})=O_{4C\rho^2(x,y)}(\mathscr{E}_{2n}),
\end{equation}
since $\text{abs}(\mathbb{A}_I)\mathscr{E}_{2n} = 2\mathscr{E}_{2n}$.

Using the bound $\rho(x,y) \leq r_m \leq 1$, we estimate the quadratic remainder term:
\begin{equation}\label{lm2223}
O_{C\rho^2(x,y)}(\mathscr{E}_{2n})O_{C\rho^2(x,y)}(\mathscr{E}_{2n})= O_{C^2\rho^4(x,y)}(\mathscr{E}_{2n}^2)=O_{2nC^2\rho^2(x,y)}(\mathscr{E}_{2n}).
\end{equation}
where we have used that $\mathscr{E}_{2n}^2 = 2n\mathscr{E}_{2n}$.
Combining the estimates \eqref{lm2221}-\eqref{lm2223}, all terms involving $\mathbb{A}_I$ combine to give the  result:
\begin{equation*}
\mathbb{A}^2 = 8\mathbb{A}_I + O_{(4C + 4C + 2nC^2)\rho^2(x,y)}(\mathscr{E}_{2n}).
\end{equation*}
By comparing with the expression of $\mathbb{A}$ in \eqref{eq:A_rep}, set $C'=4+8C + 2nC^2$, thereby establishing \eqref{eq:A2_est}. 
\end{proof}

\subsection{Some properties of control system \eqref{eqve0}}\label{s23}

While standard references like Hsu's monograph \nmycite{hsu}{Proposition 1.2.8} establish well-posedness for Stratonovich-type SDEs with time-independent smooth vector fields on manifolds, our case requires broader consideration. Hence, we first  establish the well-posedness of system \eqref{eqve0}.

\begin{proposition}\label{pSExistence}
Under Assumption \ref{condp}, the state equation \eqref{eqve0} admits a unique  $\{\mathcal{F}_s\}_{t\le s\le T}$-adapted continuous solution.
\end{proposition}

\begin{proof}
We borrow some  idea in \nmycite{hsu}{Proposition 1.2.8} to solve \eqref{eqve11} on $\dbR^{n_1}$.  

The embedded process $\wt X$ satisfies the Stratonovich SDE:
\begin{align}\label{XN}
\left\{
\begin{aligned}
&d\wt X(s)=\wt b(s,\wt X(s),u(s))ds+\sum_{i=1}^m\wt \sigma_i(s,\wt X(s),u(s))\circ dW^i(s), &\quad s\in (t,T),\\
&\wt X(t)=\wt \xi, 
\end{aligned}
\right.
\end{align}
on the compact embedded manifold $\wt M := \mathcal{E}(M)$.

Under Assumption \ref{condp} and Lemma \ref{lm62}, we extend the vector fields $\wt b$ and $\wt \sigma_i$ ($i=1,...,m$) to compactly supported fields on $\dbR^{n_1}$ with supports in $\wt M_\epsilon$ (the $\epsilon$-tubular neighborhood of $\wt M$) for fixed $\epsilon\in (0,\epsilon_{\wt M})$ (see in \eqref{meps}), maintaining the Lipschitz properties:
\begin{align}\label{condext}
&\|b_0(t,z_1,u_1)-b_0(t,z_2,u_2)\|\le \Big(1+\frac{1}{2}m\Big)C_E(\|z_1-z_2\|+\|u_1-u_2\|),\nonumber\\
&\|\wt \sigma_i(t,z_1,u_1)-\wt \sigma_i(t,z_2,u_2)\|\le C_E(\|z_1-z_2\|+\|u_1-u_2\|),\quad i=1,...,m
\end{align}
for all $z_1,z_2\in \dbR^{n_1},u_1,u_2\in U$, for which $b_0 := \wt b + \frac{1}{2}\sum_{i=1}^m D_{\wt\sigma_i}^E\wt\sigma_i$ is the corrected drift term. The extended system in It\^o form becomes:
\begin{equation}\label{XN1}
\quad \wt X(t_1)=\wt \xi+\int_t^{t_1}b_0(r,\wt X(r),u(r))dr+\sum_{i=1}^m\int_t^{t_1}\wt \sigma_i(r,\wt X(r),u(r)) dW^i(r),\quad t_1\in [t,T].
\end{equation}
By \nmycite{Yong1999}{Theorem 6.16}, the inequality \eqref{condext} guarantees that the SDE \eqref{XN1} admits a unique $\{\mathcal{F}_s\}_{t\le s\le T}$ adapted continuous
solution in $ \dbR^{n_1} $.

Next,  we prove that the solution process $\wt X(s),s\in [t,T]$ remains on the manifold $\wt M$. To this end, we employ the squared distance function $\varphi_{\wt M}(z):=||v||^2$  for  $z\in \wt M_{\epsilon}$  from Lemma \ref{lmf}. Applying It\^o's formula (Definition \ref{df61}) yields
\begin{align*}
\varphi_{\wt M}(\wt X(t_1))=&\varphi_{\wt M}(\xi)+ \int_t^{t_1}\Big\{\wt b(r,\wt X(r),u(r))\varphi_{\wt M}(\wt X(r))+\frac{1}{2}\sum_{i=1}^m \sigma_{i}^2(r,\wt X(r),u(r))\varphi_M(\wt X(r))\Big\}dr\\
&+\int_t^{t_1}\sum_{i=1}^m\sigma_{i}(r,\wt X(r),u(r))\varphi_M(\wt X(r))dW^i(r),\quad t_1\in [t,T].
\end{align*}
Since $\wt \xi \in L_{\mathcal{F}_t}(\Omega,\wt M)$, we have $\varphi_{\wt M}(\wt\xi)=0$. Applying the Burkholder-Davis-Gundy' inequality, we get that
\begin{align*}
&\mathbb{E}\sup_{r\in [t,t_1]}|\varphi_{\wt M}(\wt X(r))|^2\nonumber\\
&\le C \mathbb{E}\int_t^{t_1}\Big\{|\wt b(r,\wt X(r),u(r))\varphi_{\wt M}(\wt X(r))|^2+  \sum_{i=1}^m|\wt\sigma_{i}^2(r,\wt X(r),u(r))\varphi_{\wt M}(\wt X(r))|^2\Big\}dr.\nonumber\\
&\le C \int_t^{t_1}\Big\{\mathbb{E}\sup_{r\in [t,s]}|\wt b(r,X(r),u(r))\varphi_{\wt M}(\wt X(r))|^2+  \sum_{i=1}^m\mathbb{E}\sup_{r\in [t,s]}|\wt \sigma_{i}^2(r,\wt X(r),u(r))\varphi_{\wt M}(\wt X(r))|^2\Big\}ds.\nonumber
\end{align*}
By Lemma \ref{lmf}, there exists a positive constant $C$ such that
\begin{align*}
\mathbb{E}\sup_{r\in [t,t_1]}|\varphi_{\wt M}(\wt X(r))|^2&\le C\int_t^{t_1}\mathbb{E}\sup_{r\in [t,s]}|\varphi_{\wt M}(\wt X(r))|^2ds.\nonumber
\end{align*}
Gronwall's inequality then implies $\varphi_{\wt M}(\wt X(s)) = 0$ for all $s \in [t,T]$, proving that $\wt X$ remains on $\wt M$. Consequently, $X=\mathcal{E}^{-1}(\wt X)$ is the unique $\{\mathcal{F}_s\}_{t\le s\le T}$-adapted continuous solution on $M$. 
\end{proof}

The following result establishes the continuous dependence of solutions to equation \eqref{eqve0} on both initial conditions and control inputs:
\begin{proposition} \label{prop41}
Let Assumption \ref{condp} hold.   Let $\xi_1,\xi_2\in L_{\mathcal{F}_t}(\Omega,M)$,$u_1(\cdot),u_2(\cdot)\in \mathcal{U}[t,T]$ and denote by $X_1(\cdot):=X(\cdot;t,\xi_1,u_1(\cdot))$, $X_2(\cdot):=X(\cdot;t,\xi_2,u_2(\cdot))$ the corresponding solution processes for simplicity, then there exists  a constant $C_D = C_D(C_\rho, C_E, C_B, m) > 0$  such that the following  estimates hold:
\begin{align}
&\mathbb{E}_t\sup_{s\in [t,T]}\rho^2(X_1(s),X_2(s))\le C_{D}\mathbb{E}_t\bigg\{\rho^2(\xi_1,\xi_2)+\int_t^T\|u_1(s)-u_2(s)\|^2ds\bigg\}, \label{apeq65-1}\\
&\mathbb{E}_{t}\sup_{r\in [t,t_1]}\rho^2(X(r;t,\xi_1,u_1(\cdot)),\xi_1)\le C_{D}(t_1-t),\quad t_1\in [t,T].\label{apeq65-2}
\end{align}
\end{proposition}

\begin{proof}
Building upon Lemma \ref{LMXX}, Assumption \ref{condp}, and Lemma \ref{lm62}, and following similar arguments as in Proposition \ref{pSExistence}, we establish the  estimates by working with the extended system \eqref{XN}.	

Consider the embedded solutions $\wt X_1 := \mathcal{E}(X_1)$ and $\wt X_2 := \mathcal{E}(X_2)$, which satisfy the It\^o SDEs:
$$\wt X_1(s)=\wt \xi_1+\int_t^s b_0(r,\wt X_1(r),u(r))dr+\sum_{i=1}^m\int_t^s\wt \sigma_i(r,\wt X_1(r),u(r))dW^i(r),\quad s\in [t,T]$$ 
,where the modified drift is given by
$$b_0(r,\wt X_1(r),u(r)):=b(r,\wt X_1(r),u(r))+\frac{1}{2}\sum_{i=1}^m D_{\wt {\sigma}_i}^E\wt \sigma_i(r,\wt X_1(r),u(r)).$$
Applying the Lipschitz conditions from \eqref{condext} and standard stochastic estimates from \nmycite{yong_stochastic_2022}{Proposition 4.1}, we obtain the following Euclidean norm estimates for some constant $C = C(m,C_E,C_B) > 0$:
\begin{align}\label{apeq651}
&\mathbb{E}_t\sup_{s\in [t,T]}\|\wt X_1(s)-\wt X_2(s)\|^2\le C\mathbb{E}_t\bigg\{\|\wt \xi_1-\wt \xi_2\|^2+\int_t^T\|u_1(s)-u_2(s)\|^2ds\bigg\},\nonumber\\
&\mathbb{E}_{t}\sup_{r\in [t,t_1]}\|\wt X(r;t,\wt \xi_1,u(\cdot))-\wt \xi_1\|^2\le C(t_1-t),
\end{align}
Using Lemma \ref{lm61}, we relate these estimates to the intrinsic distance on $\mathcal{E}(M)$:
\begin{align*}
&\mathbb{E}_t\sup_{s\in [t,T]}\frac{1}{C_{\rho}^2}\wt\rho^2(\wt X_1(s),\wt X_2(s))\le \mathbb{E}_t\sup_{s\in [t,T]}\|\wt X_1(s)-\wt X_2(s)\|^2,\nonumber\\
&\mathbb{E}_t\sup_{r\in [t,t_1]}\frac{1}{C_{\rho}^2}\wt \rho^2(\wt X(r;t,\wt\xi_1,u(\cdot)),\wt\xi_1)\le \mathbb{E}_t\sup_{r\in [t,t_1]}\|\wt X(r;t,\wt\xi_1,u(\cdot))-\wt\xi_1\|^2.
\end{align*}
The desired estimates then follow by setting $C_D := C_\rho^2 C$ and applying Lemma \ref{eqrho} to convert back to the original manifold $M$.
\end{proof}

\begin{lemma} \label{prop33}
Let $\xi \in L_{\mathcal{F}_t}(\Omega,M)$ be an initial condition and $u_i(\cdot) \in \mathcal{U}[t,T]$ for $i = 1,...,N$ be admissible control processes. Given any $\mathcal{F}_t$-measurable partition $\{\Omega_i\}_{i=1}^N$ of $\Omega$, define the composite control:
$$u(\cdot)=\sum_{i=1}^N u_i(\cdot)1_{\Omega_i},$$
Then the corresponding cost functional admits the decomposition:
$$\mathcal{J}(t,\xi;u(\cdot))=\sum_{i=1}^N\mathcal{J}(t,\xi,u_i(\cdot))1_{\Omega_i}.$$
\end{lemma} 
\begin{proof}
We present the case for $N=2$, as the general case follows similarly through induction. Let us define the solution processes: 
$$X_1(\cdot)=X(\cdot,t,\xi;u_1(\cdot)),\quad X_2(\cdot)=X(\cdot,t,\xi;u_2(\cdot)).$$
The combined process $\wt X_1=\mathcal{E}(X_1),\wt X_2=\mathcal{E}(X_1)$ on $\mathcal{E}(M)$ satisfies the following stochastic differential equation for $\wt\xi=\mathcal{E}(\xi)$:
\begin{align*}
\wt X_1(s)1_{\Omega_1}+\wt X_{2}(s)1_{\Omega_2}&=\wt\xi+\int_t^s\big[\wt b(r,\wt X_1(r),u_1(r))1_{\Omega_1}+\wt b(r,\wt X_2(r),u_2(r))1_{\Omega_2}\big]dr\\
&+\sum_{i=1}^m\int_t^s\big[\wt \sigma_i(r,\wt X_1(r),u_1(r))1_{\Omega_1}+\sigma_i(r,\wt X_2(r),u_2(r))1_{\Omega_2}\big]\circ dW^i(r).
\end{align*}
This establishes that the composite solution is:
$$\wt X(\cdot;t,\wt\xi,u(\cdot))=\wt X(\cdot;t,\wt\xi,u_1(\cdot))1_{\Omega_1}+\wt X(\cdot;t,\wt\xi,u_2(\cdot))1_{\Omega_2},$$
which means
$$X(\cdot;t,\xi,u(\cdot))=X_1(\cdot)1_{\Omega_1}+X_2(\cdot)1_{\Omega_2}.$$
The cost functional decomposition follows through these steps:
\begin{align*}
&\mathcal{J}(t,\xi;u_1(\cdot))1_{\Omega_1}+\mathcal{J}(t,\xi;u_2(\cdot))1_{\Omega_2}\\
&=\mathbb{E}_t\bigg\{\int_t^Tf(r,X_1(r),u_1(r))1_{\Omega_1}dr+f(r,X_2(r),u_2(r))1_{\Omega_2}dr+h(X_1(T))1_{\Omega_1}+h(X_2(T))1_{\Omega_2}\bigg\}\\
&=\mathbb{E}_t\bigg\{\int_t^Tf(r,X_1(r)1_{\Omega_1}+X_2(r)1_{\Omega_2},u_1(r)1_{\Omega_1}+u_2(r)1_{\Omega_2})dr+h(X_1(T)1_{\Omega_1}+h(X_2(T))1_{\Omega_2}\bigg\}\\
&=\mathcal{J}(t,\xi;u_1(\cdot)1_{\Omega_1}+u_2(\cdot)1_{\Omega_2})=\mathcal{J}(t,\xi;u(\cdot)).
\end{align*}
\end{proof}

For fixed $t \in [0,T]$ and $\xi \in L_{\mathcal{F}_t}(\Omega,M)$, consider any two controls $u_1(\cdot), u_2(\cdot) \in \mathcal{U}[t,T]$. Define the combined control
$$ u_{1}(\cdot)\sqcap u_2(\cdot):=u_{1}(\cdot)1_{S}+u_{2}(\cdot)1_{S^c}\in \mathcal{U}[t,T],
$$
where $S := \{\omega \in \Omega \mid \mathcal{J}(t,\xi;u_1(\cdot)) \leq \mathcal{J}(t,\xi;u_2(\cdot))\}$, and let $\alpha \wedge \beta := \min\{\alpha,\beta\}$. 

The following lemma establishes the existence of the essential infimum in the definition of the value function $\mathbb{V}$. The proof follows arguments similar to those in \cite[Lemma 4.4]{yong_stochastic_2022}.

\begin{lemma}  \label{effJ} 
For any $\xi\in L_{\mathcal{F}_t}(\Omega,M)$, there exists a sequence $u_k(\cdot)\in \mathcal{U}[t,T]$ such that
$$\mathcal{J}(t,\xi;u_{k}(\cdot))(\omega)\rightarrow \mathbb{V}(t,\xi),\quad a.s.$$
\end{lemma}

\begin{proof}
Fix $t \in [0,T]$ and $\xi \in L_{\mathcal{F}_t}(\Omega,M)$. Define the collection of cost functionals:
$$\mathscr{J}(t,\xi):=\{\mathcal{J}(t,\xi; u(\cdot))\mid u(\cdot)\in\mathcal{U}[t, T]\}.$$
Without loss of generality, we may assume all random variables in $\mathscr{J}(t,\xi)$ are essentially bounded. If not, we can instead consider the transformed set $\arctan \mathscr{J}(t,\xi) := \{\arctan \mathcal{J} \mid \mathcal{J} \in \mathscr{J}(t,\xi)\}$.

Consider the set of expected values:
$$
\mathbb{E}\mathscr{J}(t,\xi):=\{\mathbb{E}[\mathcal{J}] \mid \mathcal{J}\in \mathscr{J}(t,\xi)\} \subseteq \dbR.
$$
There exists a minimizing sequence $\{u_k(\cdot)\}_{k=1}^\infty \subseteq \mathcal{U}[t,T]$ such that
$$
\lim_{k\rightarrow\infty} \mathbb{E}\left[\mathcal{J}\left(t,\xi; u_{k}(\cdot)\right)\right]=\inf_{u(\cdot)\in\mathcal{U}[t, T]} \mathbb{E}[\mathcal{J}(t,\xi; u(\cdot))].
$$
Define recursively the controls
$$
\widehat{u}_k(\cdot):=u_1(\cdot)\sqcap u_2(\cdot)\sqcap\cdots\sqcap u_k(\cdot), \quad k\geq 1.
$$
By Lemma \ref{prop33}, we have
$$
\mathcal{J}\left(t,\xi;\widehat{u}_k(\cdot)\right)=\min_{1\leq i\leq k} \mathcal{J}\left(t,\xi; u_i(\cdot)\right).
$$
Thus, we may assume without loss of generality that $\{\mathcal{J}(t,\xi;u_k(\cdot))\}_{k\in\mathbb{N}}$ is nonincreasing by replacing $u_k(\cdot)$ with $\widehat{u}_k(\cdot)$ when necessary.

Define the limiting random variable
$$
\bar{\mathcal{J}}(\omega)=\inf_{k\geq 1} \mathcal{J}\left(t,\xi; u_{k}(\cdot)\right), \quad \omega\in\Omega,
$$
which is $\mathcal{F}_t$-measurable. We claim that $\bar{\mathcal{J}}$ is the essential infimum:
$$
\bar{\mathcal{J}} = \underset{u(\cdot)\in\mathcal{U}[t, T]}{\operatorname{essinf}} \mathcal{J}(t,\xi; u(\cdot)).
$$
To verify this claim, first let $\widehat{\mathcal{J}}$ be any $\mathcal{F}_t$-measurable random variable satisfying
$$
\widehat{\mathcal{J}}(\omega)\leq \mathcal{J}(t,\xi; u(\cdot)), \quad \forall u(\cdot)\in\mathcal{U}[t, T].
$$
Then $\widehat{\mathcal{J}}(\omega) \leq \mathcal{J}(t,\xi; u_k(\cdot))$ for all $k \geq 1$, and consequently
$$
\widehat{\mathcal{J}}(\omega)\leq\bar{\mathcal{J}}(\omega), \quad \text{a.s.}
$$
For the reverse inequality, consider any $v(\cdot) \in \mathcal{U}[t,T]$. Observe that
$$
\mathcal{J}(t,\xi; u_k(\cdot) \sqcap v(\cdot)) = \mathcal{J}(t,\xi; v(\cdot)) \wedge \mathcal{J}(t,\xi;u_k(\cdot)) \searrow \mathcal{J}(t,\xi; v(\cdot)) \wedge \bar{\mathcal{J}}(\omega) 
$$
as $k \to \infty$. Applying the monotone convergence theorem yields
$$
\begin{aligned}
	\mathbb{E}\left[\mathcal{J}(t,\xi; v(\cdot)) \wedge \bar{\mathcal{J}}\right] &= \lim_{k\rightarrow\infty} \mathbb{E}\left[\mathcal{J}(t,\xi; v(\cdot)) \wedge \mathcal{J}\left(t,\xi; u_{k}(\cdot)\right)\right] \\
	&\geq \inf_{u(\cdot)\in\mathcal{U}[t, T]} \mathbb{E}[\mathcal{J}(t,\xi; u(\cdot))] \\
	&= \mathbb{E}[\bar{\mathcal{J}}].
\end{aligned}
$$
This implies $\mathcal{J}(t,\xi; v(\cdot)) \geq \bar{\mathcal{J}}$ almost surely, completing the proof.
\end{proof}

\begin{lemma} \label{Lmcd1}(Lipschitz Continuity of Cost Functional)
There exists a constant $C_J = C_J(C_L,C_D,T) > 0$ such that for any initial states $\xi_1, \xi_2 \in L_{\mathcal{F}_t}(\Omega,M)$ and control processes $u_1(\cdot), u_2(\cdot) \in \mathcal{U}[t,T]$, the following estimate holds:
$$
|\mathcal{J}(t,\xi_1;u_1(\cdot))-\mathcal{J}(t,\xi_2,u_2(\cdot))|^2\le C_J\mathbb{E}_t\Big\{\rho^2(\xi_1,\xi_2)+\int_t^T\|u_1(s)-u_2(s)\|^2ds\Big\}.
$$
\end{lemma}

\begin{proof}
For notational simplicity, let 
\begin{align*}
X_1(\cdot) &:= X(\cdot;t,\xi_1,u_1(\cdot)),\quad X_2(\cdot) := X(\cdot;t,\xi_2,u_2(\cdot)), \\
\Delta &:= |\mathcal{J}(t,\xi_1;u_1(\cdot))-\mathcal{J}(t,\xi_2;u_2(\cdot))|.
\end{align*} 
Using the Lipschitz continuity of the cost components, we obtain that
\begin{align}\label{Del11}
\Delta&\le C_L\mathbb{E}_t\Big\{\int_t^T\big(\rho(X_1(r),X_2(r))+\|u_1(r)-u_2(r)\|\big)dr+\rho(X_1(T),X_2(T))\Big\}.
\end{align}
The following uniform bound holds for the solution trajectories:
\begin{equation}\label{XY}
\rho(X_1(r),X_2(r))\le \sup_{s\in [t,T]}\rho(X_1(s),X_2(s)),\quad r\in [t,s].
\end{equation}
Squaring both sides of \eqref{Del11} and applying H\"older's inequality yields:
$$\Delta^2\le 3C_L^2 \mathbb{E}_t\Big\{T^2\sup_{r\in [t,T]}\rho^2(X_1(r),X_2(r))+ T\int_t^T\|u_1(r)-u_2(r)\|^2dr+ \rho^2(X_1(T),X_2(T))\Big\}.$$
Combining Lemma \ref{prop41} and the bound \eqref{XY}, we obtain the existence of a constant $C_J = C_J(C_L,C_D,T)$ such that 
$$\Delta^2\le C_J\mathbb{E}_t\Big\{\rho^2(\xi_1,\xi_2)+\int_t^T\|u_1(r)-u_2(r)\|^2dr\Big\}.$$
This completes the proof of the cost functional's Lipschitz continuity.
\end{proof}

\subsection{An approximation result of the control}\label{s24}

Let's define 
\begin{flushleft}
\quad $\mathcal{F}_s^t:=\sigma(W(s)-W(t),t\le s\le T),\quad \mathbb{F}^t:=\{\mathcal{F}_s^t\}_{s\ge t}.$\\
\quad $L_{\mathbb{F}^{t}}^{0}(t,T;U):=\Big\{u(\cdot)=\sum_{i=0}^{N_1}u_i1_{[t_i,t_{i+1}]},t=t_0\le ...\le t_{N_1}=T,u_i \in L_{\mathcal{F}_{t_i}^t}(\Omega,U),N_1\in \mathbb{N}\Big\}.$\\
\quad $\mathcal{V}[t,T]:=\mathcal{U}[t,T]\cap \overline{L_{\mathbb{F}^{t}}^{0}(t,T;U)}^{\|\cdot\|_{\mathbb{E}}}$, for  $\|\cdot\|_{\mathbb{E}}$  the norm of $L_{\mathbb{F}}^{2}(t,T;\dbR^k)$. \\
\quad $\mathcal{W}[t,T]:=\big\{u(\cdot)=\sum_{i=1}^Nu^j(\cdot)1_{\Omega_j}\big|u^j(\cdot)\in L_{\mathbb{F}^t}^{0}(t,T;U),\{\Omega_j\}_{j=1}^N\text{ is a partition of }(\Omega,\mathcal{F}_t)\big\}$.\\
\end{flushleft}
%


The following approximation result, which adapts the proof technique from \nmycite{yong_stochastic_2022}{Lemma 4.12} with necessary modifications, will be essential for proving Theorem \ref{phjbl}:
\begin{lemma} \label{Appro0}
For any given control $u(\cdot) \in \mathcal{U}[t,T]$ and approximation tolerance $\epsilon > 0$, there exists an approximating control $u_\epsilon(\cdot) \in \mathcal{W}[t,T]$ satisfying the following $L^2$-approximation property:
\begin{equation}
\mathbb{E}\left[ \int_{t}^T\|u(r) - u_{\epsilon}(r)\|^2 dr \right] \leq \epsilon.
\end{equation}
\end{lemma}

\begin{proof}
The proof is split into two steps.

\textbf{Step 1: Sigma-Field Construction and Approximation}.  For any $s \in [t,T]$, we first define the collection:
$$
\mathcal{G} = \left\{ \medcup_{k=1}^N \left( S^k \cap \widehat{S}^k \right) \mid S^k \in \mathcal{F}_t, \widehat{S}^k \in \mathcal{F}_s^t, 1 \leq k \leq N, N \geq 1 \right\}.
$$

It is easy to see that
$$
\mathcal{F}_{t} \cup \mathcal{F}_{s}^{t} \subseteq \mathcal{G} \subseteq \mathcal{F}_{t} \vee \mathcal{F}_{s}^{t} = \mathcal{F}_{s}
$$
which leads to
$$
\mathcal{F}_{s} = \mathcal{F}_{t} \vee \mathcal{F}_{s}^{t} = \bm{\sigma}(\mathcal{G}).
$$
This means that $\mathcal{G}$ generates $\mathcal{F}_{s}$. Next, it is clear that $\mathcal{G}$ is closed under finite union (which is obvious), finite intersection, and complement.

Next, we define the approximation class:
$$
\mathcal{G} \subseteq \widetilde{\mathcal{G}} \triangleq \left\{ F \in \mathcal{F}_s \mid \forall \varepsilon > 0, \exists F_{\varepsilon} \in \mathcal{G}, \mathbb{P}\left( F \Delta F_{\varepsilon} \right) < \varepsilon \right\} \subseteq \mathcal{F}_s,
$$
where $ F \Delta F_{\varepsilon} = \left( F \cap F_{\varepsilon}^c \right) \cup \left( F_{\varepsilon} \cap F^c \right) = F^c \Delta F_{\varepsilon}^c $ denotes the symmetric difference. Clearly, $\widetilde{\mathcal{G}}$ is a $\bm{\sigma}$-field.  Hence, combining $ \mathcal{G} \subseteq \widetilde{\mathcal{G}} \subseteq \mathcal{F}_{s} $, we obtain $\mathcal{F}_s = \widetilde{\mathcal{G}}$.

\ss

\textbf{Step 2: Approximation of Control Processes}. 
Let $u(\cdot) \in \mathcal{U}[t,T]$. By definition of $\mathcal{U}[t,T]$ in introduction, for any $\epsilon>0$, there exists a simple adapted process: 
$$
u_{\epsilon}(r) = \sum_{i=0}^{N_1} u_i 1_{\left[ t^i, t^{i + 1} \right)}(r),\quad u_i\in L_{\mathcal{F}_{t_i}}(\Omega,U) \quad r \in [t, T],
$$
with time partition $t = t^0 < t^1 < \cdots < t^{N_1+1} = T$, satisfying
\begin{equation}
\mathbb{E}\int_0^T\|u_{\epsilon}(r) - u(r)\|^2dr < \varepsilon, \quad r \in [t, T], \text{ a.s.}
\end{equation}
Since $U$ is compact, for $\widehat{\varepsilon} > 0$ we partition:
$$
U = \medcup_{j=1}^{N_2} U_j, \quad \operatorname{diam} \left( U_j \right) < \widehat{\varepsilon},\quad 1\le j_1\le N_2.
$$
Define measurable sets:
$$
\Omega_\epsilon^{ij} := \left\{ \omega \in \Omega \mid u_i \in U_j \right\} \in \mathcal{F}_{t^i}, \quad 0 \leq i \leq N_1, 1 \leq j \leq N_2.
$$
forming an $\mathcal{F}_{t^i}$-partition of $\Omega$ with
$$
\min_{0 \leq i \leq N_1, 1 \leq j \leq N_2} \mathbb{P} \left( \Omega_\epsilon^{ij} \right) = \varepsilon_0 >0.
$$
Define
$$
\widehat{u}_\epsilon(r) = \sum_{i=0}^{N_1} \sum_{j=1}^{N_2} u^{ij} 1_{\Omega_\epsilon^{ij}} 1_{\left[ t^i, t^{i + 1} \right)}(r), \quad r \in [t, T],
$$
where $u^{ij} \in U_j$. One has $ \widehat{u}_\epsilon(\cdot) \in \mathcal{U}[t, T] $ and
$$
\left| u_\epsilon(r) - \widehat{u}_\epsilon(r) \right| \leq \widehat{\varepsilon}, \quad r \in [t, T], \text{ a.s.}
$$

By Step 1, for each  $\Omega_{\epsilon}^{ij} \in \mathcal{F}_{t^{i}} $ and for any $ \widetilde{\varepsilon} \in \left( 0, \varepsilon_0 \right) $,  there exists $S_\epsilon^{ij} \in \mathcal{G}$ with
\begin{equation}\label{eq1001}
\mathbb{P} \left( \Omega_\epsilon^{ij} \Delta S_\epsilon^{ij} \right) < \widetilde{\varepsilon}. 
\end{equation}
Then
$$
\mathbb{P} \left( \Omega_\epsilon^{ij} \cap S_\epsilon^{ij} \right) \geq \mathbb{P} \left( \Omega_\epsilon^{ij} \right) - \widetilde{\varepsilon}.
$$

From
$$
\medcup_{j=1}^{N_2} S_\epsilon^{ij} = \medcup_{j=1}^{N_2} \left[ \left( S_\epsilon^{ij} \cap \Omega_\epsilon^{ij} \right) \cup \left( S_\epsilon^{ij} \backslash \Omega_\epsilon^{ij} \right) \right],
$$
and \eqref{eq1001}, one obtains, for each $ 1 \leq i \leq N_{1} $,
\begin{equation}\label{eq1002}
\mathbb{P} \left( \medcup_{j=1}^{N_2} S_\epsilon^{ij} \right) \geq \sum_{j=1}^{N_2} P \left( S_\epsilon^{ij} \cap \Omega_\epsilon^{ij} \right) \geq \sum_{j=1}^{N_2} \left[ P \left( \Omega_\epsilon^{ij} \right) - \widetilde{\varepsilon} \right] = 1 - N_2 \widetilde{\varepsilon}.
\end{equation}
Define exceptional sets:
$$
S_\epsilon^{i0} = \left( \medcup_{j=1}^{N_2} S_\epsilon^{ij} \right)^c, \quad 0 \leq i \leq N_1.
$$
Then we have
$$
\mathbb{P} \left( S_\epsilon^{i0} \right) \leq N_2 \widetilde{\varepsilon}, \quad 0 \leq i \leq N_1.
$$

Next, since
$$
\Omega_\epsilon^{ij} \cap \Omega_\epsilon^{ij^{\prime}} = \emptyset, \quad j \neq j^{\prime},
$$
for $j\neq j^{\prime}$, by direct computation, we have 
\begin{align*}
& S_\epsilon^{ij} \cap S_\epsilon^{ij^{\prime}}
\subseteq \left( S_\epsilon^{ij} \Delta \Omega_\epsilon^{ij} \right) \cup \left( S_\epsilon^{ij^{\prime}} \Delta \Omega_\epsilon^{ij^{\prime}} \right),
\end{align*}
which leads to
\begin{equation}\label{eq1003}
\mathbb{P} \left( S_\epsilon^{ij} \cap S_\epsilon^{ij^{\prime}} \right) < 2 \widetilde{\varepsilon}, \quad \forall j \neq j^{\prime}.
\end{equation}
Note that the following is a partition of $ \Omega $ in $ \mathcal{F}_{t^{i}} $:
$$
\left\{ S_\epsilon^{i0}, S_\epsilon^{ij} \backslash \medcup_{j^{\prime}=1}^{j - 1} S_\epsilon^{ij^{\prime}}, \quad 1 \leq j \leq N_2 \right\}.
$$

Define the approximating control:
$$
\bar{u}_\epsilon(r, \omega) = \begin{cases} 
u^{ij}, & \omega \in S_\epsilon^{ij} \backslash \medcup_{j^{\prime}=1}^{j - 1} S_\epsilon^{ij^{\prime}}, \quad 1 \leq j \leq N_2, \\
u^0, & \omega \in S_\epsilon^{i0},
\end{cases} \quad r \in \left[ t_i, t_{i + 1} \right),
$$
for some fixed $u^{0} \in U $. Thus, $ \bar{u}_{\epsilon}(\cdot) $ is well-defined, and for each $ 0 \leq i \leq N_{1} $, $ 1 \leq j \leq N_{2} $,
$$
\bar{u}_\epsilon(r, \omega) = u^{ij} = \widehat{u}_\epsilon(r, \omega), \quad \forall \omega \in \Omega_\epsilon^{ij} \cap \left( S_\epsilon^{ij} \backslash \medcup_{j^{\prime}=1}^{j - 1} S_\epsilon^{ij^{\prime}} \right).
$$
Therefore, for each $ 0 \leq i \leq N_{1} $ and $ r \in \left[ t^{i}, t^{i + 1} \right) $,
\begin{align*}
& \left\{ \omega \in \Omega \mid \bar{u}_\epsilon(r, \omega) \neq \widehat{u}_\epsilon(r, \omega) \right\} \\
& \subseteq \medcup_{j=1}^{N_2} \left\{ S_\epsilon^{ij} \backslash \left[ \Omega_\epsilon^{ij} \cap \left( S_\epsilon^{ij} \backslash \medcup_{j^{\prime}=1}^{j - 1} S_\epsilon^{ij^{\prime}} \right) \right] \right\} \medcup S_\epsilon^{i0}\\
& = \medcup_{j=1}^{N_2} \left[ \left( S_\epsilon^{ij} \backslash \Omega_\epsilon^{ij} \right) \cup \left( \medcup_{j^{\prime}=1}^{j - 1} \left( S_\epsilon^{ij} \cap S_\epsilon^{ij^{\prime}} \right) \right) \right] \medcup \left( \medcup_{j=1}^{N_2} S_\epsilon^{ij} \right)^c.
\end{align*}
Hence, by \eqref{eq1001},\eqref{eq1002} and \eqref{eq1003},
\begin{align*}
& \mathbb{P} \left( \left\{ \omega \in \Omega \mid \bar{u}_\epsilon(t, \omega) \neq \widehat{u}_\epsilon(r, \omega) \right\} \right) \\
& \leq \sum_{j=1}^{N_2} \left[ \mathbb{P}  \left( S_\epsilon^{ij} \backslash \Omega_\epsilon^{ij} \right) + \sum_{j^{\prime}=1}^{j - 1} \mathbb{P} \left( S_\epsilon^{ij} \cap S_\epsilon^{ij^{\prime}} \right) \right] + 1 - \mathbb{P}  \left( \medcup_{j=1}^{N_2} S_\epsilon^{ij} \right) \\
& \leq N_2 \widetilde{\varepsilon} + \sum_{j=1}^{N_2} 2(j - 1) \widetilde{\varepsilon} + N_2 \widetilde{\varepsilon} = \left( 2N_2 + N_2^2 \right) \widetilde{\varepsilon}, \quad r \in \left[ t_i, t_{i + 1} \right).
\end{align*}

\ss

{\bf Step 3.  Finite Approximation Construction}
Let $C_u := \max_{u\in U}\|u\|$ be the uniform bound on control values. We first estimate the approximation error:
$$
\mathbb{E} \int_{t}^{T} \| \bar{u}_{\epsilon}(r) - \widehat{u}_{\epsilon}(r) \|^2 dr \leq 4C_u^2T \left( 2N_{2} + N_{2}^{2} \right) \widetilde{\varepsilon}.
$$
Since each $S_{\epsilon}^{ij} \in \mathcal{G}$, we can express them as finite unions:
$$
S_\epsilon^{ij} = \medcup_{k=1}^{N_{ij}} \left( S_\epsilon^{ijk} \cap \widehat{S}_\epsilon^{ijk} \right), \quad 1 \leq i \leq N_1, 1 \leq j \leq N_2,
$$
where $S_{\epsilon}^{ijk} \in \mathcal{F}_{t}$ and $\widehat{S}_{\epsilon}^{ijk} \in \mathcal{F}_{t^{i}}^{t}$.

Let $ \mathcal{F}_{t}^{\{\epsilon\}} $ be the $\bm{\sigma}$-field generated by
$$
\left\{ S_{\epsilon}^{ijk} \mid 0 \leq i \leq N_{1}, 0 \leq j \leq N_{2}, 1 \leq k \leq N_{ij} \right\}\subset \mathcal{F}_{t},
$$
which is still a subset of $ \mathcal{F}_{t} $. Let
$$
\mathbf{S}_\epsilon := \left\{ S_\epsilon^{\ell} \mid 1 \leq \ell \leq N_3 \right\} 
$$
be a partition of $ \Omega $ in $ \mathcal{F}_{t} $,
and let $ \mathcal{F}_{t}^{\{\epsilon\}}:= \bm{\sigma} \left( \mathbf{S}_{\epsilon} \right) $. Likewise,
$$
\left\{ \widehat{S}_\epsilon^{ijk} \mid 0 \leq i \leq N_1, 1 \leq j \leq N_2, 1 \leq k \leq N_{ij} \right\}
$$
is a finite subset of $ \mathcal{F}_{t^{i}}^{t} $, which generates a (finite) $\bm{\sigma}$-field $ \widehat{\mathcal{F}}_{t^{i}}^{t, \{\epsilon\}} $, and we can obtain a (finite) partition of $ \Omega $ in $ \mathcal{F}_{t^{i}}^{t} $, denoted by
$$
\widehat{\mathbf{S}}_\epsilon = \left\{ \widehat{S}_\epsilon^\ell \mid 1 \leq \ell \leq \widehat{N}_3 \right\},
$$
such that $\bm{\sigma}(\widehat{\mathbf{S}}_{\epsilon}) = \widehat{\mathcal{F}}_{t^{i}}^{t, \{\epsilon\}} $. Hence, we have the following:
\begin{align*}
& S_\epsilon^{ij} \backslash \medcup_{j^{\prime}=1}^{j - 1} S_\epsilon^{ij^{\prime}} = \left( \medcup_{k=1}^{N_{ij}} S_\epsilon^{ijk} \cap \widehat{S}_\epsilon^{ijk} \right) \backslash \medcup_{j^{\prime}=1}^{j - 1} \left( \medcup_{k=1}^{N_{ij^{\prime}}} S_\epsilon^{ij^{\prime}k} \cap \widehat{S}_\epsilon^{ij^{\prime}k} \right) \\
& = \medcup_{\ell \in \Lambda(i, j)} S_\epsilon^\ell \cap \left( \medcup_{\ell^{\prime} \in \Lambda^{\prime}(i, j, \ell)} \widehat{S}_\epsilon^{\ell^{\prime}} \right), \quad 0 \leq i \leq N_1, 1 \leq j \leq N_2,
\end{align*}
\begin{align*}
S_{\epsilon}^{i0} & = \left( \medcup_{j=1}^{N_{2}} S_{\epsilon}^{ij} \right)^{c} = \left( \medcup_{j=1}^{N_{2}} \medcup_{k=1}^{N_{ij}} S_{\epsilon}^{ijk} \cap \widehat{S}_{\epsilon}^{ijk} \right)^{c} \\
& = \left( \bigcap_{j=1}^{N_{2}} \bigcap_{k=1}^{N_{ij}} \left( S_{\epsilon}^{ijk} \right)^{c} \cup \left( \widehat{S}_{\epsilon}^{ijk} \right)^{c} \right) = \medcup_{\ell \in \Lambda(i, 0)} S_{\epsilon}^{\ell} \cap \left( \medcup_{\ell^{\prime} \in \Lambda^{\prime}(i, 0, \ell)} \widehat{S}_{\epsilon}^{\ell^{\prime}} \right), \quad 0 \leq i \leq N_{1},
\end{align*}
where $ \Lambda(i, j) $ and $ \Lambda^{\prime}(i, j, \ell) $ are suitable index sets.

For any $(r,\omega) \in [t,T) \times \Omega$, there exists a unique time index $i \in \{0,1,\cdots,N_1\}$ such that $r \in [t^i,t^{i+1})$. 

If $\omega \in \bigcup_{\ell \in \Lambda(i,j)} S_\epsilon^\ell \cap \left( \bigcup_{\ell' \in \Lambda'(i,j,\ell)} \widehat{S}_\epsilon^{\ell'} \right)$ for some $1 \leq j \leq N_2$, then
$$
\bar{u}_{\epsilon}(r, \omega) = u^{ij}.
$$

Otherwise, if $\omega \in \bigcup_{\ell \in \Lambda(i,0)} S_\epsilon^\ell \cap \left( \bigcup_{\ell' \in \Lambda'(i,0,\ell)} \widehat{S}_\epsilon^{\ell'} \right)$, then
$$
\bar{u}_\epsilon(r,\omega) = u^0.
$$

Thus, the control process can be expressed as
$$
\bar{u}_\epsilon(r, \omega) = \sum_{\ell=1}^{N_3} 1_{S_\epsilon^{\ell}} \sum_{i=0}^{N_1} \sum_{\ell^{\prime} \in \widehat{\Lambda}(i, \ell)} u_{\ell}^i 1_{\widehat{S}_\epsilon^{\ell^{\prime}}}(\omega) 1_{\left[ t^i, t^{i + 1} \right)}(r) \equiv \sum_{\ell=1}^{N_1} 1_{S_\epsilon^{\ell}}(\omega) \widehat{u}_\epsilon^{\ell}(r, \omega),
$$
where $ u_{\ell}^{i} \in \left\{ u^{0}, u^{ij}, 1 \leq j \leq N_{2} \right\} $ are the possible control values, and $ \widehat{u}_{\epsilon}^{\ell}(\cdot) \in \mathcal{V}(t,T) $ are adapted processes. This decomposition establishes the desired representation, completing the proof of the lemma.
\end{proof}

\subsection{Properties of the value function}\label{s25} 

We first prove  Propositions \ref{Detm} and \ref{pLip}.

\begin{proof}[\textbf{Proof of Proposition \ref{Detm}}]
We establish the equality
$$\mathop{\rm essinf}_{u(\cdot)\in \mathcal{U}[t,T]}\mathcal{J}(t,x;u(\cdot))=\mathop{\rm essinf}_{u(\cdot)\in \mathcal{V}[t,T]}\mathcal{J}(t,x;u(\cdot))=\mathop{\rm essinf}_{u(\cdot)\in \mathcal{W}[t,T]}\mathcal{J}(t,x;u(\cdot))$$
through two key steps, following the approach in \cite{yu} and \cite{yong_stochastic_2022}.

\ss

{\bf Step 1: Approximation from $\mathcal{U}[t,T]$ to $\mathcal{W}[t,T]$}.

For any control $u(\cdot) \in \mathcal{U}[t,T]$, Lemma \ref{Appro0} guarantees the existence of approximating controls ${u_n(\cdot)}_{n=1}^\infty \subset \mathcal{W}[t,T]$ satisfying:
$$\mathbb{E}\int_{t}^T\|u(r)-u_{n}(r)\|^2dr\le \frac{1}{n}.$$
By the estimates in Lemma \ref{Lmcd1}, we have
$$\mathbb{E}\big|\mathcal{J}(t,x;u_n(\cdot))-\mathcal{J}(t,x;u(\cdot))\big|^2 \le \frac{C_J}{n},a.s..$$
Taking the limit as $n \to \infty$ yields
$$\lim_{n\rightarrow \infty}\mathcal{J}(t,x;u_n(\cdot))=\mathcal{J}(t,x;u(\cdot)),\quad a.s..$$
By the arbitrariness of $u(\cdot)$ and the definition of essential infimum, we obtain that
$$\mathop{\rm essinf}_{u(\cdot)\in \mathcal{U}[t,T]}\mathcal{J}(t,x;u(\cdot))\ge \mathop{\rm essinf}_{u(\cdot)\in \mathcal{W}[t,T]}\mathcal{J}(t,x;u(\cdot)).$$
Since $\mathcal{W}[t,T] \subset \mathcal{U}[t,T]$, the reverse inequality holds, establishing
$$\mathop{\rm essinf}_{u(\cdot)\in \mathcal{W}[t,T]}\mathcal{J}(t,x;u(\cdot))= \mathop{\rm essinf}_{u(\cdot)\in \mathcal{U}[t,T]}\mathcal{J}(t,x;u(\cdot)).$$

\ss

{\bf Step 2: Equivalence Between $\mathcal{V}[t,T]$ and $\mathcal{W}[t,T]$}.

In this step, we prove
$$\mathop{\rm essinf}_{u(\cdot)\in \mathcal{V}[t,T]}\mathcal{J}(t,x;u(\cdot))
=\mathop{\rm essinf}_{u(\cdot)\in \mathcal{W}[t,T]}\mathcal{J}(t,x;u(\cdot)).$$

First, the trivial inequality holds
$$\mathop{\rm essinf}_{u(\cdot)\in \mathcal{V}[t,T]}\mathcal{J}(t,x;u(\cdot))
\ge \mathop{\rm essinf}_{u(\cdot)\in \mathcal{W}[t,T]}\mathcal{J}(t,x;u(\cdot)).$$
For any $u(\cdot) = \sum_{i=1}^N u_i(\cdot)1_{\Omega_i} \in \mathcal{W}[t,T]$, Lemma \ref{prop33} and the deterministic nature of $\mathcal{J}(t,x;u_i(\cdot))$ yield 
$$\mathcal{J}(t,x;u(\cdot))=\sum_{i=1}^N1_{\Omega_i} \mathcal{J}(t,x;u_i(\cdot))\ge \sum_{i=1}^N1_{\Omega_i} \mathop{\rm essinf}_{u(\cdot)\in \mathcal{V}[t,T]}\mathcal{J}(t,x;u(\cdot))= \mathop{\rm essinf}_{v(\cdot)\in \mathcal{V}[t,T]}\mathcal{J}(t,x;v(\cdot)).$$
Taking the essential infimum over $\mathcal{W}[t,T]$ gives
$$\mathop{\rm essinf}_{u(\cdot)\in \mathcal{W}[t,T]}\mathcal{J}(t,x;u(\cdot))
\ge \mathop{\rm essinf}_{u(\cdot)\in \mathcal{V}[t,T]}\mathcal{J}(t,x;u(\cdot)),$$
This completes the proof of the equality chain
$$
\mathop{\rm essinf}_{u(\cdot)\in \mathcal{U}[t,T]} \mathcal{J}(t,x;u(\cdot)) = \mathop{\rm essinf}_{u(\cdot)\in \mathcal{V}[t,T]} \mathcal{J}(t,x;u(\cdot)) = \mathop{\rm essinf}_{u(\cdot)\in \mathcal{W}[t,T]} \mathcal{J}(t,x;u(\cdot)).
$$
\end{proof}

\begin{proof}[\textbf{Proof of Proposition \ref{pLip}}]
Let $x,y\in M,u_k(\cdot), \tilde{u}_k(\cdot) \in \mathcal{U}[t,T]$ be control sequences satisfying
$$\mathcal{J}(t,x;u_k(\cdot))\rightarrow V(t,x),\quad \mathcal{J}(t,x;\tilde{u}_k(\cdot))\rightarrow V(t,y).$$
By Lemma \ref{Lmcd1}, there exists a constant $C_J > 0$ such that
$$|\mathcal{J}(t,x;u_k(\cdot))-\mathcal{J}(t,y,u_k(\cdot))|,\quad |\mathcal{J}(t,x;\tilde{u}_k(\cdot))-\mathcal{J}(t,y,\tilde{u}_k(\cdot))\le 2C_{J}\rho(x,y).$$
According to the definition of $V(\cdot,\cdot)$, for any $\epsilon>0$, there exists $k_0 \in \mathbb{N}$, such that for all $k\ge k_0$:  
$$\mathcal{J}(t,x;\tilde{u}_k(\cdot))\ge V(t,x)\ge \mathcal{J}(t,x;u_k(\cdot))-\epsilon,$$
$$\mathcal{J}(t,y;u_k(\cdot))\ge V(t,y)\ge \mathcal{J}(t,y;\tilde{u}_k(\cdot))-\epsilon.$$
Combining these estimates yields
\begin{align*}
&-C_J\rho(x,y)-\epsilon\le \mathcal{J}(t,x;u_k(\cdot))-\mathcal{J}(t,y;u_k(\cdot))-\epsilon\\
&\le V(t,x)-V(t,y)\le  \mathcal{J}(t,x;\tilde{u}_k(\cdot))-\mathcal{J}(t,y;\tilde{u}_k(\cdot))+\epsilon\le
C_J\rho(x,y)+\epsilon,
\end{align*}
Taking $\epsilon \to 0$, we obtain the desired Lipschitz estimate.
\end{proof}

The next lemma is called the dynamic programing principle for $\mathbb{V}$.

\begin{lemma} \label{SDPP0}
For any $\xi\in L_{\mathcal{F}_t}(\Omega,M)$, the value function satisfies the following dynamic programming principle:
\begin{align}\label{lm52sdpp}
\mathbb{V}(t,\xi)=\inf_{u(\cdot)\in \mathcal{U}[t,s]}\mathbb{E}_t\bigg\{&\int_t^{s}f(r,X(r),u(r))dr+\mathbb{V}(s,X(s))\bigg\}.
\end{align}
\end{lemma}

\begin{proof}
For any control process $u(\cdot) \in \mathcal{U}[t,T]$ and initial condition $\xi \in L_{\mathcal{F}_t}(\Omega,M)$, the value function satisfies
\begin{align*}
\mathbb{V}(t,\xi)\le \mathcal{J}(t,\xi,u(\cdot))&=\mathbb{E}_{t}\bigg\{\int_t^{s}f(r,X(r),u(r))dr+\int_{s}^Tf(r,X(r),u(r))dr+h(X(T))\bigg\}\\
&=\mathbb{E}_{t}\bigg\{\int_t^{s}f(r,X(r),u(r))dr+\mathcal{J}(s,X(s);u(\cdot)\big|_{[s,T]})\bigg\}.
\end{align*}
Taking the infimum over all restricted controls $u(\cdot)|_{[s,T]} \in \mathcal{U}[s,T]$ and applying Lemma \ref{effJ}, we obtain
$$\mathbb{V}(t,\xi)\le \mathbb{E}_{t}\bigg\{\int_t^{s}f(r,X(r),u(r))dr+\mathbb{V}(s,X(s))\bigg\}.$$
Consequently,
$$\mathbb{V}(t,\xi)\le \inf_{u(\cdot)\in \mathcal{U}[t,s]}\mathbb{E}_{t}\bigg\{f(r,X(r),u(r))dr+\mathbb{V}(s,X(s)\bigg\}.$$

For the reverse inequality, given any $\epsilon > 0$, there exists a control $u^{\epsilon}(\cdot) \in \mathcal{U}[t,T]$ such that
\begin{align*}
\mathbb{V}(t,\xi)+\epsilon\ge \mathcal{J}(t,\xi;u^{\epsilon}(\cdot))
&=\mathbb{E}_{t}\bigg\{\int_t^{s}f(r,X^{\epsilon}(r),u^{\epsilon}(r))dr+\mathcal{J}(s,X^{\epsilon}(s);u^{\epsilon}(\cdot)\big|_{[s,T]})\bigg\}\\
&\ge\mathbb{E}_{t}\bigg\{\int_t^{s}f(r,X^{\epsilon}(r),u^{\epsilon}(r))dr+\mathbb{V}(s,X^{\epsilon}(s))\bigg\}\\
&\ge \inf_{u(\cdot)\in \mathcal{U}[t,s]}\mathbb{E}_{t}\bigg\{\int_t^{s}f(r,X(r),u(r))dr+\mathbb{V}(s,X(s))\bigg\}.
\end{align*}
By the arbitrariness of $\epsilon > 0$ and the definition of infimum, the dynamic programming principle is established.
\end{proof}

At last, we give a quite standard Lemma below for proving the  Theorem \ref{pDPP}. 

\begin{lemma} \label{LM54} 
Let $t\in [0,T)$, then for any $\xi\in L_{\mathcal{F}_t}(\Omega,M)$, we have
\begin{equation}\label{lm54}
V(t,\xi(\omega))=\mathbb{V}(t,\xi)(\omega).
\end{equation}
\end{lemma}

The proof of Lemma \ref{LM54} follows arguments similar to those in \cite[Proposition 4.8]{yong_stochastic_2022}; we include the details here for completeness.

\begin{proof}
Let $\xi = \sum_{i=1}^N x_i \mathbf{1}_{\Omega_i}$ be a simple random variable, 
where $\{\Omega_i\}_{1 \leq i \leq N}$ forms an $\mathcal{F}_t$-measurable partition of $\Omega$ 
and $x_i \in M$ for $i=1,\dots, N$. Following the approach in Lemma \ref{prop33}, we derive	
$$
\mathcal{J}(t,\xi;u(\cdot)) = \sum_{i=1}^N \mathcal{J}(t,x_i;u(\cdot)) \mathbf{1}_{\Omega_i} \geq \sum_{i=1}^N V(t,x_i) \mathbf{1}_{\Omega_i} = V(t,\xi),
$$
which establishes the lower bound
\begin{equation}\label{A}
\mathbb{V}(t,\xi) \geq V(t,\xi).
\end{equation}
For the upper bound, for each $i=1,\dots,N$, choose control strategies $\{u_{ik}(\cdot)\}_{k\geq 1} \subset \mathcal{U}[t,T]$ satisfying
$$
\lim_{k \to \infty} \mathcal{J}(t,x_i;u_{ik}(\cdot)) = \mathbb{V}(t,x_i), \quad i=1,\dots,N.
$$
Adapting the methodology of Lemma \ref{prop33}, we obtain
$$
\mathbb{V}(t,\xi) \leq \mathcal{J}\left(t,\sum_{i=1}^N x_i\mathbf{1}_{\Omega_i}; \sum_{i=1}^N u_{ik}(\cdot)\mathbf{1}_{\Omega_i}\right) = \sum_{i=1}^N \mathcal{J}(t,x_i;u_{ik}(\cdot)) \mathbf{1}_{\Omega_i}.
$$
Taking the limit as $k \to \infty$ yields the complementary inequality
\begin{equation}\label{B}
	\mathbb{V}(t,\xi) \leq \sum_{i=1}^N \mathbb{V}(t,x_i) \mathbf{1}_{\Omega_i} = V(t,\xi).
\end{equation}
The combination of \eqref{A} and \eqref{B} proves \eqref{lm54} for all simple random variables.

To extend this result to general $\xi \in L_{\mathcal{F}_t}(\Omega,M)$, consider an approximating sequence of simple random variables 
$\xi_N = \sum_{i=1}^N x_i^N \mathbf{1}_{\Omega_i^N}$ converging to $\xi$ $\mathbb{P}$-almost surely. Using the Lipschitz continuity of $\mathbb{V}$ in $\xi$ (established via arguments parallel to Proposition \ref{pLip}) 
together with Proposition \ref{pLip}, we conclude
$$
V(t,\xi)=\lim_{N \to \infty} V(t,\xi_N)=\lim_{N \to \infty} \mathbb{V}(t,\xi_N) = \mathbb{V}(t,\xi) \quad \mathbb{P}\text{-a.s.}
$$
This completes the proof.
\end{proof}

\section{Proofs of the main results} \label{s4}

We first give the proof of Theorem \ref{pDPP}.

\begin{proof}[\textbf{Proof of Theorem \ref{pDPP}}]
By applying Lemma \ref{LM54} and Lemma \ref{SDPP0}, and substituting the result from equation \eqref{lm54} into the dynamic programming principle \eqref{lm52sdpp}, we obtain the fundamental relation:
\begin{align*}
V(t,\xi)=\inf_{u(\cdot)\in \mathcal{U}[t,s]}\mathbb{E}_t\bigg\{&\int_t^{s}f(r,X(r),u(r))dr+V(s,X(s))\bigg\}.
\end{align*}
This equality establishes the desired result and completes the proof of the theorem.
\end{proof}

\begin{proof}[\textbf{Proof of Proposition \ref{pHold}}]
Let $u(\cdot) \in \mathcal{U}[t_0,t_1]$ and define $X(s) := X(s;t_0,x,u(\cdot))$. Applying Theorem \ref{pDPP} and Proposition \ref{prop41}, we derive the following estimates:
\begin{align}
V(t_0,x)-V(t_1,x)&=\mathbb{E}_{t_0}[V(t_1,X(t_1))-V(t_1,x)]+\mathbb{E}_{t_0}\int_{t_0}^{t_1} f(s,X(s),u(s))ds\nonumber\\
&\le C\mathbb{E}_{t_0}\rho(X(t_1),x)+\mathbb{E}_{t_0}\int_{t_0}^{t_1}C_mdr\nonumber\\
&\le K(t_1-t_0)^{\frac{1}{2}}.\nonumber
\end{align}

Conversely, for any $\epsilon > 0$, there exists a control $u^{\epsilon}(\cdot) \in \mathcal{U}[t_0,t_1]$. Defining $X^{\epsilon}(s) := X(s;t_0,x,u^{\epsilon}(\cdot))$, we obtain the lower bound:
\begin{align}
V(t_0,x)-V(t_1,x)+\epsilon&\ge \mathbb{E}_{t_0}[V(t_1,X^{\epsilon}(t_1))-V(t_1,x)]+\mathbb{E}_{t_0}\int_{t_0}^{t_1}f(s,X^{\epsilon}(s),u^{\epsilon}(s))ds\nonumber\\
&\ge -C\mathbb{E}_{t_0}\rho(X^{\epsilon}(s),x)-\mathbb{E}_{t_0}\int_{t_0}^{t_1}C_Mds\nonumber\\
&\ge -K(t_1-t_0)^{\frac{1}{2}}.
\end{align}
\end{proof}


\begin{proof}[\textbf{Proof of Theorem \ref{phjbl}}]
Let $(t,x)\in [0,T]\times M$ be given. For any $u\in U$, consider the constant control $u(\cdot)\equiv u$. Since $V\in C^{1,2}([0,T]\times M)$ by assumption, we take $V$ as the test function in \eqref{df61} to obtain:	
\begin{align}\label{eqHJBd}
0\le& \mathbb{E}_t\bigg\{\int_t^{t+{\epsilon}}f(s,X(s),u)ds+V(t+\epsilon,X(t+\epsilon))-V(t,x)\bigg\}\nonumber\\
=&  \mathbb{E}_t\bigg\{\int_t^{t+{\epsilon}}\Big\{f(s,X(s),u)+V_s(s,X(s))+ b(s,X(s),u)V(s,X(s))\nonumber\\
&\quad\qquad+\sum_{i=1}^m\sigma_i^2(s,X(s),u)V(s,X(s))\Big\}ds\bigg\}.
\end{align}
Dividing by $\epsilon$ and letting $\epsilon\rightarrow 0$ yields
$$0\le f(t,x,u)+V_t(t,x)+b(t,x,u)V(t,x)+\sum_{i=1}^m\sigma_i^2(t,x,u)V(t,x),$$
which implies
$$V_t+\inf_{u\in U}\mathbb{H}(t,x,u,DV,D^2V)\ge 0.$$

For the reverse inequality, given any $\delta,\epsilon>0$, there exists $u^{\delta,\epsilon}(\cdot)\in \mathcal{U}[t,T]$ and a constant $K>0$ such that
\begin{align*}
\delta\epsilon\ge& \mathbb{E}_t\bigg\{\int_t^{t+{\epsilon}}f(s,X^{\delta,\epsilon}(s),u^{\delta,\epsilon}(s))ds+
V(t+\epsilon,X^{\delta,\epsilon}(t+\epsilon))-V(t,x)\bigg\}\\
=&\mathbb{E}_t\bigg\{\int_t^{t+{\epsilon}}\Big\{f(s,X^{\delta,\epsilon}(s),u^{\delta,\epsilon}(s))ds+V_s(s,X^{\delta,\epsilon}(s))\\
&\qquad+b(s,X^{\delta,\epsilon}(s),u^{\delta,\epsilon}(s))V(s,X^{\delta,\epsilon}(s))+\sum_{i=1}^m\sigma_i^2(s,,X^{\delta,\epsilon}(s),u^{\delta,\epsilon}(s))V(s,X^{\delta,\epsilon}(s))\Big\}ds\bigg\}\\
\ge& \mathbb{E}_t\bigg\{\int_t^{t+\epsilon}\Big\{V_s(s,X^{\delta,\epsilon}(s))+\inf_{u\in U}\mathbb{H}(s,x,u,DV(s,X^{\delta,\epsilon}(s)),D^2V(s,X^{\delta,\epsilon}(s)))\Big\}ds\\
\ge& \epsilon [V_t(t,x)+\inf_{u\in U}\mathbb{H}(t,x,u,DV(t,x),D^2V(t,x)]-K\mathbb{E}_t\int_t^{t+\epsilon}\rho(X^{\delta,\epsilon}(s),x)ds\bigg\}.
\end{align*}
Dividing by $\epsilon>0$ and taking $\epsilon\rightarrow 0$, the continuity of $X^{\delta,\epsilon}(\cdot)$ gives
$$V_t(t,x)+\inf_{u\in U}\mathbb{H}(t,x,u,DV,D^2V)\le \delta .$$
Since $\delta>0$ was arbitrary, we conclude that \eqref{phjb} holds.
\end{proof}

\begin{proof}[\textbf{Proof of Corollary \ref{pVe}}]
Let $(t, x) \in [0, T) \times M$, and let $u(\cdot)$ be constructed by \eqref{pVe-eq1}. Consider the corresponding state process $X(\cdot) = X(\cdot; t, x, u(\cdot))$. Since $V \in C^{1,2}([0,T] \times M)$ is a classical solution by assumption, we may take $V$ as the test function in \eqref{df61}. This yields
\begin{align*}
&\mathbb{E}_th(X(T))-V(t,x)\\
&=\mathbb{E}_tV(T,X(T))-V(t,x)\\
&=\mathbb{E}_t\int_{t}^{T}\bigg\{V_s(s,X(s))+b(s,X(s),u(s))V(s,X(s))+\sum_{i=1}^m\sigma_i^2(s,X(s),u(s))V(s,X(s))\bigg\}ds\\
&=-\mathbb{E}_t\int_t^{T}f(s,X(s),u(s))ds.
\end{align*}
Consequently, we obtain that
$$V(t,x)=\mathcal{J}(t,x;u(\cdot)).$$
\end{proof}

\begin{proof}[\textbf{Proof of Theorem \ref{pExistence}}]

{\bf Step 1}. In this step, we prove that the value function is a viscosity solution to the equation \eqref{phjb}.

Let $\varphi \in C^{1,2}([0,T] \times M)$. Suppose $(t_0,x_0) \in [0,T) \times M$ is a local maximum point of $V - \varphi$, i.e.,
\begin{equation}\label{vvarphi}
V(t,x)-\varphi(t,x)\le V(t_0,x_0)-\varphi(t_0,x_0).
\end{equation}
in a neighborhood of $(t_0,x_0)$. By the dynamic programming principle, for any $u \in U$ and $X(\cdot) := X(\cdot;t_0,x_0,u)$, we have
\begin{align*}
V(t_0,x_0)=\inf_{u(\cdot)\in \mathcal{U}[0,T]}\mathbb{E}_{t_0}\bigg\{&\int_{t_0}^{t}f(r,X(r),u(r))dr+V(t,X(t))\bigg\}\\
\le \mathbb{E}_{t_0}\bigg\{&\int_{t_0}^{t}f(r,X(r),u)dr+V(t,X(t))\bigg\},\quad \forall u\in U.
\end{align*}
For $t-t_0 > 0$ sufficiently small, \eqref{vvarphi} implies
$$\mathbb{E}_{t_0}\big(V(t,X(t))-V(t_0,x_0)\big)\le \mathbb{E}_{t_0}\big(\varphi(t,X(t))-\varphi(t_0,x_0)\big),$$
which yields
\begin{align*}
0&\le \mathbb{E}_{t_0}\bigg\{\int_{t_0}^{t}f(r,X(r),u)dr+V(t,X(t))-V(t_0,x_0)\bigg\}\\
&\le \mathbb{E}_{t_0}\bigg\{\int_{t_0}^{t}f(r,X(r),u)dr+ \varphi(t,X(t))-\varphi(t_0,x_0)\bigg\}.
\end{align*}
Taking $\varphi$  as the test function in \eqref{df61} and dividing by $t-t_0$, then letting $t \to t_0^+$, we obtain
$$0\le \varphi_t(t_0,x_0)+\mathbb{H}(t_0,x_0,u,D\varphi(t_0,x_0),D^2\varphi(t_0,x_0)).$$
Taking the infimum over $u \in U$ gives
$$0\le \varphi_t(t_0,x_0)+\bfH(t_0,x_0,D\varphi(t_0,x_0),D^2\varphi(t_0,x_0)),$$ 
proving $V$ is a viscosity subsolution.

For the supersolution property, suppose $(t_0,x_0)$ is a local minimum point of $V - \varphi$. For any $\epsilon > 0$, there exists $u^\epsilon(\cdot) \in \mathcal{U}[0,T]$ with corresponding state process $X^\epsilon(\cdot) := X(\cdot;t_0,x_0,u^\epsilon(\cdot))$ such that
\begin{align*}
\epsilon(t-t_0)&\ge \mathbb{E}_{t_0}\bigg\{\int_{t_0}^{t}f(r,X^{\epsilon}(r),u^{\epsilon}(r))dr+V(t,X^{\epsilon}(t))-V(t_0,x_0)\bigg\}\\
&\ge \mathbb{E}_{t_0}\bigg\{\int_{t_0}^{t}f(r,X^{\epsilon}(r),u^{\epsilon}(r))dr+\varphi(t,X^{\epsilon}(t))-\varphi(t_0,x_0)\bigg\}\\
&\ge\mathbb{E}_{t_0}\bigg\{\int_{t_0}^t\Big\{\varphi_t(s,X^{\epsilon}(s))
+\mathbb{H}\big(s,X^{\epsilon}(s),u^{\epsilon}(s),D\varphi(s,X^{\epsilon}(s)),D^2\varphi(s,X^{\epsilon}(s))\big)\Big\}ds\bigg\}\\
&\ge\mathbb{E}_{t_0}\bigg\{\int_{t_0}^t \Big\{\varphi_t(s,X^{\epsilon}(s))+\bfH\big(s,X^{\epsilon}(s),D\varphi(s,X^{\epsilon}(s)),D^2\varphi(s,X^{\epsilon}(s))\big)\Big\}ds\bigg\}.
\end{align*}
Dividing by $t-t_0 > 0$ and letting $t \to t_0^+$, the continuity of $X^\epsilon(\cdot)$ yields
$$0\ge \varphi_t(t_0,x_0)+\bfH(t_0,x_0,D\varphi(t_0,x_0),D^2\varphi(t_0,x_0)),$$ 
establishing the viscosity supersolution property.
In conclusion, the value function $V(\cdot,\cdot)$ is a viscosity solution for $\eqref{phjb}$.

\ss

{\bf Step 2}.  
In this step, we establish uniqueness of viscosity solutions to \eqref{phjb} via the comparison principle, showing that $\bm{v}\le \bm{w}$ for any two such solutions.

Let $\bm{v}\in C([0,T]\times M)$ be a viscosity solution (and thus a viscosity subsolution) to  \eqref{phjb}, and let $\bm{w}\in C([0,T]\times M)$ be another viscosity solution (hence a supersolution) to \eqref{phjb}. 

Define the penalized subsolution 
$$\bar{\bm{v}}(t,x):=\bm{v}(t,x)-\frac{\epsilon}{t},$$
for $\epsilon>0$, which satisfies
$$\bar{\bm{v}}_t+\bfH\Big(t,x,D\bar{\bm{v}},D^2\bar{\bm{v}}\Big)=\bm{v}_t+\bfH\big(t,x,D\bm{v},D^2\bm{v}\big)+\frac{\epsilon}{t^2}>\frac{\epsilon}{t^2}\ge 0,$$
making $\bar{\bm{v}}$ a strict subsolution. Therefore, it suffices to prove
$$\bar{\bm{v}}(t, x)\le \bm{w}(t, x),\quad (t, x)\in (0, T]\times M.$$
Suppose by contradiction that there exists $(s,z) \in (0,T]\times M$ with
$$ \bar{\bm{v}}(s,z)-\bm{w}(s,z)\ge \delta .$$
Since $\bar{\bm{v}}-\bm{w}$ is bounded above, then there exists a sequence $(t_{\alpha},x_{\alpha},y_{\alpha})\in (0, T]\times M\times M$ that are maximizers of 
\begin{equation}\label{malpha}
\mathfrak{m}_{\alpha}:=\sup_{(t,x,y)\in (0,T]\times M\times M}\big\{\bar{\bm{v}}(t,x)-\bm{w}(t,y)-\varphi(x,y)\big\},\quad \varphi(x,y):=\frac{\alpha}{2}\rho^2(x,y),
\end{equation}
where 
$$\mathfrak{m}_{\alpha}\ge \bar{\bm{v}}(s,z)-\bm{w}(s,z)\ge \delta $$
by construction.

\ss 

Claim: There exists $c_0>0$ such that for all $\alpha>c_0$, we have $t_\alpha < T$.

Proof of Claim: If false, there exists $\alpha_k \to \infty$ with $t_{\alpha_k} = T$. The boundary condition satisfied for $\bar{\bm{v}}$ and $\bm{w}$ in difinition \ref{dfvss} yields
\begin{equation}\label{mk0}
0\le \delta \le \mathfrak{m}_{\alpha_k}\le  h(x_{\alpha_k})-h(y_{\alpha_k})-\frac{\alpha_k}{2} \rho^2(x_{\alpha_k},y_{\alpha_k})-\frac{\epsilon}{T}. 
\end{equation}
By compactness of $M$, there exists subsequences $\{x_{\alpha_{k_j}}\}_{j=1}^{\infty},\{y_{\alpha_{k_j}}\}_{j=1}^{\infty}$ and $x_0,y_0\in M$ such that  
$$\lim_{j\rightarrow \infty}x_{\alpha_{k_j}}=x_0,\quad\lim_{j\rightarrow \infty}y_{\alpha_{k_j}}=y_0.$$
Two cases arise:

Case 1: If $x_0 \neq y_0$, then   $\alpha_{k_j}\rho^2(x_{\alpha_{k_j}},y_{\alpha_{k_j}})\rightarrow \infty$ as $j\rightarrow \infty$. Since $h\in C(M)$, then the inequality \eqref{mk0} implies the contradiction
\begin{equation}\label{rhod}
\alpha_{k_j}\rho^2(x_{\alpha_{k_j}},y_{\alpha_{k_j}})+\delta\le 2\max_{x\in M}|h(x)|.
\end{equation}

\ss

Case 2: If $x_0=y_0$, then then \eqref{mk0} leads to
\begin{equation}\label{mk}
0\le \delta \le \mathfrak{m}_{\alpha_{k_j}}= h(x_{\alpha_{k_j}})-h(y_{\alpha_{k_j}})-\frac{\alpha_{k_j}}{2} \rho^2(x_{\alpha_{k_j}},y_{\alpha_{k_j}})-\frac{\epsilon}{T}\rightarrow -\frac{\epsilon}{T}<0, \quad\text{as}\quad j\rightarrow \infty.
\end{equation}
which is again contradictory.

This establishes the claim that $t_\alpha < T$ for large $\alpha$.

Following the argument in \eqref{rhod}, we observe that
$$\mathfrak{m}_{\alpha}\le \max_{t\in [0,T],x\in M}|\bm{v}|+\max_{t\in [0,T],x\in M}|\bm{w}|,$$ 
which implies $\alpha\rho^2(x_\alpha,y_\alpha) \to 0$ as $\alpha \to \infty$ in \eqref{malpha}. By compactness of $M$, there exist subsequences 
$\{x_{\alpha_j}\}_{j=1}^{\infty}$, $\{y_{\alpha_j}\}_{j=1}^{\infty}$ and some $x_0\in M$ such that
$$\lim_{j\rightarrow \infty}x_{\alpha_j}=\lim_{j\rightarrow \infty}y_{\alpha_j}=x_0.$$

Select $c_1 \geq c_0$ sufficiently large so that for all $\alpha_j > c_1$ (for notational simplicity, we continue to denote these as $\alpha$), the points $x_\alpha, y_\alpha \in B_{x_0}^\rho(r_m)$ (see Remark \ref{rk21}). This ensures the vectors $\gamma_{x_\alpha} := -\exp_{x_\alpha}^{-1}(y_\alpha)$ and $\gamma_{y_\alpha} := -\exp_{y_\alpha}^{-1}(x_\alpha)$ are well-defined. Applying Lemma \ref{rho} and equation \eqref{DDf}, we obtain the derivative relations: 
\begin{equation}\label{DxyV}
D_{x_{\alpha}}\varphi(v)=\lan \alpha\gamma_{x_{\alpha}},v\ran ,\quad D_{y_{\alpha}}\varphi(w)=\lan \alpha \gamma_{y_{\alpha}},w\ran ,
\end{equation}
for all $v \in T_{x_\alpha}M$ and $w \in T_{y_\alpha}M$.

Since $\bar{\bm{v}}$ is a subsolution and $\bm{w}$ is a supersolution, the maximum principle condition \eqref{condmaxp} from \nmycite{Fleming2006}{p.218} is satisfied. Lemma \ref{pmaxp} then yields the existence of:
\begin{equation*}
(p_1,D_{x_{\alpha}}\varphi,A_1)\in \bar{\mathcal{P}}^{2,+}\bar{\bm{v}}(t_{\alpha},x_{\alpha}),\quad (p_2,-D_{y_{\alpha}}\varphi,A_2)\in \bar{\mathcal{P}}^{2,-}\bm{w}(t_{\alpha},y_{\alpha})
\end{equation*}
satisfying
\begin{equation*}
p_1-p_2=0,\quad 
\begin{bmatrix}
A_1 &0\\
0 &-A_2
\end{bmatrix}
\le A+\epsilon A^2,
\end{equation*}
where $A=D_{(x_{\alpha},y_{\alpha})}^2\varphi(x_{\alpha},y_{\alpha})$. The viscosity solution inequalities give
\begin{equation}\label{eq43}
p_2+\bfH(t_{\alpha},y_{\alpha},-D_{y_{\alpha}}\varphi,A_2)\le 0\le \frac{\epsilon}{T^2}\le p_1+ \bfH(t_{\alpha},x_{\alpha},D_{x_{\alpha}}\varphi,A_1).
\end{equation}
Applying the inequality $\inf_{u\in U}f_1(u)-\inf_{u\in U}f_2(u)\le \sup_{u\in U}(f_1(u)-f_2(u))$ and using \eqref{DxyV}, we expand \eqref{eq43} to obtain
\begin{align}\label{eq471}
\frac{\epsilon}{T^2}\nonumber
&\le \sup_{u\in U}\bigg\{\frac{1}{2}\sum_{i=1}^m\Big (A_1(\sigma_i(t_{\alpha},x_{\alpha},u),\sigma_i(t_{\alpha},x_{\alpha},u))-A_2(\sigma_{i}(t_{\alpha},y_{\alpha},u),\sigma_{i}(t_{\alpha},y_{\alpha},u))\Big)\nonumber\\
&\qquad\qquad+\Big(\big\langle b_0(t_{\alpha},x_{\alpha},u),\alpha\gamma_{x_{\alpha}}\big\rangle_g-\big\langle b_0(t_{\alpha},y_{\alpha},u),-\alpha\gamma_{y_{\alpha}}\big\rangle_g\Big)\nonumber\\
&\qquad\qquad+\Big(f(t_{\alpha},x_{\alpha},u)-f(t_{\alpha},y_{\alpha},u)\Big)\bigg\}\nonumber\\
&\quad\qquad:=\sup_{u\in U}\{I_1+I_2+I_3\},
\end{align}
where   $b_0:=b+\frac{1}{2}\sum_{i=1}^{m}D_{\sigma_i}\sigma_i$.

Setting $\epsilon = \frac{1}{\alpha}$ and noting $\rho(x_\alpha,y_\alpha) \leq r_m$, Lemmas \ref{lm38}, \ref{SecondOrderLemma}, and \ref{LIPV} yield the estimate:
\begin{align}\label{eq481}
I_1&\le \sum_{i=1}^m(A+\epsilon A^2)\Big((\sigma_{i}(t_{\alpha},x_{\alpha},u),\sigma_i(t_{\alpha},y_{\alpha},u))^{\otimes 2}\Big)
\le \sum_{i=1}^m(5C_{L}^2+C)\alpha\rho^2(x_{\alpha},y_{\alpha}),
\end{align}
where the constant $C$ depends only on the bounded quantities $C_R$, $C_B$, and $n$.

The second term in \eqref{eq471} satisfies
\begin{align} \label{eq491}
I_2=\alpha\lan \exp_{x_{\alpha}}^{-1}(y_{\alpha}),L_{y_{\alpha}x_{\alpha}}b_0(t_{\alpha},y_{\alpha},u)-b_0(t_{\alpha},x_{\alpha},u)\ran _g \le \big(1+\tfrac{m}{2}\big)C_L\alpha\rho^2(x_{\alpha},y_{\alpha}),
\end{align}
where the inequality follows from the $C^{1}$-smoothness of $b_0$ and Lemma \ref{LIPV}.

For the third term, we have the straightforward bound
\begin{equation}\label{eq493}
I_3 \le C_L\rho(x_{\alpha},y_{\alpha}),
\end{equation}
which results from the Lipschitz continuity of $f$.

Substituting the bounds \eqref{eq481}, \eqref{eq491}, and \eqref{eq493} into \eqref{eq471} yields
\begin{equation*}
0<\frac{\epsilon}{T^2}\le \Big(5mC_L^2+mC+(1+\tfrac{m}{2})C_L\Big)\alpha\rho^2(x_{\alpha},y_{\alpha})+C_L\rho(x_{\alpha},y_{\alpha}).
\end{equation*}
Taking the limit as $\alpha \to \infty$ leads to a contradiction, since $\alpha \rho^2(x_\alpha,y_\alpha) \to 0$ while the left-hand side remains positive. This establishes that $\bar{\bm{v}}(x) \leq \bm{w}(x)$ for all $x \in M$. 

By symmetry between $\bm{v}$ and $\bm{w}$, we conclude the proof of the comparison principle.
\end{proof}


\begin{proof}[\textbf{Proof of Theorem \ref{pThmCont}}]
We borrow some idea from  \nmycite{barles_error_2007}{Theorem A.3},\nmycite{jakobsen_continuous_2005}{Theorem 3.1} to prove Theorem \ref{pThmCont}.

We establish the theorem by first proving one direction and then interchanging the roles of $V_1$ and $V_2$ for the converse. For clarity, we define the following quantities:
\begin{align*}
&\varphi(x,y):=\frac{1}{2}\alpha \rho^2(x,y),\nonumber\\
&\psi(t,x,y):=V_1(t,x)-V_2(t,y)-\varphi(x,y)-\frac{\epsilon}{t-s},\nonumber\\
&\mathfrak{m}_T:=\sup_{x,y\in M}\{\psi(T,x,y)\}^{+},\nonumber\\
&\mathfrak{m}:=\sup_{t\in (s,T],x,y\in M}\psi(t,x,y)-\mathfrak{m}_T,\nonumber\\
&\bar{\mathfrak{m}}:=\sup_{t\in (s,T],x,y\in M}\Big\{\psi(t,x,y)-\theta \mathfrak{m}\frac{T-t}{T-s}\Big\},\quad \theta,\epsilon\in (0,1).
\end{align*}

Without loss of generality, we may assume that $\mathfrak{m}> 0$ holds
 for sufficiently small $\epsilon$ and sufficiently large $\alpha$. If this were not the case, taking $\alpha \rightarrow \infty$ and $\epsilon\rightarrow 0$ would yield
\begin{equation}\label{eqVT}
\sup_{t\in [s,T],x\in M}\{V_1(t,x)-V_2(t,x)\}\le \sup_{x\in M}\{V_1(T,x)-V_2(T,x)\}^{+},
\end{equation}
which would immediately prove the theorem. 

Suppose the supremum defining $\bar{\mathfrak{m}}$ is attained at $(t_{\alpha}, x_{\alpha}, y_{\alpha})$. 
Due to the presence of the penalty term $\frac{\epsilon}{t-s}$, we conclude that $t_{\alpha} > s$. We claim that $t_{\alpha} < T$. By the elementary supremum inequality for $\bar{\mathfrak{m}}$ and  the definition of
$\mathfrak{m}$, it follows that
\begin{align}\label{mb}
\bar{\mathfrak{m}}&\ge \sup_{t\in (s,T],x,y\in M}\psi(t,x,y)-\sup_{t\in [s,T]}\theta \mathfrak{m}\frac{T-t}{T-s}
\nonumber\\
&\ge   \mathfrak{m} + \mathfrak{m}_T - \theta \mathfrak{m}\nonumber\\
&> \mathfrak{m}_T,
\end{align}
where the second inequality follows from $\mathfrak{m} > 0$. If $t_{\alpha} = T$, then $\bar{\mathfrak{m}}\le \mathfrak{m}_T$, contradicting \eqref{mb}, thus proving our claim. 

Following arguments similar to those for inequality \eqref{rhod}, there exists $x_0 \in M$ and subsequences $\{x_{\alpha_j}\}_{j=1}^{\infty}$, $\{y_{\alpha_j}\}_{j=1}^{\infty}$ such that
$$\lim_{j\rightarrow\infty}x_{\alpha_j}=\lim_{j\rightarrow\infty}y_{\alpha_j}=x_0.$$ 
Let $c_0$ be the minimal value ensuring $x_{\alpha_j}, y_{\alpha_j} \in B_{x_0}^{\rho}(r_m)$ (see Remark \ref{rk21}) for all $\alpha_j > c_0$. For notational simplicity, we retain $\alpha$ as the index in subsequent arguments.

Define $\gamma_{x_{\alpha}}:=-\exp_{x_{\alpha}}^{-1}(y_{\alpha})$ and $\gamma_{y_{\alpha}}:=-\exp_{y_{\alpha}}^{-1}(x_{\alpha})$. Applying Lemma \ref{pmaxp} and arguments parallel to those in Theorem \ref{pExistence}, there exist
\begin{equation*}
(p_1,D_{x_{\alpha}}\varphi,A_1)\in \mathcal{\bar{P}}^{2,+}V_1(t_{\alpha},x_\alpha),\quad (p_2,-D_{y_{\alpha}}\varphi,A_2)\in \mathcal{\bar{P}}^{2,-}V_2(t_{\alpha},y_\alpha)
\end{equation*}
satisfying
\begin{equation*}
\begin{bmatrix}
A_1 &0\\
0 &-A_2
\end{bmatrix}
\le A+\epsilon A^2,\quad p_1-p_2=-\frac{\theta \mathfrak{m}}{T-s}-\frac{\epsilon}{(t_{\alpha}-s)^2}
\end{equation*}
where $A = D_{(x_{\alpha}, y_{\alpha})}^2 \rho^2(x_{\alpha}, y_{\alpha})$. The definitions of $V_1$ and $V_2$ imply
\begin{equation*}
p_2+\bfH_2(t_{\alpha},y_{\alpha},-D_{y_{\alpha}}\varphi,A_2)\le 0\le p_1+\bfH_1(t_{\alpha},x_{\alpha},D_{x_{\alpha}}\varphi,A_1),
\end{equation*}
which expands to
\begin{align*}
\frac{\theta \mathfrak{m}}{T-s}+\frac{\epsilon}{(t_{\alpha}-s)^2}\le \sup_{u\in U}\bigg\{&\frac{1}{2}\sum_{i=1}^m\Big(A_1\big(\sigma_{1i}(t_{\alpha},x_{\alpha},u),\sigma_{2i}(t_{\alpha},x_\alpha,u)\big)-A_2\big(\sigma_{2i}(t_{\alpha},y_{\alpha},u),\overline{\sigma}_{2i}(t_{\alpha},y_{\alpha},u)\big)\Big)\nonumber\\
&+\big(\langle b_{01}(t_{\alpha},x_{\alpha},u),\alpha\gamma_{x_{\alpha}}\rangle_g-\langle b_{02}(t_{\alpha},y_{\alpha},u),-\alpha\gamma_{y_{\alpha}})\rangle_g\big)\nonumber\\
&+\big(f_1(t_{\alpha},x_{\alpha},u)-f_2(t_{\alpha},y_{\alpha},u)\big)\bigg\},
\end{align*}
with $b_{0j}:=b_j+\frac{1}{2}\sum_{i=1}^{m}D_{\sigma_{ji}}\sigma_{ji}$ for $j = 1, 2$. Standard calculations (cf. Theorem \ref{pExistence}) yield, as $\theta \rightarrow 1$,
\begin{align}\label{eq525}
\frac{\mathfrak{m}}{T-s}\le &5\alpha\sup_{x\in M,t\in [s,T],u\in U}\bigg\{\sum_{i=1}^m\|\sigma_{1i}-\sigma_{2i}\|_g^2+\|b_{01}-b_{02}\|_g^2\bigg\}+\sup_{x\in M,t\in [s,T],u\in U}|f_1-f_2|\nonumber\\
&+ \alpha C\rho^2(x_{\alpha},y_{\alpha})+ C\rho(x,y),
\end{align}
where $C > 0$ depends on $C_{L},C_B$ and $C_{R}$.

Moreover, we have the bound
\begin{equation*}
2\alpha \rho^2(x_{\alpha},y_{\alpha})\le \psi(t_{\alpha},x_{\alpha},x_{\alpha})+\psi(t_{\alpha},y_{\alpha},y_{\alpha}).
\end{equation*}
Proposition \ref{pLip} guarantees that $V_1$ and $V_2$ are Lipschitz in space. Denoting their Lipschitz constants by $|V_1|_1$ and $|V_2|_1$, respectively, we obtain
\begin{equation*}
2\alpha \rho^2(x_{\alpha},y_{\alpha})\le (|V_1|_1+|V_2|_{1})\rho(x_{\alpha},y_{\alpha}).
\end{equation*}
which implies
\begin{equation}\label{eq437}
\rho(x_{\alpha},y_{\alpha})\le \frac{ (|V_1|_1+|V_2|_1)}{2}\alpha^{-1},\qquad \alpha \rho(x_{\alpha},y_{\alpha})^2\le \frac{(|V_1|_1+|V_2|_1)^2}{4}\alpha^{-1}.
\end{equation}

To estimate $\sup_{x,y\in M} \psi(T,x,y)$, suppose the supremum is attained at $(\hat{x}_{\alpha}, \hat{y}_{\alpha})$. Then,
\begin{align}\label{eq526}
\psi(T,\hat{x}_{\alpha},\hat{y}_{\alpha})&\le |h_1(\hat{x}_{\alpha})-h_2(\hat{x}_{\alpha})|+ |h_2(\hat{x}_{\alpha})-h_2(\hat{y}_{\alpha})|+\alpha \rho(\hat{x}_{\alpha},\hat{y}_{\alpha})^2\nonumber\\
&\le \max_{x\in M}|h_1-h_2|+C\rho(\hat{x}_{\alpha},\hat{y}_{\alpha})+\alpha\rho(\hat{x}_{\alpha},\hat{y}_{\alpha})^2\nonumber\\
&\le \max_{x\in M}|h_1-h_2|+O_{C}(\alpha^{-1}),
\end{align}
where $C$ depends on $C_L$.

Let 
\begin{align}\label{dfS}
&S_1:= 5\sup_{x\in M,t\in [s,T] .u\in U}\bigg\{\sum_{i=1}^m\|\sigma_{1i}-\sigma_{2i}\|_g^2+\|b_{01}-b_{02}\|_g^2\bigg\},\nonumber\\
&S_2:=(T-s)\sup_{t\in [s,T],x\in M,u\in U}|f_1-f_2|+\max_{x\in M}|h_1-h_2|.
\end{align}
Combining \eqref{eq525}-\eqref{dfS} with the definition of $\mathfrak{m}$, we derive
\begin{align}\label{eqphixy}
\sup_{t\in (s,T],x,y\in M}\psi(t,x,y)\le & (T-s)S_1 \alpha + O_{C}(\alpha^{-1})+S_2,
\end{align}
with $C$ depending on $|V_1|_1$, $|V_2|_1$, $C_L$, $T$ and $C_R$.  

Since
\begin{equation*}
\sup_{t\in (s,T]}\sup_{x\in M}\Big\{V_1(t,x)-V_2(t,x)-\frac{\epsilon}{t-s}\Big\}\le \sup_{t\in[s,T],x,y\in M}\psi(t,x,y),
\end{equation*}
combining with \eqref{eqphixy}, thus
\begin{equation*}
\sup_{t\in (s,T]}\sup_{x\in M}\Big\{V_1(t,x)-V_2(t,x)-\frac{\epsilon}{t-s}\Big\}\le (T-s)S_1 \alpha + O_{C}(\alpha^{-1})+S_2.
\end{equation*}
Taking $C'$  sufficiently large so that $\alpha = \frac{C'}{\sqrt{(T-s)S_1}} \ge c_0$ yields
\begin{equation}\label{eqveps}
\sup_{t\in (s,T]}\sup_{x\in M}\Big\{V_1(t,x)-V_2(t,x)-\frac{\epsilon}{t-s}\Big\}\le (1+C'+\tfrac{C}{C'})\sqrt{(T-s)}\sqrt{S_1}+S_2.
\end{equation}
Let $\epsilon\rightarrow 0 $ in left hand of \eqref{eqveps}, the result then follows from the definition of $b_0$ and $\overline{b}_0$, completing half of the proof.
\end{proof}


\end{document}